\documentclass[reqno]{amsart}
\allowdisplaybreaks

\usepackage{amsfonts,amssymb,amsthm,mathrsfs,mathtools,fancyhdr,verbatim,newlfont,latexsym,amscd,dsfont,color,cancel}
\usepackage{url}
\usepackage[pdfencoding=auto, psdextra]{hyperref}
\usepackage{xcolor}
\usepackage[numbers]{natbib}

\calclayout

\numberwithin{equation}{section}

\newtheorem{proposition}{Proposition}[section]
\newtheorem{corollary}[proposition]{Corollary}
\newtheorem{lemma}[proposition]{Lemma}
\newtheorem{remark}[proposition]{Remark}
\newtheorem{definition}[proposition]{Definition}
\newtheorem{theorem}[proposition]{Theorem}

\usepackage{graphicx} 

\newcommand{\e}[1]{\mathrm{e}\!\left(#1\right)}
\newcommand{\ti}{\to\infty}

\newcommand{\N}{\mathbb{N}}
\newcommand{\Z}{\mathbb{Z}}
\newcommand{\Q}{\mathbb{Q}}
\newcommand{\R}{\mathbb{R}}
\newcommand{\C}{\mathbb{C}}

\newcommand{\h}{\mathfrak{H}}
\newcommand{\sltr}{\mathrm{SL}(2,\R)}
\newcommand{\sltz}{\mathrm{SL}(2,\Z)}

\newcommand{\tsltr}{\widetilde{\mathrm{SL}}(2,\R)}

\newcommand{\ha}{\frac{1}{2}}
\newcommand{\tha}{\tfrac{1}{2}}

\newcommand{\Hei}{\mathbb{H}(\mathbb{R})}

\newcommand{\sve}[2]{\left(\begin{smallmatrix}\!#1\!\\ \!#2\!\end{smallmatrix}\right)}

\newcommand{\bm}[1]{\mbox{\boldmath{$#1$}}}
\newcommand{\LtR}{\mathrm L^2(\R)}

\newcommand{\SEtaR}{\mathcal{S}_\eta(\R)}

\newcommand{\sma}[4]{\left(\begin{smallmatrix} #1&#2\\#3&#4\end{smallmatrix}\right)}

\def\GamG{\Gamma\backslash G}

\def\vecxi{{\text{\boldmath$\xi$}}}
\newcommand{\de}{\mathrm{d}}

\begin{document}


\title[Stochastic Calculus for the Theta Process]{Stochastic Calculus for the Theta Process}

\author[F. Cellarosi]{Francesco Cellarosi}
\address[Francesco Cellarosi]{Department of Mathematics and Statistics. Queen's University, Kingston, Ontario. Canada.}
\email{francesco.cellarosi@queensu.ca}
\author[Z. Selk]{Zachary Selk}
\address[Zachary Selk]{Department of Mathematics and Statistics. Queen's University, Kingston, Ontario. Canada. Corresponding author.}
\email{zachary.selk@queensu.ca}
\date{\today} 
\subjclass[2020]{60L20, 60L90, 11F27, 11L03, 37C85}
\keywords{Rough paths; theta functions; exponential sums}

\begin{abstract}
The theta process is a stochastic process of number theoretical origin arising as a scaling limit of quadratic Weyl sums. It can be described in terms of the geodesic flow and an automorphic function on a homogeneous space. This process has several properties in common with Brownian motion such as its H\"older regularity, uncorrelated increments and quadratic variation. However, crucially, we show that the theta process is not a semimartingale, making It\^o calculus techniques inapplicable. Instead, we use the celebrated rough paths theory to develop the stochastic calculus for the theta process. We do so  by constructing the iterated integrals - the ``rough path" - above the theta process. Rough paths theory takes a signal and its iterated integrals and produces a vast and robust theory of stochastic differential equations. In addition, the rough path we construct can be described in terms of higher rank theta sums, via equidistribution of horocycle lifts. \end{abstract}
\maketitle

\tableofcontents

\section{Introduction}
The theta process, $X$, is a stochastic process of number theoretical origin introduced in \cite{Cellarosi-Marklof}. In this paper, we study the stochastic calculus of $X$. In particular, we demonstrate the need for what is called rough paths theory and we develop all tools required to apply rough paths theory. To the authors' best knowledge, this is the first application of rough paths theory to number theory.

\subsection{Theta Sums and the Theta Process}
We consider quadratic Weyl sums 
\begin{equation*}
    S_N(x;\alpha,\beta):=\sum_{n=1}^N \e{\left(\tfrac{1}{2}n^2+\beta n\right)x+n\alpha},
\end{equation*}
where $\e{z}:=e^{2\pi i z}$ and the parameters $x,\alpha,\beta$ are real. Various instances of $S_N(x;\alpha,\beta)$ are ubiquitous in number theory. For instance, classical quadratic Gauss sums can be written as $g(a;q)=S_{q}(\frac{2a}{q};0,0)$ and can be evaluated explicitly, depending on the remainder of $q$ modulo 4,  in terms of Kronecker symbols. Bounds for $S_N$ (and more generally for Weyl sums where the quadratic polynomial in $n$ is replaced by a degree $d$ polynomial) have a long history, starting with 
Hardy and Littlewood \cite{HL-23, HL-1914} and Weyl \cite{Weyl-1916}. A review of the history of the subject can be found, e.g., in \cite{flaminio2023equidistribution}. For instance, Fiedler, Jurkat and K\"{o}rner \cite{FJK1977} proved that for a non-decreasing sequence $(\psi(n))_{n\geq1}$ we have that  $\limsup_{N\to\infty}|S_N(x;0,0)|/(\sqrt{N}\psi(N))<\infty$ for almost every $x$ if and only if $\sum_{n=1}^\infty n^{-1}\psi(n)^{-4}<\infty$. Similarly, Fedotov and Klopp \cite{Fedotov-Kopp2012} showed that for a non-increasing  function $g:\R_+\to\R_+$ we have $\limsup_{N\to\infty}|S_N(x;\alpha,0)| g(\log N)/\sqrt{N}<\infty$ for almost every $(x,\alpha)$ if and only if $\sum_{n=1}^\infty g(n)^6<\infty$. They also proved that for almost every $x$ one can find a dense $G_\delta$ set $A_x$ so that for all $\alpha\in A_x$ we have $\limsup_{N\to\infty}|S_N(x;\alpha,0)| g(\log N)/\sqrt{N}<\infty$ if and only if $\sum_{n=1}^\infty g(n)^4<\infty$. The latter result was  generalized by Flaminio and Forni \cite{Flaminio-Forni-2006} using dynamical tools such as quantitative equidistribution for nilflows. Sums of the form $S_N$ also arise in various problem of mathematical physics, from the semiclassical dynamics of precessing spins \cite{Berry-Goldberg, Marklof-precessingspins}, to the Berry-Tabor conjecture for the geodesic flow on the 2-torus with quasi-periodic boundary conditions \cite{Marklof2003Annals}.

$S_N(x;\alpha,\beta)$ with random $x$ was originally studied in \cite{Jurkat-2,Jurkat-3,Jurkat-4}, where Jurkat and van Horne showed a central limit theorem where $x$ was distributed randomly uniformly on an interval. Marklof in \cite{Marklof-Duke-1999} introduced automorphic form techniques for studying $S_N(x;\alpha,\beta)$ with random $x$. In this study, following \cite{Cellarosi-Marklof}, we fix $(\alpha, \beta)\in \mathbb R^2\smallsetminus \mathbb Q^2$ and we let $x$ be randomly distributed according to a probability measure $\lambda$ on $\R$ which is absolutely continuous with respect to Lebesgue measure. 

The sum $S_N(x;\alpha,\beta)$ can then be viewed as the position after $N$ steps of a a deterministic walk in the complex plane with a random seed $x$. The classical Donsker's theorem says that a random walk, piecewise linearly interpolated, scaled by $\sqrt N$, converges weakly to a Brownian motion on the space of continuous functions. Analogously, the theta process arises as a limit of quadratic Weyl sums, scaled by $\sqrt N$. More precisely, it is shown in \cite{Cellarosi-Marklof}  that the sequence of stochastic processes
\begin{equation}\label{def:informal-X_N(t)}
    t\mapsto X_N(t)=\frac{1}{\sqrt N} \sum_{n=1}^{\lfloor N t\rfloor} \e{\left(\frac{1}{2}n^2+\beta n\right)x+n\alpha}+\text{piecewise linear interpolation}
\end{equation}
converges weakly on the space of continuous functions to a stochastic process $X$, called the \textbf{theta process}. Furthermore, the process $X$ may be given explicitly in terms of an automorphic function $\Theta_\chi$ on a particular Lie group $G$. The function $\Theta_\chi$ is invariant under a lattice $\Gamma$ and thus well defined on the finite volume quotient $\GamG$. The process $X$ may be given by 
\begin{equation*}
    X(t)=\sqrt{t} \Theta_\chi (\Gamma g\Phi^{2 \log t}),
\end{equation*}
where $\Gamma g$ is drawn according to the Haar measure $\mu_{\GamG}$ on $\GamG$ and $(\Phi^s)_s$ is the geodesic flow. For another instances of  stochastic processes coming from exponential sums, see   \cite{Demirci-Akarsu-incomplete-Gauss-sums, Kloosterman}. 

As $X$ arises in a similar way as Brownian motion, it shares many of the same properties. For example, both are almost surely $\alpha$-H\"older for any $\alpha \in (0,\frac{1}{2})$. They both have quadratic variation $t$ and the same mean and  covariance. They also both are almost surely of infinite $p$ variation if $p<2$. The main difference between Brownian motion and $X$ is that Brownian motion has infinitely many moments where $X$ only has $6-\varepsilon$ moments for any $\varepsilon>0.$

\subsection{Stochastic Calculus and Rough Paths Theory}
Stochastic calculus is concerned with differential equations of the form
\begin{equation}\label{eq:intro-rde}
    dY=b(Y) dt+\sigma(Y) dZ,
\end{equation}
where $Z$ is a stochastic process. Equations such as \eqref{eq:intro-rde} are known as \textbf{stochastic differential equations (SDEs)} and the literature on SDEs is vast. SDEs have found use in all areas of applied mathematics such as engineering, physics, biology, and finance to name a few. However, one of the main achievements of SDE theory is the further analysis of the stochastic process $Z$. For instance, the celebrated It\^o calculus developed in the 1950s (see \cite{Ito-2,Ito-3,Ito-1}) was created to study SDEs where $Z$ is Brownian motion. Not only does It\^o allow one to study noisy differential equations (which has vast applications in applied mathematics), it led to a much deeper understanding of Brownian motion. In particular, It\^o calculus proves a myriad of results (see \cite{Protter-Stochastic-Calculus,Schilling-book}) about Brownian motion such as its conformal invariance, the L\'evy characterization of Brownian motion, Girsanov's change of measure, It\^o's representation theorem and several other important, useful and interesting facts. Furthermore, It\^o calculus connects Brownian motion to the Laplacian and has led to further study of partial differential equations (via Feynman-Kac formulae) and even in geometry via its help in studying the Laplace-Beltrami operator. It\^o calculus has been generalized to general martingales and to a larger class called semimartingales. 

Although it is hard to overstate the importance of the It\^o calculus, it has several limitations. Most importantly, It\^o calculus is limited to semimartingales. It therefore leaves out several important examples such as the fractional Brownian motion (which is a generalization of Brownian motion that allows for correlations in the increments - see \cite{fbm-book}). Another crucial limitation of It\^o calculus is that all solutions of \eqref{eq:intro-rde} are defined in an $L^2$ sense and not for each individual realization of $Z$. Furthermore, SDEs defined in It\^o sense are not robust to changes in the noise. More precisely, Brownian paths that are close in the $\|\cdot\|_\infty$ norm might produce solutions that behave very differently. 

These limitations have been overcome by rough paths theory which was introduced by Lyons in 1998 (see \cite{Lyons-Original}), using the pioneering work of Chen (see \cite{Chen-1}) on the algebra of iterated integrals. Rough paths theory has since revolutionized stochastic analysis. Not only does it allow for general stochastic processes (not just semimartingales), it is a pathwise theory of differential equations that is robust to changes in the noise and gives an extensive collection of further tools to study SDEs. 

Rough paths theory allows for general signal $Z\in C^\alpha([0,T],\mathbb R^d)$ which is $\alpha$-H\"older. The key observation of rough paths theory is that defining solutions to equation \eqref{eq:intro-rde} can be reduced to defining the iterated integrals
\begin{equation*}
   \int_s^t \int_s^{r_2} dZ(r_1) \otimes dZ(r_2)\in \mathbb R^{d\times d}, 
\end{equation*}
where the $(i,j)$th component is the integral of the $i$th component of $Z$ against the $j$th component. The pair $(Z(t)-Z(s),\int_s^t \int_s^{r_2} dZ(r_1) \otimes dZ(r_2))\in \mathbb R^d \oplus \mathbb R^{d\times d}$ is called a \textbf{rough path} above $Z$. 

If $Z=B$ is Brownian motion then one can ``enhance" $B$ with the It\^o iterated integrals. In this way, one can ``embed" It\^o theory into rough paths theory. There are some slight subtleties in the difference between It\^o calculus for Brownian motion and rough paths theory for Brownian motion enhanced with It\^o integration (see \cite{Friz-Hairer-Book}, Chapter 3) but overall, rough paths theory can be seen as a generalization of the It\^o calculus.

Rough paths theory has seen immense popularity in recent years in both pure and applied mathematics. For example, incorporating iterated integrals in machine learning has led to the ``signature method." This in turn lead to advances in Chinese character recognition \cite{Chinese-Recognition}, diagnosing psychiatric illnesses \cite{Bipolar-Borderline} and machine learning in finance. Additionally, Hairer used rough paths theory to solve stochastic PDEs such as the KPZ equation \cite{MH-KPZ}. Hairer further generalized rough paths theory to his theory of regularity structures in \cite{MH-RS}. Rough paths theory has been used to solve differential equations driven by fractional Brownian motion \cite{fBm-sdes} and even irregular deterministic functions such as Weierstrass functions \cite{Weierstrass-Self}. 

As rough paths theory is a theory of stochastic calculus that allows for general H\"older processes and the theta process $X$ is almost surely H\"older (see \cite{Cellarosi-Marklof}), rough paths theory is an ideal setting for the stochastic calculus of $X$. In this paper we demonstrate the need for rough paths theory by showing that It\^o calculus and other types of calculus do not work for $X$. We then construct the iterated integrals (the rough path) above $X$ in a natural way. This allows for the vast collection of powerful tools in rough paths theory and is to the authors' best knowledge the first application of rough paths theory to number theory.

\subsection{Main Result}

In \cite{Random-Walk-rough-path} the authors show that a classical random walk along with its canonical iterated integrals converges weakly on the space of geometric rough paths, proving a ``rough" version of Donsker's theorem. The main contribution of the current paper is showing the analogous weak convergence of $X_N$ with the canonical Riemann-Stieltjes integrals on the space of geometric rough paths. We use the notations $X_N^1=\Re X_N$, $X_N^2=\Im X_N$ and thus identify $\C$ with $\R^2$. We show the following theorem. 
\begin{theorem}[Main Theorem]\label{theorem:main} Consider $X_N$ as in \eqref{def:informal-X_N(t)}, with fixed $(\alpha,\beta)\notin\Q^2$ and $x$ randomly distributed on $\R$ according to a probability measure which is absolutely continuous with respect to the Lebesgue measure. Consider the two-parameter process $\mathbf X_N(s,t)=(X_N(t)-X_N(s),\mathbb X_N(s,t))$  taking values in $\mathbb R^2\oplus \mathbb R^{2\times 2}$, where  $\mathbb X_N^{i,j}(s,t):=\int_s^t (X_N^i(r)-X_N^i(s)) dX_N^j(r)$ is the standard Riemann-Stieltjes integral for $i,j\in\{1,2\}$. Then the laws of $\mathbf X_N$ converge on the rough path space $\mathscr C_g^{\gamma}$ for all $\gamma<\frac{1}{2}$. The limiting process is a geometric rough path above $X$. 
\end{theorem}
The matrix-valued function $\mathbb X_N$ can be written in terms of higher ordered theta sums of the kind studied in \cite{Cosentino-Flaminio-2015, Gotze-Gordin, Marklof-Welsh-I, Marklof-Welsh-II}. The main observation of our current paper is that iterated integrals which are the double integral over the continuous simplex $\Delta_{s,t}:=\{(r_1,r_2)\in (s,t)^2: s\leq r_1\leq r_2\leq t\}$ correspond to a rank-2 theta sum over the discrete simplex. The second order process can thus be written in an analogous way to $X$, in terms of a higher ordered automorphic function on a larger Lie groups with a triangular cutoff. Although higher rank theta sums have been studied in e.g. \cite{Marklof-Welsh-I,Marklof-Welsh-II}, they have been restricted to smooth cutoff functions or to the sharp indicator of boxes. The primary technical barrier to proving Theorem \ref{theorem:main} is the analysis of theta sums with non-smooth yet non-rectangular cutoffs. To the authors' best knowledge this is the first analysis of theta sums over non-smooth yet non-rectangular cutoffs. The theory developed here thus suggests a method for handling theta sums with sharp cutoffs of other polygonal regions. 

\section*{Acknowledgments} We would like to thank Tariq Osman, Jens Marklof and Trevor Wooley for helpful discussions. The first author acknowledges the support from the NSERC Discovery Grant RGPIN2022-04330.

\section{The Theta Process}
Let $\e{z}=e^{2\pi i z}$. Let $x,\alpha,\beta\in\R$ and $N\in\N$. We consider theta sums (quadratic Weyl sums)
\begin{align}\label{def-S_N(x,alpha,beta)}
    S_N(x;\alpha,\beta)=\sum_{n=1}^{N}\e{\left(\tha n^2+\beta n\right)x+\alpha n}.
\end{align}
Since we will fix $(\alpha,\beta)\notin\Q^2$, we simplify our notation by writing $S_N(x)$ for $S_N(x;\alpha,\beta)$.

Fix $T>0$. For each $x$ we consider the piecewise affine curves $[0,T]\ni t\mapsto X_N(x;t)\in\C$ obtained by joining the partial sums of \eqref{def-S_N(x,alpha,beta)}, normalized by $\sqrt{N}$, i.e.
\begin{align}\label{def-X_N(t)}
    X_N(x;t)=\frac{1}{\sqrt{N}}\left(S_{\lfloor tN\rfloor}(x)+
    (tN-\lfloor tN\rfloor)\left(S_{\lfloor tN\rfloor+1}(x)-S_{\lfloor tN\rfloor}(x)\right)
    \right).
\end{align}

An example of the curve $t\mapsto X_N(x;t)$ is given in Figure \ref{fig:BM_v_Theta}.
\begin{center}
\begin{figure}[b!]
    \includegraphics[width=15cm]{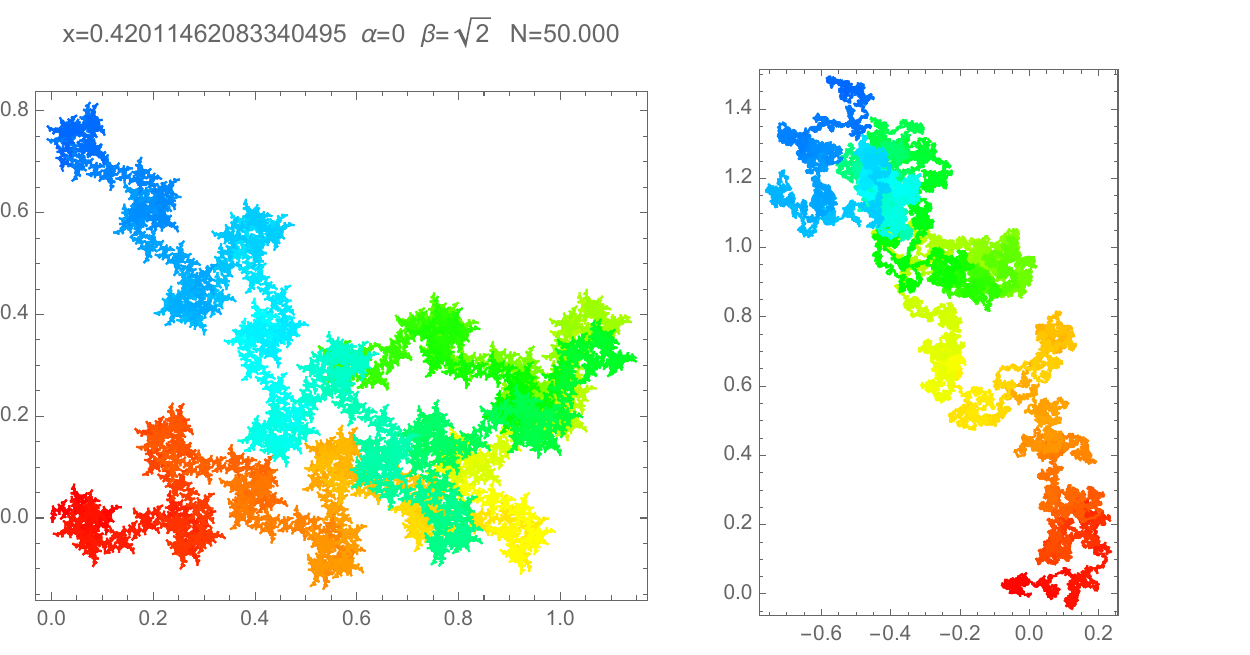}
    \caption{Left panel: the curve $t\mapsto X_N(x;t)$, with $N=50,\!000$, $\alpha=0$, $\beta=\sqrt{2}$, and $x\approx0.42011$ as indicated on the figure. Right panel: a realization of a rescaled simple symmetric random walk, obtained as in \eqref{def-S_N(x,alpha,beta)}-\eqref{def-X_N(t)} replacing $\e{\left(\tha n^2+\beta n\right)x+\alpha n}$ with $\e{\xi_n}$ for a sequence of i.i.d uniform random variables $(\xi_n)_{n\geq1}$ on $[0,1]$. In both panels, $t$ runs from $0$ (red) to $1$ (blue).}
    \label{fig:BM_v_Theta}
\end{figure}
 \end{center}
 
Let $\lambda$ be a  Borel probability measure on $\R$, absolutely continuous with respect to the Lebesgue measure. We let $x$ be randomly distributed according to $\lambda$ and we think  of \eqref{def-X_N(t)} as a deterministic walk on $\C$ with a random seed $x$, with the $n$th increment given by $\mathrm{e}((\ha n^2+\beta n)x+\alpha n)$. More precisely, we consider the space $\mathcal{C}_0([0,T],\C)$ of complex-valued continuous functions on $[0,1]$ taking value $0$ at $0$. We equip $\mathcal{C}_0([0,T],\C)$ with the topology induced by the $\sup$ norm and hence equip it with the corresponding Borel $\sigma$-algebra.

For each $N$, the probability measure $\lambda$ induces a probability measure, namely the push-forward measure $\mathbb{P}_N=((X_N(\cdot;t))_{0\leq t\leq T})_*\lambda$) on $\mathcal{C}_0([0,T],\C)$. As shown by Theorem 1.3 in \cite{Cellarosi-Marklof}, such measures converge weakly as $N\ti$. We refer to the stochastic process whose law is given by the limiting measure as the \emph{theta process}. 

\begin{theorem}[Invariance principle for theta sums, \cite{Cellarosi-Marklof}]\label{thm-existence-of-theta-process} Let $\lambda$ be a Borel probability measure, absolutely continuous with respect to the Lebesgue measure. Let $(\alpha,\beta)\notin\Q^2$. There exists a stochastic process $X\in \mathcal{C}_0([0,T],\C)$ (the \emph{theta process}) such that $X_N\Rightarrow X$ as $N\ti$. In other words, the induced probability measures $\mathbb{P}_N$ converge weakly as $N\ti$ to a probability measure $\mathbb{P}$ (the law of $X$) on  $\mathcal{C}_0([0,T],\C)$. The theta process $X$ does not depend on $\lambda$, or on $(\alpha,\beta)\in\R^2\smallsetminus\Q^2$.
\end{theorem}

\begin{remark}
    To prove Theorem \ref{thm-existence-of-theta-process} the two main steps are:
    \begin{itemize}
        \item showing the existence of limiting finite-dimensional distributions;
        \item showing tightness of $\{\mathbb{P}_N\}_{N\geq1}$ in $\mathcal{C}_0([0,T],\C)$.
    \end{itemize} 
\end{remark}

The theta process is interesting for its number theoretical origin, as well as for its properties. Some properties (stationarity, scaling, time inversion, law of large numbers, rotational invariance, H\"{o}lder continuity, nondifferentiability) match those of a standard Brownian motion, while other properties are strikingly different (heavy tails, dependent increments, smaller modulus of continuity), as shown in Theorem 1.4 in \cite{Cellarosi-Marklof}. Below we use the notation ``$\sim$" to denote equality in distribution.
\begin{theorem}[Properties of the theta process, \cite{Cellarosi-Marklof}]
\textcolor{white}{.}
\begin{itemize}
    \item[(a)] Stationarity. Let $t_0>0$, and $Y(t)=X(t_0+t)-X(t_0)$. Then $Y\sim X$.
    \item[(b)] Scaling. Let $a>0$ and $Y(t)=\frac{1}{a}X(a^2 t)$. Then $Y\sim X$.
    \item[(c)] Time inversion. Let $$Y(t)=\begin{cases}
        0&\mbox{if $t=0$;}\\
        t X(1/t)&\mbox{if $t>0$}.
    \end{cases}$$
     Then $Y\sim X$.
    \item[(d)] Law of Large Numbers. Almost surely, $\lim_{t\to\infty}\frac{X(t)}{t}=0$.
    \item[(e)] Rotational invariance. Let $\theta\in\R$ and $Y(t)=e^{2\pi i \theta}X(t)$. Then $Y\sim X$.
     \item[(f)] H\"{o}lder continuity. Let $\gamma<\ha$. Then, almost surely, the curve $t\mapsto X(t)$ is everywhere locally $\gamma$-H\"{o}lder continuous. 
     \item[(g)] Nondifferentiability. Let $t_0\geq0$ Then, almost surely, the curve $t\mapsto X(t)$ is not differentiable at $t_0$.
    \item[(h)] Tails (improved version, see Theorem 1.2 in \cite{Cellarosi-Griffin-Osman}). For $R\geq1$ and $\varepsilon>0$
    $$\mathbb{P}\{|X(1)|>R\}=\frac{6}{\pi^2}R^{-6}\left(1+O_\varepsilon(R^{-2+\varepsilon})\right).$$
    \item[(i)] Increments. For every $k\geq1$ and every $0\leq t_0<t_1<t_2<\ldots<t_{k-1}<t_k$ the increments
    $$X(t_1)-X(t_0), X(t_2)-X(t_1),\ldots, X(t_{k-1})-X(t_k)$$ are not independent.
    \item[(j)] Modulus of continuity. For every $\varepsilon>0$ there exists $C_\varepsilon$ such that
    $$\limsup_{h\downarrow0}\sup_{0\leq t\leq 1-h}\frac{|X(t+h)-X(t)|}{\sqrt{t}(\log(1/t))^{1/4+\varepsilon}}\leq C_\varepsilon$$
    almost surely.
\end{itemize}
\end{theorem}
\subsection{Explicit representation of the theta process} The theta process can be realized explicitly as follows. We consider a 6-dimensional Lie group $G$ and a lattice $\Gamma<G$ such that the homogeneous space $\GamG$ has finite volume with respect to the projection $\mu_{\GamG}$ of the Haar measure $\mu$ on $G$ onto the quotient $\GamG$. We assume that $\mu$ is normalized so that $\mu_{\GamG}$ is a \emph{probability} measure, i.e. $\mu_{\GamG}(\GamG)=1$. The geodesic flow on $\GamG$ is given by right-multiplication by $\Phi^s$. We have a  $\mu$-almost everywhere defined complex-valued function $\Theta_\chi$ on $G$ that is $\Gamma$-invariant and hence well defined $\mu_{\GamG}$-almost everywhere on $\GamG$. 
The realizations of the theta process  $t\mapsto X(t)$ can be written in terms of the geodesic trajectories of $\mu_{\GamG}$-typical cosets $\Gamma g\in \GamG$. That is, for $t>0$, 
\begin{align}\label{def-X}
    X(t)=\sqrt{t}\,\Theta_{\chi}(\Gamma g\Phi^{2\log t}),
\end{align}
while $X(0)=0$.

\subsection{The universal Jacobi group \texorpdfstring{$G$}{}}\label{section:universalJacobi}
Let $\h=\{w\in\C:\:\Im(w)>0\}$. The group $\sltr$ acts on $\h$ by M\"{o}bius transformations: for $g=\sma{a}{b}{c}{d}\in\sltr$ we have $z\mapsto gz=\frac{az+b}{cz+d}$. For $g\in\sltr$ as above, set $\epsilon_g(z)=\frac{cz+d}{|cz+d|}$. The universal cover of $\sltr$ is defined as 
\begin{align*}
    \tsltr=\{[g,\beta_g]:\: &g\in\sltr,\mbox{ $\beta_g$ a continuous function on $\h$ s.t. $e^{i\beta_g(z)}=\epsilon_g(z)$}\}.
\end{align*}
The group structure on $\tsltr$ is given by
\begin{align*}
    &[g,\beta_g^1][h,\beta_h^2]=[gh,\beta_{gh}^3], &\beta_{gh}^3(z)=\beta_g^1(hz)+\beta_h^2(z),\\
    &[g,\beta_g]^{-1}=[g^{-1},\beta_{g^{-1}}'], &\beta_{g^{-1}}'(z)=-\beta_{g}(g^{-1}z).
\end{align*}
The Heisenberg group $\Hei$ is defined as $\R^2\times \R$ with group law
\begin{align*}
    (\vecxi, t)(\vecxi',t')=\left(\vecxi+\vecxi', t+t'+\ha\omega(\vecxi,\vecxi')\right), 
\end{align*}
where $\omega\!\left( \sve{x\textcolor{white}{'}\!}{y\textcolor{white}{'}\!},\sve{x'}{y'}\right)=xy'-yx'$ denotes the standard symplectic form on $\R^2$. The universal Jacobi group is then defined as 
\begin{align*}
    G=\tsltr\ltimes\Hei
\end{align*}
and has group law given by
\begin{align}\label{group-law-on-G}
    ([g,\beta_g];\vecxi,\zeta)([g',\beta_{g'}'];\vecxi',\zeta')=([gg',\beta_{gg'}''];\vecxi+g\vecxi',\zeta+\zeta'+\tha\omega(\vecxi,g\vecxi')),
\end{align}
where $\beta_{gg'}''(z)=\beta_g(g'z)+\beta_{g'}'(z)$. Using the identification of $\tsltr$ with $\h\times\R$ (via $[g,\beta_g]\mapsto(gi,\beta_g(i))$) the Haar measure on $G$ is given in coordinates 
\begin{align}\label{coordinates-on-G}
    (x+iy,\phi;\sve{\xi_1}{\xi_2},\zeta)
    \end{align}
    as
\begin{align}\label{def-Haar-measure-on-G}
    \de\mu(g)=\frac{3}{\pi^2}\frac{\de x\de y\de\phi\de\xi_1\de\xi_2\de\zeta}{y^2}.
\end{align}
One can define the \emph{Schr\"{o}dinger-Weil} representation $R$ of $G$ onto the group of unitary operators on $\LtR$ (see \cite{Cellarosi-Marklof}). We will only need the operator $R(g)$ for special group elements $g\in G$. Let us first discuss `rotation matrices'.
Let $k_\phi=\sma{\cos\phi}{-\sin\phi}{\sin\phi}{\cos\phi}$ and $\tilde{k}_\phi=[k_\phi,\beta_{k_\phi}]\in\tsltr$. We have, for $f\in\LtR$,  
\begin{align}
    [R((\tilde{k}_\phi;{\mathbf  0},0))f](w)=\e{\sigma_\phi/8}[\mathscr{F}_\phi f](w)\label{phi-transform_with_prefactor}
\end{align}
where 
\begin{align*}
    \sigma_\phi=\begin{cases}
        2\nu,&\mbox{if $\phi=\nu\pi$, $\nu\in\Z$;}\\
        2\nu+1,&\mbox{if $\nu\pi<\phi<(\nu+1)\pi$, $\nu\in\Z$,}
    \end{cases}
\end{align*}
and the unitary operator $\mathscr{F}_\phi$ is defined by 
\begin{align*}
    [\mathscr{F}_\phi f](w)=\begin{cases}
        f(w),&\mbox{if $\phi\equiv 0\bmod 2\pi$,}\\
        f(-w),&\mbox{if $\phi\equiv \pi\bmod 2\pi$,}\\
        |\sin\phi|^{-\ha}\displaystyle\int_{\R}\psi(w,w')f(w')\de w',&\mbox{if $\phi\not\equiv 0\bmod \pi$,}
    \end{cases}
\end{align*}
where $$\psi(w,w')=\e{\frac{\frac{1}{2}(w^2+(w')^2)\cos\phi-w w'}{\sin\phi}}.$$
We use the notation $f_\phi=R((\tilde{k}_\phi;{\mathbf 0},0))f$.  We set 
\begin{align*}
    \kappa_\eta(f)=\sup_{w,\phi}|f_\phi(w)|(1+|w|^2)^{\eta/2} 
\end{align*}
and define the class of functions
\begin{align*}
    \SEtaR=\left\{f:\R\to\R:\:\kappa_\eta(f)<\infty\right\}.
\end{align*}
The class $ \SEtaR$ extends that of Schwartz functions and we refer to functions in $\SEtaR$ for some $\eta>1$ as \emph{regular}. 
Other elements $g\in G$ we need to consider are `geodesic' ones. For $y>0$ let  $a_y=\sma{\sqrt{y}}{0}{0}{1/\sqrt{y}}\in\sltr$ and let $\tilde{a}_y=[a_y,0]\in\tsltr$. We have, for $f\in\LtR$,  
\begin{align}\label{representation-geodesic}
    [R((\tilde{a}_y;{\bm0},0))f](w)=y^{1/4}f(y^{1/2}w).
\end{align}
The geodesic flow on $G$ is given by right multiplication by $\Phi^t=(\tilde{a}_{e^{-t}};{\bm 0,0})$ and descends to a flow on the quotient $\GamG$. Similarly, for $x\in\R$, we set $n_x:=\sma{1}{x}{0}{1}\in\sltr$ and $\tilde n_x:=[n_x,0]\in\tsltr$. The horocycle flow on $G$ is given by right multiplication by $\Psi^x=[\tilde n_x,\mathbf{0},0]$ and also descends to a flow on $\GamG$.
We also need 
\begin{align}\label{representation-Heisenberg-2}
    [R(([I,0];\sve{0}{\tau},0))f](w)=f(w-\tau)
\end{align}
 and
 \begin{align}\label{representation-Heisenberg-3}
    [R(([I,0];\sve{0}{0},\zeta))f](w)=\e{\zeta}f(w).
\end{align}

\subsection{\texorpdfstring{Defining $\Theta_f$ on $G$ for regular $f$}{}}
Given a function $f\in\SEtaR$ with $\eta>1$ we define a function $\Theta_f:G\to\C$ using the Schr\"{o}dinger-Weil representation $R$ of $G$, namely $\Theta_f(g)=\sum_{n\in\Z}[R(g)f](n)$. This definition can be rewritten for $g=(x+i y,\phi;\sve{\xi_1}{\xi_2},\zeta)\in G$ as
\begin{align}
    \Theta_f(g)&=\sum_{n\in\Z}[R(g)f](n)\nonumber\\
    &=y^{1/4}\e{\zeta-\tha\xi_1\xi_2}\sum_{n\in\Z}f_\phi((n-\xi_2)y^{1/2})\e{\tha(n-\xi_2)^2 x+n\xi_1}.\label{def-Theta_f(g)}
\end{align}
The assumption that $f$ be a regular function guarantees that the series in \eqref{def-Theta_f(g)} converges absolutely.

\subsection{\texorpdfstring{$\Gamma$-invariance of $\Theta_f$ for regular $f$}{}}\label{section:Gamma-invariance}
It is shown in \cite{Marklof2003Annals} that for regular $f$ the function $\gamma\mapsto\Theta_f(g)$ is invariant under the left multiplication by elements of a lattice $\Gamma<G$. Hence $\Theta_f$ is well defined on the homogeneous space $\GamG$. The lattice $\Gamma$ can be written in terms of its generators as $\Gamma=\langle\gamma_1,\gamma_2,\gamma_3,\gamma_4,\gamma_5\rangle$, where
\begin{align}
    \gamma_1&=\left(\left[\sma{0}{-1}{1}{0},\arg\right];\sve{0}{0},\tfrac{1}{8}\right),\label{generators-1}\\
    \gamma_2&=\left(\left[\sma{1}{1}{0}{1},0\right];\sve{1/2}{0},0\right),\label{generators-2}\\
    \gamma_3&=\left(\left[\sma{1}{0}{0}{1},0\right];\sve{1}{0},0\right),\label{generators-3}\\
    \gamma_4&=\left(\left[\sma{1}{0}{0}{1},0\right];\sve{0}{1},0\right),\label{generators-4}\\
    \gamma_5&=\left(\left[\sma{1}{0}{0}{1},0\right];\sve{0}{0},1\right).\label{generators-5}
\end{align}

We think of $\Theta_f$ as a function on a fundamental domain $\mathcal{F}_{\Gamma}$ for the action of $\Gamma$ on $G$.
We can take $\mathcal{F}_\Gamma=\{(x+iy,\phi;\vecxi,\zeta)\}=\mathcal{F}_{\sltz}\times[-\tfrac{\pi}{2},\tfrac{\pi}{2})\times[-\ha,\ha)^2\times[-\ha,\ha)$, where $\mathcal{F}_{\sltz}=\{z\in\h:\: |z|>1,\:-\ha\leq\Re z<\ha \}\cup\{z\in\h:\: |z|=1,\: \Re z\leq0\}$ is a fundamental domain\footnote{We choose  the interval $[-\tfrac{\pi}{2},\tfrac{\pi}{2})$ in the $\phi$-coordinate ---instead of $[0,\pi)$  as done in \cite{Cellarosi-Marklof}--- as it makes our analysis easier in the proof of Theorem \ref{thm-joint-equidistribution}} for the action of $\sltz$ on $\h$.

Since $\Gamma$ is a lattice, the Haar measure $\mu$ on $G$ descends to a finite measure $\mu_{\GamG}$ on $\GamG$ and the normalization in \eqref{def-Haar-measure-on-G} is chosen so that $\mu_{\GamG}(\GamG)=1$, i.e. $\mu_{\GamG}$ is a probability measure on $\GamG$ (equivalently, on $\mathcal{F}_\Gamma$).

\subsection{\texorpdfstring{Defining $\Theta_\chi$ along the entire geodesic of almost every point in $\GamG$}{}} \label{section-Theta_chi}
We aim to discuss $\Theta_\chi$, where $\chi=\mathbf{1}_{(0,1)}$. Unfortunately, $\chi$ is not a regular function and hence $\Theta_\chi$ cannot be defined at every point using \eqref{def-Theta_f(g)}.
Although \eqref{def-Theta_f(g)} defines $\Theta_f$ as an element of $\mathrm{L^4}(\GamG,\mu_\GamG)$ for every $f\in\LtR$ (see Corollary 2.4 in \cite{Cellarosi-Marklof}), we may define $\Theta_\chi$ as an absolutely convergent series  $\mu_{\GamG}$-almost everywhere on $\GamG$ in the following strong sense. For $\mu_{\GamG}$-almost every $\Gamma g\in\GamG$ we need to define $\Theta_\chi(\Gamma g\Phi^{2\log t})$ as an absolutely convergent series \emph{for every $t>0$}. In view of \eqref{def-X}, this allows us to make sense of the realizations $t\mapsto X(t)$ of the theta process $X$ as well define curves in $\C$.

It is important to observe that the map $f\mapsto \Theta_f$ is linear. We write $\chi=\chi_L+\chi_R$ where $\chi_L$ is supported on $[0,\frac{2}{3}]$ and is only discontinuous at $0$, and where $\chi_R=\chi_L(1-\cdot)$. Furthermore, we write $\chi_L(w)=\sum_{j=0}^\infty \Delta(2^j w)$ for some \emph{regular} function $\Delta$ supported on $[\frac{1}{6},\frac{2}{3}]$. Such a function $\Delta$ is given explicitly in Section 3.1 of \cite{Cellarosi-Marklof}; it is twice (but not three times) differentiable and has the property that $\Delta\in\mathcal{S}_2(\R)$. Another possibility is to take $\Delta$ to be a $C^\infty$ compactly supported function, as done in Section 5.3 of \cite{Marklof-Welsh-I}. 
Here we take the latter approach. Recalling \eqref{representation-geodesic}-\eqref{representation-Heisenberg-2} 
we write 
\begin{align}
\chi(w)&=\chi_L(w)+\chi_R(w)\nonumber\\
    &=\sum_{j=0}^\infty \Delta(2^j w)+\sum_{j=0}^\infty \Delta(2^j(1-w))\label{partition-of-unity}\\
    &=\sum_{j=0}^\infty 2^{-j/2} [R((\tilde{a}_{2^{2j}},{\bm 0},0))\Delta](w)+\sum_{j=0}^\infty 2^{-j/2} [R\!\left(([I,0];\sve{0}{1},0)(\tilde{a}_{2^{2j}},{\bm0 },0)\right)\Delta_-](w),\label{partition-of-unity-1}
\end{align}
where $\Delta_{-}(w)=\Delta(-w)$. One can also modify the series in \eqref{partition-of-unity} to obtain, as their sum, the indicator of arbitrary open intervals in $\R$. For intervals of the form $(0,b)$ the two series are denoted $\chi_{b,L}$ and $\chi_{b,R}$ in Section 5.1 of \cite{Cellarosi-Griffin-Osman}. By linearity, \eqref{partition-of-unity-1} yields
\begin{align}
    \Theta_\chi(\Gamma g)=\sum_{j=0}^\infty 2^{-j/2} &\Theta_\Delta(\Gamma g \Phi^{-j\log4})+\sum_{j=0}^\infty 2^{-j/2} \Theta_{\Delta_{-}}(\Gamma g ([I,0];\sve{0}{1},0)\Phi^{-j\log4}),\nonumber
\end{align}
and hence 
\begin{align}
    \Theta_\chi(\Gamma g \Phi^s)=\sum_{j=0}^\infty 2^{-j/2} &\Theta_\Delta(\Gamma g \Phi^{s-j\log4})+\sum_{j=0}^\infty 2^{-j/2} \Theta_{\Delta_{-}}(\Gamma g ([I,0];\sve{0}{1},0)\Phi^{s-j\log4}).\label{Theta_chi(gPhi^s)}
\end{align}
Note that the $j$-th term in each of the series in \eqref{Theta_chi(gPhi^s)} is itself defined by \eqref{def-Theta_f(g)} as an absolutely convergent series for every $\Gamma g\in\GamG$ and every $s\in\R$.  
It is shown in \cite{Cellarosi-Marklof} that there exists a full $\mu_{\GamG}$-measure set $D\subset\GamG$ such that for every $\Gamma g\in D$ the series in in \eqref{Theta_chi(gPhi^s)} converges absolutely for every $s\in\R$ (see Theorem 3.10 and Remark 3.3 in \cite{Cellarosi-Marklof}). As a consequence, for each $t$, we can interpret \eqref{def-X} using \eqref{Theta_chi(gPhi^s)} as sum of two exponentially weighted sums along the backward trajectories of the geodesic flow through the points $\Gamma g\Phi^{2\log t}$ and $\Gamma g ([I,0];\sve{0}{1},0)\Phi^{2\log t}$, where $\Gamma g\in\GamG$ is $\mu_{\GamG}$-random.

\section{Background on Rough Paths}\label{section:background-on-rough-paths}
Rough paths theory is a pathwise theory of calculus for irregular processes such as $X$. In this section, we give a brief introduction to rough paths theory. See \cite{Friz-Hairer-Book,Friz-Victoir-Book} for two excellent introductions to rough paths.
\subsection{\texorpdfstring{H\"older Continuity and $p$-variation}{}}

Rough paths theory treats H\"older continuous driving signals or (in the case of stochastic differential equations) stochastic processes that are almost surely H\"older. Equivalently, rough paths theory can treat driving signals which have finite $p$-variation. We recall the definition of H\"older continuity and finite $p$-variation below. 

\begin{definition}\label{def:Holder}
    Define for $\alpha\in (0,1]$ the space $C^\alpha([0,T],\mathbb R^d)$ of $\alpha$-H\"older functions $f:[0,T]\to \mathbb R^d$ equipped with the norm $$\|f\|_\alpha:=\sup_{t\neq s}\frac{|f(t)-f(s)|}{|t-s|^\alpha}.$$
\end{definition}
\begin{definition}\label{def:p-var}
    Define for $p\geq 1$ the space
    $C^{p-\operatorname{var}}([0,T],\mathbb R^d)$ the space of finite $p$-variation functions equipped with the norm
    $$\|f\|_{p-\operatorname{var}}^p:=\sup_{\mathcal P=\{0=t_0<...<t_{n+1}=T\}}\sum_{k=1}^n|f(t_{k+1})-f(t_k)|^2.$$
\end{definition}
\begin{remark}
   The quadratic variation from It\^o calculus should be contrasted with the $2$-variation. We supply the definition below. Note that crucially for the Brownian motion, the quadratic variation on $[0,T]$ is $T$ however the $2$-variation is infinite almost surely. See section 13.9 in \cite{Friz-Victoir-Book} or Section 9 in \cite{Schilling-book}.
\end{remark}
\begin{definition}\label{def:quadratic-variation}
Consider a stochastic process $Z$ on $[0,T]$ and a sequence of partitions $\mathcal P_n=\{0=t_0<...<t_{n+1}=T\}$ with mesh size $|\mathcal P_n|\to 0$. If the limit in probability 
    $$P-\lim_{n\to\infty} \sum_{k=1}^n |Z(t_{k+1})-Z(t_k)|^2 =:[Z,Z](T)$$
    exists and is independent of the choice of partitions then $[Z,Z](T)$ is called the \textbf{quadratic variation} of $Z.$
\end{definition}
Again, we emphasize that the quadratic variation and the $2$ variation are distinct notions. 
\subsection{Rough Paths Theory}

Rough paths are concerned with differential equations of the form
\begin{equation}\label{eq:diff-eq}
    dY(t)=b(Y(t),t)dt+c(Y(t),t) dX(t),
\end{equation}
where $b,c$ are nice enough functions and $X$ is some ``driving" signal $X(t)=(X^1(t),...,X^d(t))$ with bad regularity - in particular $X$ lacking differentiability. In this case, we may interpret equation \eqref{eq:diff-eq} as the integral equation
\begin{equation}\label{eq:int-eq}
    Y(t)=Y(0)+\int_0^t b(Y(s),s)ds+\int_0^T c(Y(s),s) dX(s),
\end{equation}
where the integral of a vector-valued function is the vector of the integrals of each component.

However, if $X$ is not of bounded variation the last integral in equation \eqref{eq:int-eq} might not be well-defined as a Riemann-Stieltjes integral. In order to solve equation \eqref{eq:int-eq} through some type of Picard iteration, one must first make sense of integrals of the form \begin{equation}\label{eq:int-picard}
    \int_0^t c(X(s),s) dX(s).
\end{equation}
We have the following classical result, due to Young. 
\begin{theorem}[\cite{Young-1936}]
    Let $f\in C^\alpha([0,T],\mathbb R^d)$ and let $g\in C^\beta([0,T],\mathbb R^d)$ with $\alpha+\beta>1$. Let $\mathcal P_n=\{0=t_0<t_1<...<t_n=T\}$ be a sequence of partitions whose mesh size is going to $0$. Then the limit
    \begin{equation}\label{eq:Young-integral}
        \int_0^T f(s) dg(s)=\lim_{n\to \infty}\sum_{k=0}^n f(t_k)\left(g(t_{k+1})-g(t_k)\right)
    \end{equation}
    exists and is independent of the sequence of partitions. The integral in equation \eqref{eq:Young-integral} is known as the Young or Riemann-Stieltjes integral.
\end{theorem}

Therefore if we assume that $c$ is smooth, then the integral in equation \eqref{eq:int-picard} makes sense as a Young integral if $X\in C^\alpha([0,t],\mathbb R^d)$ for $\alpha>1/2$. However, for lower regularity driving signals, the Riemann-Stieltjes integral defined in equation \eqref{eq:Young-integral} needs to include extra terms which represent the ``iterated integrals" of $X$ against itself. We provide a heuristic derivation below.

\textbf{Heuristic Derivation of Rough Integral}\textit{ If $X$ is a signal that is $\alpha$-H\"older continuous with $\alpha\in (1/3,1/2]$ and $F$ is a smooth function, then for a partition of $[0,t]$, $\mathcal P=\{0=t_0<...<t_n=t\}$ we have that (the a-priori ill defined) integral
\begin{align*}
    \int_0^t F(X(r)) dX(r)&=\sum_{k=0}^n \int_{t_k}^{t_{k+1}} F(X(r))dX(r)\\
    &=\sum_{k=0}^n \int_{t_k}^{t_{k+1}} (F(X(t_k))+F'(X(t_k))(X(r)-X(t_k))+O(|r-t_k|^{2\alpha})dX(r)\\
    &=\sum_{k=0}^n \bigg(F(X(t_k))(X(t_{k+1})-X(t_k))+F'(X(t_k))\int_{t_k}^{t_{k+1}}(X(r)-X(t_k)) dX(r)\\
    &~~~~~~~~~~+O(|t_{k+1}-t_k|^{3\alpha})\bigg).
\end{align*}
As $3\alpha>1$, the remainder term should tend to $0$ as the mesh size of the partition $|\mathcal P|=\displaystyle\max_{0\leq k\leq n-1}(t_{k+1}-t_k)$ tends to $0$. \\\\
Note that we restricted $\alpha\in (1/3,1/2]$. This is purely for exposition: for lower $\alpha$ we would need more iterated integrals. In particular, if $\alpha\in (\frac{1}{n+1},\frac{1}{n}]$ then we need $n-1$ additional terms so that $\alpha(n+1)>1$. In the rest of this paper, it will suffice to consider only $\alpha\in (1/3,1/2]$.} \\

The above heuristic reduces the problem of defining the integral of (smooth) functions of $X$ against $X$, i.e. $\int_0^t F(X(r))dX(r)$, to just defining integrals of the increments of $X$ against $X$, i.e. $\int_{t_k}^{t_{k+1}}(X(r)-X(t_k)) dX(r)$.
The latter is referred to as an \emph{iterated integral} since we may interpret $X(r)-X(t_k)$ as $\int_{t_k}^r dX(\tau)$.
If $X$ is irregular then the iterated integral does not exist ``canonically" as a Riemann-Stieltjes integral and therefore must be \textit{postulated}. There are two main properties we would like such a definition to satisfy: one algebraic and one analytic (see properties (ii) and (iii) in Definition \ref{def:rough-path}).

A rough path above a signal $X\in C^\alpha([0,T],\mathbb R^d)$ is therefore a pair $\mathbf{X}_{s,t}=(X_{s,t}, \mathbb X_{s,t})$ where $X_{s,t}$ is the increment of $X$ from $s$ to $t$ and $\mathbb X_{s,t}$ is a \textit{definition} or \textit{postulation} of the iterated integral $\int_{s}^{t}(X(r)-X(s)) \otimes dX(r)$ where again we note that $\int_{s}^{t}(X(r)-X(s)) \otimes dX(r)$ represents the matrix whose $(i,j)$th entry is $\int_{s}^{t}(X^i(r)-X^i(s))  dX^j(r)$. We give a precise definition below.

\begin{definition}\label{def:rough-path}
    Let $T>0$ and let $\Delta_2^{(0,T)}=\{(s,t):0\leq s\leq t\leq T\}$ denote the $2$-simplex. Given a signal $X\in C^\alpha([0,T], \mathbb R^d)$ with $\alpha\in (1/3,1/2]$ we say $\mathbf X=(X, \mathbb X):\Delta_2^{(0,T)}\to \mathbb R^d\oplus \mathbb R^{d\times d}$ is a \textbf{rough path} above $X$ if for all $s,u,t\in [0,T]$ we have
    \begin{align*}
    &(i) \qquad X_{s,t}=X(t)-X(s)\\
        &(ii) \qquad\mathbb X_{s,t}-\mathbb X_{s,u}-\mathbb X_{u,t}=X_{s,u}\otimes X_{u,t}\\
        &(iii)\qquad \|\mathbf X\|_{\alpha, 2\alpha}:=\sup_{t\neq s}\frac{|X_{s,t}|}{|t-s|^\alpha}+\sup_{t\neq s}\frac{|\mathbb X_{s,t}|}{|t-s|^{2\alpha}}<+\infty.
    \end{align*}
In (iii), $|X_{s,t}|$ denotes any vector norm of $X_{s,t}\in\R^d$ and $|\mathbb{X}_{s,t}|$ any matrix norm of $\mathbb{X}_{s,t}\in\R^{d\times d}$.
We denote the space of rough paths by $\mathscr C^{\alpha}$. The topology generated by the seminorm $\|\cdot\|_{\alpha, 2\alpha}$ is called the \textbf{rough topology}.
\end{definition}

\begin{remark}
    It should be stressed that in Definition \ref{def:rough-path} there is no reference to any iterated integrals. The second order process $\mathbb X$ intuitively encodes a notion of iterated integral but formally, $\mathbb X$ is just some process that takes values in $\mathbb R^{d\times d}$. Once we have this purely abstract object $\mathbb X$ we may make the definition
    \begin{equation*}
        \int_s^t (X(r)-X(s))\otimes dX(r) := \mathbb X_{s,t}.
    \end{equation*}
\end{remark}

\begin{definition}
    We define the space $C^{2\alpha}(\Delta_2^{(0,T)},\mathbb R^{2\times 2})$ to be the set of all functions $\mathbb F:\Delta_2^{(0,T)}\to \mathbb R^{2\times 2}$ so that its seminorm
    $$\|\mathbb F\|_{2\alpha}:=\sup_{s\neq t}\frac{|\mathbb F_{s,t}|}{|t-s|^{2\alpha}}<\infty.$$ Note that this is not quite a H\"older space as defined in Definition \ref{def:Holder}. 
\end{definition}

\begin{remark}\label{remark:Levy-area}
    The difference of iterated integrals
    $$L^{i,j}(s,t):=\int_s^t X^i(r) dX^j(r)-\int_s^t X^j(r) dX^i(r)$$
    is often referred to as the \textbf{L\'evy area}. 
    The reason for this terminology is that if $X^i$ and $X^j$ were piecewise smooth, then by Green-Stokes theorem, the (signed) area swept by the curve $(X^i, X^j)$ in $\mathbb R^2$ from time $s$ to time $t$ would be twice $L^{i,j}(s,t)$.
\end{remark}

One can check that if $X\in C^{\infty}([0,T],\R^d)$ is smooth and $\int_s^t (X^i(r)-X^i(s)) dX^j(r)$ is a Riemann-Stieltjes integral, then the pair $\mathbf{X}_{s,t}:=(X(t)-X(s),\int_{s}^{t}(X(r)-X(s)) \otimes dX(r))$ satisfies Definition \ref{def:rough-path}. In this sense, there is a natural embedding of $C^\infty$ smooth functions into the space of rough paths. The closure of these functions under the rough topology is the space of geometric rough paths.

\begin{definition}\label{def:geometric-rough-paths}
    For $\alpha\in (1/3,1/2]$, denote by $\mathring{\mathscr C}_g^{\alpha}$ the image of the embedding $\iota: C^\infty([0,T],\mathbb R^d)\to\mathscr C^\alpha$ where $\iota(f)_{s,t}=(f(t)-f(s),\int_s^t (f(r)-f(s))f'(r) dr)$. Let $\mathscr C_g^\alpha$ be the closure of $\mathring{\mathscr C}_g^{\alpha}$ under the seminorm $\|\cdot \|_{\alpha, 2\alpha}$ defined in Definition \ref{def:rough-path}. We call $\mathscr C_g^\alpha$ the space of $\textbf{geometric rough paths}$. 
\end{definition}

\begin{remark}
    The space of rough paths $\mathscr C^\alpha$, although a subset of the vector space $\mathbb R^d \oplus \mathbb R^{d\times d}$ is not itself a vector space because of the nonlinear relation (ii) in Definition \ref{def:rough-path}. However, the seminorm $\|\cdot\|_{\alpha,2\alpha}$ can be used to define a metric - thus $\mathscr C^\alpha$ is a metric space. 
\end{remark}

We recall the following lemma, from \cite{Friz-Hairer-Book}, Section 2.2. 
\begin{lemma}
    Let $\mathbf X=(X, \mathbb X)\in \mathscr C_g^\alpha$ for some $\alpha \in (1/3,1/2]$. Then for all $(s,t)\in \Delta_2^{(0,T)}$ we have 
    \begin{equation}\label{eq:weak-geometric}
        \frac{1}{2}(\mathbb X_{s,t}+\mathbb X_{s,t}^t)=\frac{1}{2}X_{s,t}\otimes X_{s,t},
    \end{equation}
    where $\mathbb X^t$ is the transpose of $\mathbb X$.
\end{lemma}

\begin{remark}
    The set of rough paths satisfying equation \eqref{eq:weak-geometric} is often called the set of \textbf{weakly geometric} rough paths. See \cite{Friz-Hairer-Book}, Section 2.2 for a discussion on why the distinction between geometric and weakly geometric rarely matters. 
\end{remark}

Given a rough path, we may define the class of admissible integrands as the set of \textit{controlled rough paths}. These are the paths that have a Taylor-type expansion.

\begin{definition}\label{def:controlled-rough-path}
    Given a function $X\in C^\alpha$ with $\alpha\in (1/3,1/2]$ we say that $Y\in C^\alpha$ is \textbf{controlled by} $X$ if there exists a matrix-valued function $Y'\in C^\alpha$ (called the \textbf{Gubinelli derivative}) so that
    \begin{equation*}
        Y(t)-Y(s)=Y'(s)(X(t)-X(s))+R_Y(s,t),
    \end{equation*}
where $R_Y(s,t)\in C^{2\alpha}$. 
\end{definition}

We conclude with theorems on integration of rough paths and on the continuity and solution properties of rough differential equations (RDEs).

\begin{theorem}[\cite{Friz-Hairer-Book}, Theorem 4.10]\label{theorem:definition-integral}
    Let $T>0$. Let $\mathbf{X}=(X,\mathbb X)$ be a rough path as in Definition \ref{def:rough-path}. Let $(Y,Y')$ be controlled by $X$ as in Definition \ref{def:controlled-rough-path}. Then the following limit exists
    \begin{equation}\label{eq:rough-integral}
        \int_0^T Y\, d\mathbf X:=\lim_{|\mathcal P|\to 0}\sum_{[s,t]\in \mathcal P} Y_sX_{s,t}+Y_s' \mathbb X_{s,t}
    \end{equation}
    for any sequence of partitions with $|\mathcal P|\to 0$. The expression in \eqref{eq:rough-integral} is called the \textbf{rough integral} of $Y$ against $\mathbf X$.
\end{theorem}

\begin{theorem}\label{theorem:continuity}[\cite{Friz-Hairer-Book}, Theorems 8.3 and 8.5]
   Let $\mathbf X^1$ and $\mathbf X^2$ be two rough paths as in Definition \ref{def:rough-path} with $\|\mathbf X^1\|_{\alpha, 2\alpha},\|\mathbf X^2\|_{\alpha, 2\alpha}\leq M<\infty$. Consider the two rough differential equation
    \begin{align*}
        dY^1&=f(Y^1)\, d\mathbf X^1\\
        dY^2&=f(Y^2)\, d\mathbf X^2,
    \end{align*}
    where $f\in C_b^3$ is thrice differentiable with bounded derivatives and $Y^1(0)=\xi_1$ and $Y^2(0)=\xi_2$. Then there exist unique solutions $Y^1,Y^2$, that are controlled by $X$ with Gubinelli derivatives $f(Y^1), f(Y^2)$ respectively, in the sense of Definition \ref{def:controlled-rough-path}. We also have the estimate
    \begin{equation*}
        \|Y^1-Y^1\|_\alpha\leq C_M(|\xi_1-\xi_2|+\|\mathbf X^1-\mathbf X^2\|_{\alpha, 2\alpha})
    \end{equation*}
    for some $C_M>0$.
\end{theorem}
\begin{remark}
    Theorem \ref{theorem:continuity} stands in stark contrast with the It\^o setting. In the It\^o setting, the solution map is not continuous in the driving signal (see the first two chapters of \cite{Friz-Hairer-Book}) - that is, two different driving signals might be ``close" but the corresponding solutions have wildly different behavior. However, by adding the ``iterated integrals" we can restore continuity giving rise to a robust approximation theory. 
\end{remark}

\section{\texorpdfstring{Stochastic Calculus for $X$ - Additional Properties}{}}
First we show that the classical It\^o and Young theory do not work for the theta process. To this end,  we discuss additional analytic and probabilistic properties of $X$. We rely on the explicit representation \eqref{def-X}. In this section, we will use the notation $X=(X^1,X^2)$ following the usual identification of $\C$ with $\R^2$.  
\begin{proposition}\label{prop:mean-zero}
    For all $t\geq 0$, $E(X(t))=0$.
\end{proposition}
\begin{proof}
  By property (vii) in Theorem 1.3 in \cite{Cellarosi-Marklof} we have that $X=(X^1,X^2)=(X^2,-X^1)$ in distribution, where we rotated by $\pi/2$. Therefore the distribution of each component is symmetric. 
\end{proof}
\begin{proposition}\label{prop:covariance} We have for all $s,t$ that 
    $$E(X(t) \overline{X(s)})-E(X(t))E(\overline{X(s)})=E(X(t) \overline{X(s)})=\min(s,t).$$ Furthermore, we have that $E[X^i(t)X^j(s))]=\frac{1}{2}\min(t,s)\delta_{i,j}$ for $i,j\in \{1,2\}$.
\end{proposition}
\begin{proof}
    We know that $X$ is centered by Proposition \ref{prop:mean-zero} so we just need to compute $E(X(t) \overline{X(s)})$. By relation (4.57) in \cite{Cellarosi-Marklof} we have that $X(t)=\Theta_{(0,t)}(\Gamma g).$ By Lemma 2.3 or Corollary 2.4 in \cite{Cellarosi-Marklof} we know that $E(X(t) \overline{X(s)})=\int_{\mathbb R} \chi_{(0,t)}(r) \chi_{(0,s)}(r) dr=\min(s,t).$

    Now we will check the components. To do this, again we will use property (vii) in Theorem 1.3 in \cite{Cellarosi-Marklof}.  
    
    If $i\neq j$ we can rotate by $\pi$ to get that $(X^1,X^2)=(-X^1,X^2)$ in distribution. Therefore
    \begin{equation*}
        E(X^1(t)X^2(s))=E(-X^1(t)X^2(s)),
    \end{equation*}
    so $E(X^1(t)X^2(s))=0$.
    
    If $i=j$, rotate by $\pi$ to get that $(X^1,X^2)=(X^2,-X^1)$ in distribution. Therefore $X^1=X^2$ in distribution and we have proved the claim. 
\end{proof}
\begin{corollary}\label{cor:uncorrelated-increments}
    The increments of $X$ are uncorrelated. That is, for $s_1\leq t_1\leq s_2\leq t_2$ we have that $E((X(t_1)-X(s_1))\overline{(X(t_2)-X(s_2))}=0.$ Additionally, we have that 
    \begin{align*}E(|(X(t_1)-X(s_1))|^2&|(X(t_2)-X(s_1))|^2)=(t_1-s_1)(t_2-s_2)\\
    &=E(|(X(t_1)-X(s_1))|^2)E(|(X(t_2)-X(s_2))|^2).
    \end{align*}
\end{corollary}
\begin{proof}
    The first part is immediate from Proposition \ref{prop:covariance}. The second part is from Lemma 2.3 and Corollary 2.4, second part, in \cite{Cellarosi-Marklof} and relation (4.58) in \cite{Cellarosi-Marklof}.
\end{proof}
\begin{proposition}\label{prop:QV}
    The quadratic variation of $X$ is on the interval $(0,t)$ is $t$. The quadratic variations of $X^1$ and $X^2$ are both $\frac{t}{2}$.
\end{proposition}
\begin{proof}
    Take a sequence of partitions $\mathcal P_n=\{0=t_0<...<t_{n+1}=t\}$ with $|\mathcal P_n|\to 0.$ Then we have that 
    \begin{align*}
        E\left(\sum_{k=0}^n |X(t_{k+1})-X(t_k)|^2\right)&=\sum_{k=0}^n E(|X(t_{k+1})-X(t_k)|^2)\\
        &= \sum_{k=0}^n t_{k+1}-t_k\\
        &=t.
    \end{align*}
Now, checking that the quadratic variation we have that
\begin{align*}
    E\left(\left(\sum_{k=0}^n |X(t_{k+1})-X(t_k)|^2-t\right)^2\right)&=  \operatorname{var}\left(\sum_{k=0}^n |X(t_{k+1})-X(t_k)|^2\right)\\
    &=\sum_{k=0}^n\operatorname{var}(|X(t_{k+1})-X(t_k)|^2,
\end{align*}
where in the last step we invoked Corollary \ref{cor:uncorrelated-increments}. Again, using Lemma 2.3 and Corollary 2.4 from \cite{Cellarosi-Marklof}, we have that
\begin{align*}
    \operatorname{var}(|X(t_{k+1})-X(t_k)|^2)&=E[|X(t_{k+1})-X(t_k)|^4]-(E[|X(t_{k+1})-X(t_k)|^2])^2\\
    &=2(t_{k+1}-t_k)^2-(t_{k+1}-t_k)^2\\
    &=(t_{k+1}-t_k)^2.
\end{align*}
Therefore we have that 
\begin{align*}
  \sum_{k=0}^n\operatorname{var}(|X(t_{k+1})-X(t_k)|^2&=\sum_{k=0}^n(t_{k+1}-t_k)^2\\
  &\leq |\mathcal P_n| t\\
  &\to 0
\end{align*}
as $n\to\infty$.
To check that the quadratic variation of each component is $\frac{t}{2}$, we simply use that the quadratic covariation between $X^1$ and $X^2$ is $0$ because $X^1$ and $X^2$ are uncorrelated. 
\end{proof}

\begin{proposition}\label{prop:no-p-var-plt2}
    For all $p<2$ and for all intervals $[0,t]$ we have that $\mu_{\GamG}(\{\Gamma g\in \Gamma\setminus G: \|X\|_{p-\operatorname{var}}=\infty\})=1$. 
\end{proposition}
\begin{proof}
Take a sequence of partitions $\mathcal P_n=\{0=t_0<...<t_{n+1}=t\}$ with $|\mathcal P_n|\to 0.$ By Proposition \ref{prop:QV} we have that the variational sum 
$$V_{n,2}=\sum_{k=0}^n |X(t_{k+1})-X(t_k)|^2$$
satisfies 
\begin{equation*}
    \lim_{n\to\infty} V_{n,2}=t,
\end{equation*}
in $L^2$ and hence in probability. Therefore there exists a subsequence $\mathcal P_{n_k}$ where 
\begin{equation*}
    \lim_{n_k\to\infty} V_{n_k,2}=t,
\end{equation*}
almost surely. For almost all $\Gamma g\in \Gamma \setminus G$ and for $p<2$ we have that 
\begin{equation*}
    V_{n_k,2}(g)\leq \max_{0\leq k\leq n_k+1} (|X(t_{k+1})-X(t_k)|^{2-p}) V_{n_k,p}(g),
\end{equation*}
whose left hand side goes to $t$. The first term on the right hand side goes to $0$ as $X$ is a continuous function on the interval $[0,t]$ and hence uniformly continuous. Therefore $V_{n_k,p}(g)$ must diverge as $k\to\infty$.
\end{proof}
\begin{corollary}\label{cor:not-gamma-holder-ggt12}
    For all $\gamma >1/2$ and for all intervals $[0,t]$ we have that $\mu_{\GamG}(\{\Gamma g\in \Gamma\setminus G: \|X\|_{\gamma}=\infty\})=1$. 
\end{corollary}
\begin{corollary}\label{corollary:young-doesn't-work}
    Young techniques do not work for $X$. 
\end{corollary}
\begin{theorem}\label{theorem:not-semimartingale}
      The process $X=(X^1,X^2)$ is not a semimartingale with respect to the natural filtration $\mathcal F_{t}=\sigma(\{X(s)\}_{0\leq s\leq t})$. That is, there do not exist adapted continuous $V_1,V_2$ of bounded variation and local continuous martingales $M_1,M_2$ so that $X=(V_1+M_1,V_2+M_2).$ 
\end{theorem}

\begin{proof}
Assume that both $X^1$ and $X^2$ have this decomposition.  By Lemma \ref{lemma:semimartingale-is-BM} we have that $X^1, X^2$ are either scaled Brownian motions or of bounded variation. Therefore by Proposition \ref{prop:no-p-var-plt2}, we have that $X^1,X^2$ are just scaled  Brownian motions. However by Theorem 1.3 in \cite{Cellarosi-Marklof}, $X$ only has finitely many moments, a contradiction.

\end{proof}
\subsection{\texorpdfstring{Failure of It\^o calculus for $X$}{}}
It\^o calculus has been a celebrated theory for stochastic calculus for martingales such as the Brownian motion. As the process $X$ shares several of the same properties of Brownian motion, it is natural to first see if we can apply classical It\^o techniques. The well known Bichteler-Dellacherie theorem gives a condition for when It\^o calculus can be applied. See \cite{Protter-Stochastic-Calculus} for an introduction to stochastic calculus with semimartingales. 
\begin{theorem}[Bichteler-Dellacherie]
    Let $X(t)$ be a c\`adl\`ag process. Then the following are equivalent
    \begin{itemize}
        \item $X$ is a ``good integrator".
        \item $X$ is a semimartingale. That is there exists some process $Y$ adapted to $X$ with locally finite variation so that the process $Z=X-Y$ is a local martingale. 
    \end{itemize}
\end{theorem}
\begin{corollary}
    It\^o calculus techniques do not work for $X$. 
\end{corollary}
\begin{proof}
    See Theorem \ref{theorem:not-semimartingale}.
\end{proof}
\section{Strategy of Proof of the Main Theorem}\label{section:strategy}
Theorem \ref{theorem:not-semimartingale} and Corollary \ref{corollary:young-doesn't-work} show that classical approaches to stochastic calculus for $X$ such as the It\^o calculus or Riemann-Stieltjes integration do not work. Therefore we intend to construct a canonical rough path above $X$ which would allow for all of the benefits detailed in Section \ref{section:background-on-rough-paths}. 

To this aim, we consider the piecewise smooth $X_N:[0,T]\to \mathbb R^2$ as in \eqref{def-X_N(t)} and ``enhance" it with well defined, canonical iterated integrals defined through Riemann-Stieltjes integration. We denote this by
$$\mathbf X_N=(X_N,\mathbb X_N),$$
where $\mathbb X_N:\Delta_2^{(0,T)}\to \mathbb R^{2\times 2}.$ Here, we have that
$$\mathbb X_N^{i,j}(s,t):=\int_s^t (X_N^i(r)-X_N^j(s))dX_n^j(r)$$
is a well defined Riemann-Stieltjes integral as $X_N$ is piecewise smooth.

We stress that the sums $S_{\lfloor tN\rfloor}$ appearing in the definition of $X_N$ are random since $x$ is randomly distributed according to a probability measure $\lambda$ on $\R$ absolutely continuous with respect to the Lebesgue measure. The parameters $(\alpha,\beta)$ are fixed and not both rational throughout our discussion.

There are two things we need to show:
\begin{enumerate}
    \item Convergence of the finite dimensional distributions of $\mathbf X_N$.
    \item Tightness of the laws of $\mathbf X_N$ on the space $\mathscr C_{g}^\gamma$ for any $\gamma\in (0,\frac{1}{2})$.
\end{enumerate}
As $\mathscr C_g^\gamma$ can be seen as the space of $\gamma$-H\"older functions ``enhanced" with extra information (see Definition \ref{def:geometric-rough-paths} and part (iii) of Definition \ref{def:rough-path}), it is necessary to first show tightness of $X_N$ on the space of H\"older continuous functions $C^\gamma([0,T],\mathbb R^2)$ for any $\gamma\in (0,\frac{1}{2}).$ This is done in Section \ref{section:tightness-of-X_N}. 

Interestingly, $\mathbb X_N$ can be written in terms of higher rank theta functions and some remainders that go to $0$ in probability. In particular, we need to study a rank two theta sum over a triangle which arises because the iterated integrals $\mathbb X_N$ are double integrals over the triangle. This decomposition is described in Section \ref{section:rough-path-to-theta}. However, before that, we must introduce the higher ranked theta functions that we need. We introduce them in Section \ref{section:intro-theta-2-sum}. In Section \ref{section:defining-theta-T} we construct the rank two theta sum over a triangle. We show tightness in Section \ref{section:tightness} and convergence of the finite dimensional distributions in Section \ref{section:fdds}. 

\section{\texorpdfstring{Tightness of $X_N$}{}}\label{section:tightness-of-X_N}
In this section, we show that the laws of $X_N$ are tight on $C^\gamma([0,T],\R^2)$ for $\gamma<1/2.$ For simplicity, in our proofs we will set $T=1$ and consider the spaces $C^\gamma([0,1],\R^2)$.

\subsection{\texorpdfstring{Lack of Tightness on $C^{\gamma}$ for $\gamma>\frac{1}{2}$}{}}
Although we know from Corollary \ref{cor:not-gamma-holder-ggt12} that the laws of $X_N$ cannot be tight on $C^\gamma([0,T],\R^2)$  for $\gamma>1/2$ because the limiting process is not in $C^\gamma$, for expository reasons we give a proof directly from $X_N$. Note that we consistently suppress the explicit dependence of $X_N$ upon $x$ in our notation.

\begin{proposition}\label{prop:lambda(p-var>R)=1}
    Let $R>0$, let $p<2$ and let $\lambda$ be a Borel probability measure on $\mathbb R$ that is absolutely continuous with respect to Lebesgue measure, then
    \begin{equation}\label{eq:p-var-lt2}
        \sup_{N\in \mathbb N}\lambda \left(\left\{x\in \mathbb R: \|X_N\|_{p-\operatorname{var}}>R
        \right\}\right)=1.
    \end{equation}
\end{proposition}
\begin{proof}
    For each $N$ we can consider the uniform partition to get the lower bound
    \begin{align*}
        \lambda (\{x\in \mathbb R: &\|X_N\|_{p-\operatorname{var}}>R\})\\
        &\geq \lambda \left(\left\{x\in \mathbb R: \left(\sum_{k=0}^{N-1} \left|X_N\left(\frac{k+1}{N}\right)-X_N\left(\frac{k}{N}\right)\right|^p\right)^{\frac{1}{p}}>R
        \right\}\right)\\
        &=\lambda \left(\left\{x\in \mathbb R: \left(\sum_{k=0}^{N-1} \left|\frac{1}{\sqrt{N}}\right|^p\right)^{\frac{1}{p}}>R
        \right\}\right)\\
        &=\lambda \left(\left\{x\in \mathbb R: N^{\frac{1}{p}-\frac{1}{2}}>R
        \right\}\right)\\
        &=\begin{cases}
            1 &\text{ if }N^{\frac{1}{p}-\frac{1}{2}}>R\\
            0&\text{ else}.
        \end{cases}
    \end{align*}
If $p<2$ then $\frac{1}{p}-\frac{1}{2}>0$, proving equation \eqref{eq:p-var-lt2}.
\end{proof}
\begin{corollary}
    Let $R>0$, let $\gamma>1/2$ and let $\lambda$ be a Borel probability measure on $\mathbb R$ that is absolutely continuous with respect to Lebesgue measure, then
    \begin{equation}\label{eq:not-gamma-holder-gt1/2}
        \sup_{N\in \mathbb N}\lambda \left(\left\{x\in \mathbb R: \|X_N\|_{\gamma}>R
        \right\}\right)=1.
    \end{equation}
\end{corollary}
\begin{proof}
    We have the inclusion 
    $$\left\{x\in \mathbb R: \|X_N\|_{p-\operatorname{var}}>R \right\}\subseteq \left\{x\in \mathbb R: \|X_N\|_{\gamma}>R
        \right\}$$ where $p=\frac{1}{\gamma}$. Proposition \ref{prop:lambda(p-var>R)=1} then yields \eqref{eq:not-gamma-holder-gt1/2}.
\end{proof}\subsection{Kolmogorov Tightness Criterion is Not Enough}
The standard way of proving tightness on H\"older spaces is through Kolmogorov's tightness criterion. However Kolmogorov is often non optimal when the underlying sequence of processes lacks moments. The theta process $X$ only has $6-\varepsilon$ moments so Kolmogorov's tightness criterion in insufficient which we show here for expository reasons. 
\begin{proposition}\label{prop:moment-of-increments-1st-order}
    Let $p<4$. Then uniformly over $N$, we have that $$E(|X_N(t)-X_N(s)|^p)\ll |t-s|^{\frac{p}{2}}.$$
\end{proposition}
\begin{proof}
    We have by inverse CDF formula for expectation that
    \begin{align*}
        E\!\left(\frac{|X_N(t)-X_N(s)|^p}{|t-s|^{p/2}}\right)&=\int_0^\infty\lambda\!\left(\left\{x\in \mathbb R:\frac{|X_N(t)-X_N(s)|^p}{|t-s|^{\frac{p}{2}}}>R\right\}\right)dR\\
        &= \int_0^\infty \lambda\!\left(\left\{x\in \mathbb R:\frac{|X_N(t)-X_N(s)|}{|t-s|^{\frac{1}{2}}}>R^{\frac{1}{p}}\right\}\right)dR\\
        &\leq \int_0^\infty \lambda\left(\left\{x\in \mathbb R:|\Theta_\chi|>R^{\frac{1}{p}}\right\}\right)dR\\
        &<\infty,
    \end{align*}
    by Proposition 3.17 in \cite{Cellarosi-Marklof}.
\end{proof}
\begin{proposition}
    Let $\gamma\in (0,\frac{1}{4})$. Then the laws of $X_N$ are tight on $C^\gamma([0,T],\R^2)$.
\end{proposition}
\begin{proof}
    By Kolmogorov tightness criterion, we have that the laws of $X_N$ are tight for all $0<\gamma<\frac{p/2-1}{p}=\frac{1}{2}-\frac{1}{p}.$ Proposition \ref{prop:moment-of-increments-1st-order} with $p=4-\varepsilon$ and $0<\varepsilon<2$ gives the statement.
\end{proof}
\begin{remark}
    If we had a version of Proposition \ref{prop:moment-of-increments-1st-order} that held for all $p$, Kolmogorov tightness would give us tightness on $C^\gamma([0,T],\R^2)$ for all $\gamma<\frac{1}{2}-\frac{1}{p}$ which, as $p\to\infty$, would give us the expected tightness on $C^\gamma([0,T],\R^2)$ for all $\gamma<\frac{1}{2}$. This is not a unique phenomenon about $X$. For example, let $f(t)$ be a smooth deterministic function and let $Z$ be a random variable with no (or few) moments. Then the law of $Z(t)=Zf(t)$ is obviously tight on any H\"older space but Kolmogorov tightness criterion does not work.
\end{remark}

\subsection{\texorpdfstring{Tightness of $X_N$ on $C^\gamma$ for $\gamma<\frac{1}{2}$}{}}
In this subsection we finally prove tightness of $X_N$.

The key observation from \cite{Cellarosi-Marklof}, relating the normalized sum $X_N$ and the theta function $\Theta_\chi$,  is that
\begin{align}\label{key-relation-X_N-Theta_chi}
    X_N(x;t)=e^{\frac{s}{4}}\Theta_\chi(x+i ye^{-s},0;\sve{\alpha+\beta x}{0},0)+\mathcal{R}_N(x),
\end{align}
where $\chi=\mathbf{1}_{(0,1)}$, $y=N^{-2}$, $s=2\log t$, and $\mathcal{R}(x)_N$ goes to zero as $N\to\infty$ uniformly in $t$.
Using the notations of Section \ref{section-Theta_chi} we write the indicator of $(0,1)$ as $\chi=\chi_{L}+\chi_{R}$. 
We will use the following result (or slight variations thereof for indicators of other intervals) several times.
\begin{proposition}[\cite{Cellarosi-Marklof}, Proposition 3.17 therein]\label{prop:3.17-simplified} Let $\lambda$ be a Borel probability measure on $\R$ which is absolutely continuous with respect to the Lebesgue measure. Let $f\in\{\chi_{L},\chi_{R},\chi\}$. Then, uniformly in $y\leq 1$, $R>0$, and $(\vecxi,\zeta)\in\Hei$, we have 
\begin{align}\label{statement-prop3.17}
    \lambda\left(\left\{x\in\R:\:\left|\Theta_{f}(x+i y,0;\vecxi,\zeta)\right|>R\right\}\right)\ll \frac{1}{(1+R)^4}.
\end{align}
\end{proposition}
\begin{proof}
For $f=\chi$, this is Proposition 3.17 in \cite{Cellarosi-Marklof}. A simpler version of the proof of that proposition (obtained by replacing the sum of two series in \eqref{partition-of-unity-1} by either one of the series) yields the statement for $f=\chi_L$ and $f=\chi_R$.
\end{proof}

In \eqref{statement-prop3.17} we have a fixed $0<y\leq 1$ and the implied constant is independent of $y$. We will need a version of Proposition \eqref{prop:3.17-simplified} where $y$ is a function of $x$, as long as it is bounded away from $0$. We have the following
\begin{proposition}\label{proposition:tail-with-random-y}
    Let $\lambda$ be a Borel probability measure on $\mathbb R$ which is absolutely continuous with respect to Lebesgue measure. Let $f\in\{\chi_{L},\chi_{R},\chi\}$. Let $R>0$ and let $0<\varepsilon<\frac{1}{4}$. Then for all $y=y(x)$ with $0<y_0\leq  y(x)\leq 1$,  $(\vecxi(x), \zeta(x))\in\Hei$ we have the bound
    \begin{equation*}
        \lambda\left(\left\{x\in \mathbb R: \left|\Theta_{f}(x+iy(x),0; \vecxi(x), \zeta(x))\right|>\frac{R}{(y(x))^\varepsilon}\right\}\right)\ll_\varepsilon\frac{1}{(1+R)^4},
    \end{equation*}
   where the implied constant is independent of $y_0$.
\end{proposition}
We need a few lemmas to proceed. The first lemma is perhaps interesting in its own right due to the connection with the Riemann-$\zeta$ function. It is the reason why the `discount' by $(y(x))^\varepsilon$ in Proposition \ref{proposition:tail-with-random-y} is necessary. As $\varepsilon\to 0$ the following constant blows up due to the pole in the $\zeta$ function at $1$. We first recall the function $H$ defined in equation (3.125) in \cite{Cellarosi-Marklof}
\begin{equation}
    H(x+iy):=\sum_{\gamma \in \Gamma_\infty \setminus \operatorname{PSL}(2,\mathbb Z)}y_\gamma^{1/2}\chi_{[1/2,\infty)}(y_\gamma^{1/2}),
\end{equation}
where $\Gamma_\infty=\left\{\begin{pmatrix}
    1&&x\\
    0&&1
\end{pmatrix}: x\in \mathbb Z\right\}$ is the stabilizer of the cusp and $y_\gamma:=\Im(\gamma (x+iy))$ where $\gamma$ acts by M\"obius transformations. $H$ can be extended to a function on $G$ by defining $H\left(x+iy;\phi,\boldsymbol \xi,\zeta\right):=H(x+iy).$

\begin{lemma}\label{lemma:integral-over-gamma-R4}
    Let $R> 1$, $0<\varepsilon<\frac{1}{4}$ and $r^\varepsilon_\gamma(y)=\mathbf 1_{\{y_\gamma^{1/4}>R/y^\varepsilon\}}$, where $y(x)$ is dependent on $x$ and $0<y\leq 1$, as Proposition \ref{proposition:tail-with-random-y}. Then 
    \begin{equation*}
        \int_0^1\sum_{\gamma \in \Gamma_\infty \setminus \hat \Gamma} r_\gamma^\varepsilon(y)dx \leq 2\sum_{c=1}^\infty \frac{\phi(c)}{c^{2/(1-4\varepsilon)}}R^{-4}=2\frac{\zeta\left(\frac{2}{1-4\varepsilon}-1\right)}{\zeta\left(\frac{2}{1-4\varepsilon}\right)}R^{-4},
    \end{equation*}
where $\zeta$ is the classical Riemann-$\zeta$ function and $\phi$ is the Euler totient function.
\end{lemma}
\begin{proof}
    Similarly to equation (3.129) in the proof of Lemma 3.19 in \cite{Cellarosi-Marklof}, we have that 
    \begin{equation}
        \int_0^1\sum_{\gamma \in \Gamma_\infty \setminus \hat \Gamma} r_\gamma^\varepsilon(y)dx=r_{\operatorname{Id}}^\varepsilon(y)+\sum_{c=1}^\infty \sum_{d' \operatorname{mod} c, (c,d')=1}\int_{-\infty}^\infty r_{\operatorname{Id}}^\varepsilon\left(\frac{y(x)}{c^2(x^2+y^2(x))}\right)dx.
    \end{equation}
Since $R> 1$ and $y\leq 1$ we have that $r_{\operatorname{Id}}^\varepsilon(y)=0$, so we just need to study $J_c:=\int_{-\infty}^\infty r_{\operatorname{Id}}^\varepsilon\left(\frac{y(x)}{c^2(x^2+y^2(x))}\right)dx$. Rewriting $J_c$ gives
\begin{equation}\label{eq:trivial-number-bound}
    J_c=\int_{-\infty}^\infty \mathbf 1_{\left\{\frac{c^2(x^2+y^2)}{y}< R^{-4}y^{4\varepsilon}\right\}}(x)dx=\int_{-\infty}^\infty \mathbf 1_{\left\{\frac{c^2(x^2+y^2)}{y^{1+4\varepsilon}}< R^{-4}\right\}}(x)dx.
\end{equation}
We have the bound
\begin{equation*}
    x^{1-4\varepsilon}y^{1+4\varepsilon}\leq x^2+y^2
\end{equation*}
for positive $x,y$. Indeed, letting $x=ky$ where $k=x/y>0$ gives that equation \eqref{eq:trivial-number-bound} reduces to showing $k^{1-4\varepsilon}\leq k^2+1$, which is easily verifiable. 
By enlarging the region of integration, we have that 
\begin{equation*}
    J_c \leq \int_{-\infty}^\infty \mathbf 1_{\left\{c^2 |x|^{1-4\varepsilon}< R^{-4}\right\}}(x)dx,
\end{equation*}
which is readily computed to yield the bound
\begin{equation*}
    J_c\leq 2 R^{-\frac{4}{1-4\varepsilon}}\frac{1}{c^{\frac{2}{1-4\varepsilon}}}\leq 2 R^{-4}\frac{1}{c^{\frac{2}{1-4\varepsilon}}}.
\end{equation*}
Therefore
\begin{equation*}
     \int_0^1\sum_{\gamma \in \Gamma_\infty \setminus \hat \Gamma} r_\gamma^\varepsilon(y)dx\leq \sum_{c=1}^\infty 2 R^{-4}\frac{1}{c^{\frac{2}{1-4\varepsilon}}} \phi(c),
\end{equation*}
where $\phi$ is the Euler totient function. It is well known that the Dirichlet series $\sum_{c=1}^\infty \frac{\phi(c)}{c^s}$ converges to $\frac{\zeta(s-1)}{\zeta(s)}$ if $\Re(s)>2$. Using that $\frac{2}{1-4\varepsilon}>2$ concludes the proof.
\end{proof}
\begin{lemma}\label{lemma:tail-bound-smooth-f-random-y}
    Let $f\in \mathcal S_\eta(\R)$, $\eta>1$, $y(x)$ and $\lambda$ as in Proposition \ref{proposition:tail-with-random-y}. Then
    \begin{equation*}
        \lambda\left(\left\{x\in \mathbb R:|\Theta_f(x+iy(x),\phi, \mathbf \xi(x), \zeta(x))|>\frac{R}{(y(x))^\varepsilon}\right\}\right)\ll \frac{1}{(1+R)^4}
    \end{equation*}
\end{lemma}
\begin{proof}
    Lemma 3.18 in \cite{Cellarosi-Marklof} gives that 
    \begin{align*}
        \int_{\mathbb R} \mathbf 1_{\left\{|\Theta_f(x+iy(x),\phi, \mathbf \xi(x), \zeta(x))|>R/y^\varepsilon(x)\right\}}&\lambda (dx)\leq \int_{\mathbb R} \mathbf 1_{\left\{H(x+iy(x))>R/(C_1y^\varepsilon(x))\right\}}\lambda (dx).
    \end{align*}
Since $H(z+1)=H(z)$ we have 
\begin{equation*}
    \int_{\mathbb R} \mathbf 1_{\left\{H(x+iy(x))>R/(C_1y^\varepsilon(x))\right\}}\lambda (dx)=\int_{\mathbb R/\mathbb Z} \mathbf 1_{\left\{H(x+iy(x))>R/(C_1y^\varepsilon(x))\right\}}\lambda_{\mathbb Z}(dx),
\end{equation*}
where $\lambda_{\mathbb Z}$ is pushforward under the projection map. Since $\lambda$ is absolutely continuous with respect to Lebesgue measure, we have that 
\begin{equation*}
    \int_{\mathbb R/\mathbb Z} \mathbf 1_{\left\{H(x+iy(x))>R/(C_1y^\varepsilon(x))\right\}}\lambda_{\mathbb Z}(dx)\ll \int_{\mathbb R/\mathbb Z} \mathbf 1_{\left\{H(x+iy(x))>R/(C_1y^\varepsilon(x))\right\}}dx
\end{equation*}
For $R\geq C_1$ we have that
\begin{align*}
   \int_{\mathbb R/\mathbb Z} \mathbf 1_{\left\{H(x+iy(x))>R/(C_1y^\varepsilon(x))\right\}}dx&=\int_0^1\sum_{\gamma \in \Gamma_\infty \setminus \hat \Gamma} r_\gamma^\varepsilon(y)dx\\
   &\ll R^{-4},
\end{align*}
where the last inequality comes from Lemma \ref{lemma:integral-over-gamma-R4}.
\end{proof}

\begin{proof}[Proof of Proposition \ref{proposition:tail-with-random-y}]
The proof is analogous to that of Proposition 3.17 given in \cite{Cellarosi-Marklof}. We only give a proof for $f=\chi$, since for $f=\chi_L,\chi_R$ we can simply suppress one of the two sums in the argument below. We can write, for  $J_x=\lceil \log_2 (y(x))^{-1}\rceil$,
\begin{align*}
\Theta_\chi(x+i&y(x),0, \mathbf \xi(x), \zeta(x))=\sum_{j=0}^{J_x} 2^{-j/2}\Theta_\Delta((x+iy(x),0, \mathbf \xi(x), \zeta(x))\Phi^{-(2\log 2) j})\\
&+\sum_{j=0}^{J_x} 2^{-j/2}\Theta_{\Delta_-}\left((x+iy(x),0, \mathbf \xi(x), \zeta(x))\left(1; \begin{pmatrix}
    0\\
    1
\end{pmatrix},0\right)\Phi^{-(2\log 2) j}\right)
\end{align*}
Using the triangle inequality we have that 
\begin{align*}
|\Theta_\chi(x+iy(x),0, \mathbf \xi(x), \zeta(x))|&\leq\sum_{j=0}^{J_x} 2^{-j/2}|\Theta_\Delta((x+iy(x),0, \mathbf \xi(x), \zeta(x))\Phi^{-(2\log 2) j})|\\
&+\sum_{j=0}^{J_x} 2^{-j/2}\left|\Theta_{\Delta_-}\!\left((x+iy(x),0, \mathbf \xi(x), \zeta(x))\left(1; { \begin{pmatrix}
    0\\
    1
\end{pmatrix}},0\right)\Phi^{-(2\log 2) j}\right)\right|\\
&\leq \sum_{j=0}^{J_0} 2^{-j/2}|\Theta_\Delta((x+iy(x),0, \mathbf \xi(x), \zeta(x))\Phi^{-(2\log 2) j})|\\
&~~~~+\sum_{j=0}^{J_0} 2^{-j/2}\left|\Theta_{\Delta_-}\!\left((x+iy(x),0, \mathbf \xi(x), \zeta(x))\left(1; {\begin{pmatrix}
    0\\
    1
\end{pmatrix}},0\right)\Phi^{-(2\log 2) j}\right)\right|,
\end{align*}
where $J_0=\lceil \log_2y_0^{-1}\rceil$. The rest of the proof is identical to the proof of Proposition 3.17 in \cite{Cellarosi-Marklof}.
\end{proof}
    With Proposition \ref{proposition:tail-with-random-y} in hand, we can show the following H\"older estimate. 
\begin{proposition}\label{prop:tail-bound-Holder}
    Let $R>0$, let $0<\gamma<\frac{1}{2}$ and let $\lambda$ be a probability measure on $\mathbb R$ absolutely continuous with respect to Lebesgue measure. Then we have the tail bound
    \begin{equation*}
        \lambda(\left\{x\in \mathbb R:\|X_N\|_\gamma>R\right\})\ll \frac{1}{(1+R)^4},
    \end{equation*}
    uniformly over $N$, but where the implied constant is dependent on $\gamma$.
\end{proposition}
\begin{proof}
By standard considerations we can just check the H\"older norm on grid points. More precisely, we have that 
\begin{equation}\label{comparison-Holder-norm_v_grid-points}
     \|X_N\|_\gamma\ll \max_{1<k<k'\leq N} \frac{\frac{1}{\sqrt N}\left|X_N\!\left(\frac{k'}{N}\right)-X_N\!\left(\frac{k}{N}\right)\right|}{\left|\frac{k'-k}{N}\right|^{\gamma} },
\end{equation}
where the implied constant is independent of $N$. Therefore, uniformly over $N$ we have ($C$ being the constant implied by \eqref{comparison-Holder-norm_v_grid-points})
\begin{align*}
     \lambda(\{x\in \mathbb R:\|X_N\|_\gamma>R\})
     &\ll 
     \lambda\left(\left\{x\in \mathbb R:\max_{1\leq k<k'\leq N} \frac{\left|X_N\!\left(\frac{k'}{N}\right)-X_N\!\left(\frac{k}{N}\right)\right|}{\left|\frac{k'-k}{N}\right|^{\gamma} }>\frac{R}{C}\right\}\right)\\
     &= \lambda\left(\left\{x\in \mathbb R: \frac{\frac{1}{\sqrt N}\left|\sum_{k_\ast(x)<j\leq k_\ast'(x)}z_j\right|}{\left|\frac{k_\ast'(x)-k_\ast(x)}{N}\right|^{\gamma} }>\frac{R}{C}\right\}\right),
\end{align*}
where $k_\ast(x)$ and $k_\ast'(x)$ are the arg max (if there are multiple pairs the arbitrary choice of the least indices is made). Note that we have $y(x):=\frac{1}{(k_\ast'(x)-k_\ast(x))^2}\geq \frac{1}{N^2}$. Rewriting the sum in terms of the function $\Theta_\chi$, similarly to \eqref{key-relation-X_N-Theta_chi}, we have that 
\begin{align*}
     \lambda(\{x\in \mathbb R:&\|X_N\|_\gamma>R\})\\
     &\ll \lambda\bigg(\bigg\{x\in \mathbb R: \left|\frac{N}{k_\ast'(x)-k_\ast(x)}\right|^{\gamma-\frac{1}{2}}\left|\Theta_\chi\left(x+i\frac{1}{(k_\ast'(x)-k_\ast(x))^2},0;{\begin{pmatrix}
     \alpha+\beta x\\
     0\end{pmatrix}}, 0 
     \right)\right|>\frac{R}{C}\bigg\}\bigg)\\
     &=\lambda\bigg(\bigg\{x\in \mathbb R: \left|\Theta_\chi\left(x+iy(x),0;{ \begin{pmatrix}
     \alpha+\beta x\\
     0\end{pmatrix}}, 0 
     \right)\right|>\frac{RN^{\frac{1}{2}-\gamma}}{Cy^{\frac{1}{4}-\frac{\gamma}{2}}(x)}\bigg\}\bigg)\\
     &\leq \lambda\bigg(\bigg\{x\in \mathbb R: \left|\Theta_\chi\left(x+iy(x),0;{\begin{pmatrix}
     \alpha+\beta x\\
     0\end{pmatrix}}, 0 
     \right)\right|>\frac{R}{Cy^{\frac{1}{4}-\frac{\gamma}{2}}(x)}\bigg\}\bigg),
\end{align*}
where $y(x)=\frac{1}{(k_\ast'(x)-k_\ast(x))^2}$. Proposition \ref{proposition:tail-with-random-y} concludes the proof.
\end{proof}
The final proposition of this section gives tightness on $C^\gamma$.
\begin{proposition}\label{prop:1st-order-tightness}
    For all $\gamma\in (0,\frac{1}{2})$, the laws of $X_N$ are tight on $C^\gamma([0,T],\R^2)$.
\end{proposition}
\begin{proof}
Let $\gamma\in(0,\frac{1}{2})$ and pick $\gamma'\in (\gamma,\frac{1}{2})$. The embedding $C^{\gamma'}\hookrightarrow C^\gamma$ is compact. Proposition \ref{prop:tail-bound-Holder} gives that balls in $C^{\gamma'}$ can have arbitrarily large probabilities uniformly in $N$. 
\end{proof}
\section{\texorpdfstring{The 2-dimensional theta function  $\Theta^{(2)}$  and $\mathbb X_N$}{}}\label{section:intro-theta-2-sum}
As we will shortly see, $\mathbb X_N$ can be described in terms of higher ranked theta sums. We introduce the basic theory for the theta sums we need in this section. For a more general theory, see \cite{Marklof-Welsh-I,Marklof-Welsh-II}. The definitions here are analogues of those given in \cite{Cellarosi-Marklof} and in addition agree with those in \cite{Marklof-Welsh-I,Marklof-Welsh-II} when we make the appropriate mappings.
\begin{definition}
    For $F\in L^2(\mathbb R^2,\mathbb R)$ and pair of angles $(\phi_1,\phi_2)$ we define the function $\mathscr{F}_{\phi_1,\phi_2} F$ in $9$ cases. First, if $\phi_1\not \equiv 0 \mod \pi $ and $\phi_2\not \equiv 0 \mod \pi $ then 
    \begin{equation*}
    \begin{split}
        [\mathscr{F}_{\phi_1,\phi_2}f](w_1,w_2)
        &=\frac{1}{|\sin \phi_1|^{1/2}}\frac{1}{|\sin \phi_2|^{1/2}}\int_{\mathbb R^2}e\left(\frac{1/2(w_1^2+(w_1')^2)\cos \phi_1-w_1 w_1'}{\sin \phi_1}\right)\\
   & \times e\left(\frac{1/2(w_2^2+(w_2')^2)\cos \phi_2-w_2 w_2'}{\sin \phi_2}\right)F(w_1',w_2')dw_1' dw_2'.
        \end{split}
    \end{equation*}
    If $\phi_1 \equiv 0 \mod 2\pi$ then $[\mathscr{F}_{\phi_1,\phi_2} F](w_1,w_2)=[\mathscr{F}_{\phi_2} F](w_1,w_2)$ where $\mathscr{F}_{\phi_2} F$ denotes the first-order $\phi$-transform in the second coordinate. If $\phi_1 \equiv \pi \mod 2\pi$ then $[\mathscr{F}_{\phi_1,\phi_2}F](w_1,w_2)=[\mathscr{F}_{\phi_2}F](-w_1,w_2)$ where $\mathscr{F}_{\phi_2} F$ denotes the first-order $\phi$-transform in the second coordinate. We similarly define for the other $6$ cases. We also similarly define $f_{\phi_1,\phi_2}$ as in equation \eqref{phi-transform_with_prefactor}. 
\end{definition}
\begin{definition}
    Given a function $F:\R^2\to\R$ and $\eta_1,\eta_2>0$ define $\kappa_{\eta_1,\eta_2}(f)$ by 
    \begin{equation*}
        \kappa_{\eta_1,\eta_2}(F):=\sup_{\phi_1,\phi_2, w_1,w_2}|F_{\phi_1,\phi_2}(w_1,w_2)|(1+|w_1|)^{\eta_1}(1+|w_2|)^{\eta_2}.
    \end{equation*}
Additionally, define $\mathcal S_{\eta_1,\eta_2}(\R^2)=\{F: \mathbb R^2\to\mathbb R: \kappa_{\eta_1,\eta_2}(F)<\infty\}.$
We shall refer to functions in $\mathcal{S}_{\eta_1,\eta_2}(\R^2)$ with $\eta_1,\eta_2>1$ as \emph{regular}.
\end{definition}

Using the Schr\"{o}dinger-Weil representation $R$ of $G$ on $\mathcal U(\LtR)$, we can construct the product representation $R^{(2)}$ of $G^2$ on $\mathcal U(\mathrm L^2(\R^2))$. In practice, for $(g_1,g_2)\in G^2$ the unitary operator $R^{(2)}(g_1,g_2)$ acts on $f\in\mathrm{L}^2(\R^2)$ as the operator $R(g_1)$ on $w_1\mapsto f(w_1,w_2)$ and as the operator $R(g_2)$ on $w_2\mapsto f(w_1,w_2)$.

\begin{definition}\label{def:Theta^(2)}
    For regular $f\in\mathcal{S}_{\eta_1,\eta_2}(\R^2)$ with define the 2-dimensional theta function $\Theta^{(2)}_f:G\times G\to\C$ as 
    \begin{align}\label{def-Theta-2-representation-R^2}
        \Theta^{(2)}_{F}(g_1,g_2):=\sum_{n_1,n_2\in\Z}[R^{(2)}(g_1,g_2)F](n_1,n_2).
    \end{align}
    Using the Iwasawa-Heisenberg coordinates 
    $g_j=(x_j+iy_j,\phi_j;\sve{\xi_{j,1}}{\xi_{j,2}},\zeta_j)$, with $j=1,2$,  definition \eqref{def-Theta-2-representation-R^2} reads as 
        \begin{align}
        \label{def-big-Theta}
        \Theta_F^{(2)}(g_1,g_2):=&y_1^{\frac{1}{4}}y_2^{\frac{1}{4}}\e{\zeta_1-\xi_{1,1}\xi_{1,2}}\e{\zeta_2-\xi_{2,1}\xi_{2,2}}\sum_{n_1,n_2\in \mathbb Z} F_{\phi_1,\phi_2}\!\left((n_1-\xi_{1,2})y_1^{\frac{1}{2}},(n_2-\xi_{2,2})y_2^{\frac{1}{2}}\right)\\
        &\times e\left(\tfrac{1}{2}(n_1-\xi_{1,2})^2x_1+n_1\xi_{1,1}\right)e\left(\tfrac{1}{2}(n_2-\xi_{2,2})^2x_2+n_2\xi_{2,1}\right).\nonumber
        \end{align}
\end{definition}
We show an analogue of Lemma 2.1 in \cite{Cellarosi-Marklof} where we bound $\Theta_F^{(2)}$ in terms of the height in the cusp of $G\times G$ 
\begin{lemma}\label{lemma:growth-in-cusp-second-order}
    Given $\xi_{1,2}\in \mathbb R$ and $\xi_{2,2}\in \mathbb R$, write $\xi_{1,2}=m_1+\theta_1$ and $\xi_{2,2}=m_2+\theta_2$ where $m_1,m_2\in \mathbb Z$ and $-\frac{1}{2}\leq \theta_1,\theta_2< \frac{1}{2}$. Let $\eta_1,\eta_2>1$. Then there exists a constant $C_{\eta_1,\eta_2}$ so that for all $F\in \mathcal{S}_{\eta_1,\eta_2}(\R^2)$, all $y_1,y_2\geq \frac{1}{2}$, and all $x_1,x_2, \phi_1,\phi_2,\boldsymbol{\xi}_1, \boldsymbol{\xi}_2, \zeta_1,\zeta_2$ that 
    \begin{equation}
    \begin{split}
        \left|\Theta_F^{(2)}(g_1,g_2)\right|\leq C_{\eta_1,\eta_2}\kappa_{\eta_1,\eta_2}(F)y_1^{\frac{1}{4}}y_2^{\frac{1}{4}},
        \end{split}
    \end{equation}
    where $g_1, g_2$ are as in Definition \ref{def:Theta^(2)}.
\end{lemma}
\begin{proof}
    We just need to bound \begin{align*}M:=\sum_{n_1,n_2\in\Z}& F_{\phi_1,\phi_2}\!\left((n_1-\xi_{1,2})y_1^{\frac{1}{2}},(n_2-\xi_{2,2})y_2^{\frac{1}{2}}\right)\\
    &~\times e\left(\tfrac{1}{2}(n_1-\xi_{1,2})^2x_1+n_1\xi_{1,1}\right)e\left(\tfrac{1}{2}(n_2-\xi_{2,2})^2x_2+n_2\xi_{2,1}\right).\end{align*} In this case, we have that
    \begin{align*}
        |M|
        &\leq\sum_{n_1,n_2\in\Z} \left|F_{\phi_1,\phi_2}\left((n_1-\xi_{1,2})y_1^{\frac{1}{2}},(n_2-\xi_{2,2})y_2^{\frac{1}{2}}\right)\right|\\
        &\leq \sum_{n_1,n_2}\frac{\kappa_{\eta_1,\eta_2}(F)}{\left(1+\left|(n_1-\xi_{1,2})y_1^{\frac{1}{2}}\right|\right)^{\eta_1}\left(1+\left|(n_2-\xi_{2,2})y_2^{\frac{1}{2}}\right|\right)^{\eta_2}}\\
         &\leq \sum_{n_1,n_2}\frac{\kappa_{\eta_1,\eta_2}(F)}{\left(1+\left|(n_1-\xi_{1,2})(\frac{1}{2})^{\frac{1}{2}}\right|\right)^{\eta_1}\left(1+\left|(n_2-\xi_{2,2})(\frac{1}{2})^{\frac{1}{2}}\right|\right)^{\eta_2}}\\
        &= C_{\eta_1,\eta_2} \kappa_{\eta_1,\eta_2}(F).
    \end{align*}
\end{proof}
\subsection{\texorpdfstring{Invariance of $\Theta^{(2)}$ under $\Gamma\times \Gamma$}{}}
Let $\Gamma$ be as in Section \ref{section:Gamma-invariance}. We establish automorphy of $\Theta_f^{(2)}$ under $\Gamma\times \Gamma$ for regular $F$.
\begin{proposition}\label{prop:invariance-under-gamma-gamma}
    Let $F\in S_{\eta_1,\eta_2}(\mathbb R^2)$. Then $\Theta_F^{(2)}$ is a well defined function $(\Gamma\times \Gamma)\backslash (G\times G)\to \mathbb C$.
\end{proposition}
\begin{proof}
    Using the Iwasawa-Heisenberg coordinates $g_j=(x_j+iy_j,\phi_j;\sve{\xi_{j,1}}{\xi_{j,2}},\zeta_j)$, $j=1,2$ we have that 
        \begin{align*}
        \Theta_F^{(2)}(g_1,g_2)
        &=y_1^{\frac{1}{4}}\,\e{\zeta_2-\xi_{2,1}\xi_{2,2}} \sum_{n_2\in\Z} \e{\frac12(n_2-\xi_{2,2})^2x+n_2\xi_{2,1}}\Theta_{f_{g_2,n_2}}(g_1),
        \end{align*}
        where $f_{g_2,n_2}:\mathbb R\to\mathbb R$ is defined by $w\mapsto F_{0,\phi_2}(w, (n_2-\xi_{2,2})y_2^{1/2}).$ Using the invariance of $\Theta$ under $\Gamma$ gives that $\Theta_F^{(2)}(\gamma g_1,g_2)=\Theta_F^{(2)}(g_1,g_2)$ for any $\gamma\in \Gamma$. A similar argument shows that $\Theta_F^{(2)}( g_1,\gamma g_2)=\Theta_F^{(2)}(g_1,g_2)$ for any $\gamma\in \Gamma$. Therefore $\Theta_F^{(2)}(\gamma_1 g_1,\gamma_2 g_2)=\Theta_F^{(2)}(g_1,g_2)$ for any $\gamma_1,\gamma_2\in \Gamma$, concluding the proof. 
\end{proof}
We also define a fundamental domain of $(\Gamma\times \Gamma)$   acting on $G\times G$ as $\mathcal F_{\Gamma}\times \mathcal F_{\Gamma},$ where $\mathcal F_\Gamma$ is defined as in equation (2.57) in \cite{Cellarosi-Marklof}.

\section{\texorpdfstring{Rewriting $\mathbf{X}_N$ in terms of horocycle lifts and $(\Theta,\Theta^{(2)})$}{}}\label{section:rough-path-to-theta}

Recall $X_N(x;t)$ defined in \eqref{def-X_N(t)}. For simplicity, as already done above, we drop the dependence upon $x$ (as well as the already dropped dependence upon $\alpha,\beta$ which are fixed) in the notation and simply write $X_N(t)$.
We use the shorthands   
\begin{align*}
    X_N^1(t)=\Re(X_N(t)),\hspace{.3cm}X_N^2(t)=\Im(X_N(t)).
\end{align*}
Equivalently, $X_N^1(t)$ (resp. $X_N^2(t)$) is obtained by replacing $\e{\cdot}$ with $\cos(2\pi\cdot)$ (resp. $\sin(2\pi\cdot)$) in \eqref{def-X_N(t)}.
We also use the notations
\begin{align}
    a_k&=a_k(x;\alpha,\beta)=\cos\!\left(2\pi\left[(\tha k^2+k \beta)x+\alpha k\right]\right),\label{shorthand-a_k}\\
    b_k&=b_k(x;\alpha,\beta)=\sin\!\left(2\pi\left[(\tha k^2+k \beta)x+\alpha k\right]\right),\label{shorthand-b_k}\\
    z_k&=z_k(x;\alpha,\beta)=\e{(\tha k^2+k \beta)x+\alpha k}=a_k+ib_k.\label{shorthand-z_k}
\end{align}

Define $\mathbb{X}_N:\Delta_2^{(0,T)}\to\R^{2\times2}$ as 
$\mathbb{X}_N(s,t)=\left(\mathbb{X}_N^{i,j}(s,t)\right)_{i,j\in\{1,2\}}$, where 
\begin{align*}
    \mathbb{X}_N^{i,j}(s,t)=\int_{s}^{t}\int_{s}^{r} \de X_N^i(u)\,\de X_N^j(r)=\int_{s}^{t}(X_N^i(r)-X_N^i(s))\,\de X_N^j(r)
\end{align*}
We need to consider the functions $\mathbf{X}_N:\Delta_2^{(0,T)}\to\R^2\times\R^{2\times 2}$ given by 
\begin{align}\label{mathbfX_N(s,t)}
\mathbf{X}_N(s,t)=\mathbf{X}_N(x;\alpha,\beta;s,t)=\left(X_N(t)-X_N(s);\mathbb{X}_N(s,t)\right).
\end{align}
Note that in our set up, $\mathbf{X}_N$ is a \emph{random} function on the simplex since it depends on $x$ and $x$ is distributed according to the probability measure $\lambda$. The random function $\mathbf{X}_N$ also depends on the parameters $\alpha$ and $\beta$, which are fixed and not both rational. 
We aim to show that the finite-dimensional distributions of these random functions have a limit in distribution as $N\ti$. Before doing that, we will write $\mathbf{X}_N(s,t)$ in terms of a pair of horocycle lifts in $\GamG$ and in terms of theta functions $\Theta$ and $\Theta^{(2)}$.

\begin{lemma}\label{lemma:grid-approximation-with-continuous-time}
Given $s<t$, and $i,j\in\{1,2\}$, we have for almost every $x$ that
\begin{align}\label{XX_N-at-times_m/N_n/N+O-term}
\mathbb{X}^{i,j}_N(s,t)=\mathbb{X}^{i,j}_N\!\left(\tfrac{m}{N},\tfrac{n}{N}\right)+O\!\left(\tfrac{\log N}{\sqrt{N}}\right)
\end{align}
as $N\ti$, where $m=\lfloor sN\rfloor$ and $n=\lfloor tN\rfloor$.  
\end{lemma}
\begin{proof}
    Observe that
    \begin{align*}
      \bigg|&\mathbb{X}^{i,j}_N(s,t)-\mathbb{X}^{i,j}_N\!\left(\tfrac{m}{N},\tfrac{n}{N}\right) \bigg|\\
      &=\bigg|\int_s^t (X_N^i(r)-X_N^i(s))dX_N^j(r)-\int_{\frac{m}{N}}^{\frac{n}{N}} \left(X_N^i(r)-X_N^i\!\left(\tfrac{m}{N}\right)\right)dX_N^j(r)\bigg|\\
      &\leq \left|\int_{[s,t]\smallsetminus [\frac{m}{N},\frac{n}{N}]} X_N^i(r) dX_N^j(r)\right|+\left|X_N^i(s)\left(X_N^j(t)-X_N^j(s)\right)-X_N^i\!\left(\tfrac{m}{N}\right)\left(X_N^j\left(\tfrac{n}{N}\right)-X_N^j\!\left(\tfrac{m}{N}\right)\right)\right|\\
      &:=I_1+I_2,
    \end{align*}
where $$I_1=\left|\int_{[s,t]\smallsetminus [\frac{m}{N},\frac{n}{N}]} X_N^i(r) dX_N^j(r)\right|$$ and $$I_2=\left|X_N^i(s)\left(X_N^j(t)-X_N^j(s)\right)-X_N^i\left(\tfrac{m}{N}\right)\left(X_N^j\!\left(\tfrac{n}{N}\right)-X_N^j\!\left(\tfrac{m}{N}\right)\right)\right|.$$
To bound $I_1$, note that for all $r$  we have $|X_N^i(r)|\ll \log(N)$ for almost every $x$ by paucity. Also $|\dot X_N^j(r)|=\sqrt N$ for every $r$ which is not a grid point. Therefore
\begin{align*}
    I_1&\ll \sqrt N \log N\left(t-\frac{\lfloor tN\rfloor}{N}+s-\frac{\lfloor sN\rfloor}{N}\right)\\
    &=\frac{\sqrt N \log N}N\left(\{tN\}+\{sN\}\right)\\
    &\ll \frac{\log N}{\sqrt N}.
\end{align*}
To handle $I_2$, add and subtract $X_N^i\left(\frac{m}{N}\right)(X_N^j(t)-X_N^j(s))$ to get that
\begin{align*}
    I_2&\leq \left|\left(X_N^i(s)-X_N^i\left(\tfrac{m}{N}\right)\right)(X_N^j(t)-X_N^j(s))\right|+\left|X_N^i\!\left(\tfrac{m}{N}\right)\left((X_N^j(t)-X_N^j(s))-\left(X_N^j\!\left(\tfrac{n}{N}\right)-X_N^j\!\left(\tfrac{m}{N}\right)\right)\right)\right|\\
    &=\left(s-\frac{m}N\right)\left|X_N^j(t)-X_N^j(s)\right|+\left|X_N^i\left(\frac{m}{N}\right)\left((X_N^j(t)-X_N^j(s))-\left(X_N^j\left(\frac{n}{N}\right)-X_N^j\left(\frac{m}{N}\right)\right)\right)\right|\\
    &\ll \log N\left(s-\frac{m}N\right) +\log N\left(t-\frac{n}{N}\right)+\log N\left(s-\frac{m}{N}\right)
\end{align*} 
by paucity. Therefore, as $s-\frac{m}{N}\ll \frac{1}{N}$, we have proved claim \eqref{XX_N-at-times_m/N_n/N+O-term}.
\end{proof}

In view of Lemma \ref{lemma:grid-approximation-with-continuous-time}, it is enough to consider $\mathbb{X}_N^{i,j}\!\left(\tfrac{m}{N},\tfrac{n}{N}\right)$, where for sufficiently large $N$ we have  $m<n$. The following proposition allows us to write $\mathbb{X}_N^{i,j}$ at arbitrary grid points in terms of $X_N^i$ and $X_N^j$ at \emph{consecutive} grid points. 
\begin{proposition}
For $i,j\in \{1,2\}$ we have that
\begin{align}
    \mathbb{X}_N^{i,j}\!\left(\tfrac{m}{N},\tfrac{n}{N}\right)=
    &\sum_{m+1\leq k<\ell\leq n}\left(X_N^{i}\!\left(\tfrac{k}{N}\right)-X_N^i\!\left(\tfrac{k-1}{N}\right)\right)\left(X_N^j\!\left(\tfrac{\ell}{N}\right)-X_N^j\!\left(\tfrac{\ell-1}{N}\right)\right)+\nonumber\\ 
    &+\ha\sum_{m+1\leq k\leq n}\left(X_N^i\!\left(\tfrac{k}{N}\right)-X_N^i\!\left(\tfrac{k-1}{N}\right)\right)\left(X_N^j\!\left(\tfrac{k}{N}\right)-X_N^j\!\left(\tfrac{k-1}{N}\right)\right).\label{XX_N^ij-2}
\end{align}
\end{proposition}
\begin{proof}
    First, let us check what happens on consecutive grid points. 
\begin{align*}
    \mathbb X_N^{i,j}\left(\tfrac{k-1}{N},\tfrac{k}{N}\right)&=\int_{\frac{k-1}{N}}^{\frac{k}{N}} \left(X_N^i(r)-X_N^i\left(\tfrac{k-1}{N}\right)\right)dX_N^j(r)\\
    &=N \left(X_N^i\!\left(\tfrac{k}{N}\right)-X_N^i\!\left(\tfrac{k-1}{N}\right)\right)\left(X_N^j\!\left(\tfrac{k}{N}\right)-X_N^j\!\left(\tfrac{k-1}{N}\right)\right) \int_{\frac{k-1}{N}}^{\frac{k}{N}} \left(r-\tfrac{k-1}{N}\right) dr \\
    &=\frac{1}{2N}\left(X_N^i\!\left(\tfrac{k}{N}\right)-X_N^i\!\left(\tfrac{k-1}{N}\right)\right)\left(X_N^j\!\left(\tfrac{k}{N}\right)-X_N^j\!\left(\tfrac{k-1}{N}\right)\right).
\end{align*}
Using an extended version of Chen's relation (see Exercise 2.4 in \cite{Friz-Hairer-Book}) concludes the proof.
\end{proof}

Rewriting 
\eqref{XX_N^ij-2} using \eqref{shorthand-a_k}-\eqref{shorthand-b_k}, we have
\begin{align}
    \mathbb{X}_N^{1,1}\!\left(\tfrac{m}{N},\tfrac{n}{N}\right)&=\frac{1}{N}\sum_{m+1\leq k<\ell\leq n}a_k a_\ell+\frac{1}{2N}\sum_{m+1\leq k\leq n}a_k^2=:\mathcal{A}^{1,1}_N+\mathcal{B}^{1,1}_N
    \label{XX_N^11-1}\\
    \mathbb{X}_N^{1,2}\!\left(\tfrac{m}{N},\tfrac{n}{N}\right)&=\frac{1}{N}\sum_{m+1\leq k<\ell\leq n}a_k b_\ell+\frac{1}{2N}\sum_{m+1\leq k\leq n}a_k b_k=:\mathcal{A}^{1,2}_N+\mathcal{B}^{1,2}_N\label{XX_N^12-1}\\
    \mathbb{X}_N^{2,1}\!\left(\tfrac{m}{N},\tfrac{n}{N}\right)&=\frac{1}{N}\sum_{m+1\leq k<\ell\leq n}b_k a_\ell+\frac{1}{2N}\sum_{m+1\leq k\leq n}a_k b_k=:\mathcal{A}^{2,1}_N+\mathcal{B}^{1,2}_N\label{XX_N^21-1}\\
    \mathbb{X}_N^{2,2}\!\left(\tfrac{m}{N},\tfrac{n}{N}\right)&=\frac{1}{N}\sum_{m+1\leq k<\ell\leq n}b_k b_\ell+\frac{1}{2N}\sum_{m+1\leq k\leq n}b_k^2=:\mathcal{A}^{2,2}_N+\mathcal{B}^{2,2}_N,
    \label{XX_N^22-1}
\end{align}
where $\mathcal{A}^{i,j}_N=\mathcal{A}^{i,j}_N(x;\alpha,\beta;s,t)$ and $\mathcal{B}^{i,j}_N=\mathcal{B}^{i,j}_N(x;\alpha,\beta;s,t)$.
Let us now rewrite the sums $\mathcal{A}^{i,j}_N$.  Consider the two sums  $\mathcal{H}_N^{\pm}=\mathcal{H}_N^{\pm}(x;\alpha,\beta;s,t)
$ defined as
\begin{align}\label{H^pm}
    \mathcal{H}_N^{\pm}=\frac{1}{N}\sum_{m+1\leq k<\ell\leq n} (z_k\pm\overline z_k)z_\ell=:\mathcal{I}_N\pm\mathcal{J}_N,
\end{align}
where $\mathcal{I}_N=\mathcal{I}_N(x;\alpha,\beta;s,t)$ and $\mathcal{J}_N=\mathcal{J}_N(x;\alpha,\beta;s,t)$.The limiting distribution, as $n\to\infty$, for the double sum $\mathcal{J}_N=\frac{1}{N}\sum_{m+1\leq k<\ell\leq n}\overline{z_k}z_\ell$ (note the ``triangular'' set of indices) will be crucial in our analysis. Note that  
$\mathcal{H}_N^{+}=\frac{2}{N}\sum_{m+1\leq k<\ell\leq n}a_k z_\ell$ and $\mathcal{H}_N^{-}=\frac{2i}{N}\sum_{m+1\leq k<\ell\leq n}b_k z_\ell$ and therefore we can write
\begin{align}\label{rewriting_A^ij_in_terms_of_H^pm}
    \mathcal{A}^{1,1}_N=\Re(\tha\mathcal{H_N^+}),\hspace{.5cm}\mathcal{A}^{1,2}_N=\Im(\tha\mathcal{H_N^+}),\hspace{.5cm}\mathcal{A}^{2,1}_N=\Re(\tfrac{1}{2i}\mathcal{H_N^-}),\hspace{.5cm}\mathcal{A}^{2,2}_N=\Im(\tfrac{1}{2i}\mathcal{H_N^-}).
\end{align}
Finally,  consider the sums $\mathcal{L}_N=\mathcal{L}_N(x;\alpha,\beta;s,t)$ and 
$\mathcal{M}_N=\mathcal{M}_N(x;\alpha,\beta;s,t)$ defined as
\begin{align}
    \mathcal{L}_N&=
    \frac{1}{\sqrt{N}}\sum_{m+1\leq k\leq n}z_k
    =X_N(t)-X_N(s),\label{def-L_N}\\
    \mathcal{M}_N&=\frac{1}{N}\sum_{m\leq k\leq n}z_k^2,
\end{align}
so that 
\begin{align}\label{L=2I+M}
    \mathcal{L}_N^2=2\mathcal{I}_N +\mathcal{M}_N.
\end{align}
Therefore, combining \eqref{XX_N^11-1}--\eqref{L=2I+M} we can write
\begin{align}
    \mathbb{X}_N^{1,1}\!\left(\tfrac{m}{N},\tfrac{n}{N}\right)&=\Re(\tfrac{1}{4}\mathcal{L}_N^2-\tfrac{1}{4}\mathcal{M_N}+\tha\mathcal{J}_N)+\mathcal{B}^{1,1}_N\label{rewriting_X_N^11},\\
    \mathbb{X}_N^{1,2}\!\left(\tfrac{m}{N},\tfrac{n}{N}\right)&=\Im(\tfrac{1}{4}\mathcal{L}_N^2-\tfrac{1}{4}\mathcal{M_N}+\tha\mathcal{J}_N)+\mathcal{B}^{1,2}_N\label{rewriting_X_N^12},\\
    \mathbb{X}_N^{2,1}\!\left(\tfrac{m}{N},\tfrac{n}{N}\right)&=\Re(\tfrac{1}{4i}\mathcal{L}_N^2-\tfrac{1}{4i}\mathcal{M}_N-\tfrac{1}{2i}\mathcal{J}_N)+\mathcal{B}^{1,2}_N\label{rewriting_X_N^21},\\
    \mathbb{X}_N^{2,2}\!\left(\tfrac{m}{N},\tfrac{n}{N}\right)&=\Im(\tfrac{1}{4i}\mathcal{L}_N^2-\tfrac{1}{4i}\mathcal{M}_N-\tfrac{1}{2i}\mathcal{J}_N)+\mathcal{B}^{2,2}_N\label{rewriting_X_N^22}.
\end{align}


\begin{lemma}\label{lem-B_N^ij_and_M_N-tend-to-0-in-probability} Let $\lambda$ be a probability measure on $\R$ which is absolutely continuous with respect to the Lebesgue measure. Let $x$ be randomly distributed according to $\lambda$. Let $(\alpha,\beta)\notin\Q^2$ and let $s<t$. Then the four sum $\mathcal{M}_N(x;\alpha,\beta;s,t)$, $\mathcal{B}_N^{1,1}(x;\alpha,\beta;s,t)$, $\mathcal{B}_N^{1,2}(x;\alpha,\beta;s,t)$, $\mathcal{B}_N^{2,2}(x;\alpha,\beta;s,t)$ all tend to zero in probability as $N\ti$. That is, for every $\varepsilon>0$,
\begin{align}
&\lim_{N\to\infty}\lambda\left\{x\in\R:\: \left|\frac{1}{N}\sum_{m+1\leq k\leq n}z_k(x;\alpha,\beta)^2\right|>\varepsilon\right\}=0,\label{statement-sum-M-goes-to-0-in-prob}\\
    &\lim_{N\to\infty}\lambda\left\{x\in\R:\: \left|\frac{1}{N}\sum_{m+1\leq k\leq n}a_k(x;\alpha,\beta)^2\right|>\varepsilon\right\}=0,\label{statement-sum-B^11-goes-to-0-in-prob}\\
    &\lim_{N\to\infty}\lambda\left\{x\in\R:\: \left|\frac{1}{N}\sum_{m+1\leq k\leq n}a_k(x;\alpha,\beta) b_k(x;\alpha,\beta)\right|>\varepsilon\right\}=0,\label{statement-sum-B^12-goes-to-0-in-prob}\\
    &\lim_{N\to\infty}\lambda\left\{x\in\R:\: \left|\frac{1}{N}\sum_{m+1\leq k\leq n} b_k(x;\alpha,\beta)^2\right|>\varepsilon\right\}=0.\label{statement-sum-B^22-goes-to-0-in-prob}
\end{align}
In addition we have that 
\begin{align}
&\lim_{N\to\infty}\lambda\left\{x\in\R:\: \sup_{0\leq m\neq n\leq N}\left|\frac{1}{N}\sum_{m+1\leq k\leq n}z_k(x;\alpha,\beta)^2\right|>\varepsilon\right\}=0,\label{statement-sum-M-goes-to-0-in-prob-sup}\\
    &\lim_{N\to\infty}\lambda\left\{x\in\R:\: \sup_{0\leq m\neq n\leq N}\left|\frac{1}{N}\sum_{m+1\leq k\leq n}a_k(x;\alpha,\beta)^2\right|>\varepsilon\right\}=0,\label{statement-sum-B^11-goes-to-0-in-prob-sup}\\
    &\lim_{N\to\infty}\lambda\left\{x\in\R:\: \sup_{0\leq m\neq n\leq N}\left|\frac{1}{N}\sum_{m+1\leq k\leq n}a_k(x;\alpha,\beta) b_k(x;\alpha,\beta)\right|>\varepsilon\right\}=0,\label{statement-sum-B^12-goes-to-0-in-prob-sup}\\
    &\lim_{N\to\infty}\lambda\left\{x\in\R:\: \sup_{0\leq m\neq n\leq N}\left|\frac{1}{N}\sum_{m+1\leq k\leq n} b_k(x;\alpha,\beta)^2\right|>\varepsilon\right\}=0.\label{statement-sum-B^22-goes-to-0-in-prob-sup}
\end{align}
\end{lemma}
\begin{proof} We write
\begin{align}
    \frac{1}{N}\sum_{m+1\leq k\leq n}z_k(x;\alpha,\beta)^2=\frac{1}{N}\sum_{m+1\leq k\leq n}z_k(2x;2\alpha,\beta).\label{pf-sum-M-goes-to-0-in-prob-1}
\end{align}
Similarly, using the identities $\cos^2(\theta)=\tha+\tha\cos(2\theta)$,  $\cos(\theta)\sin(\theta)=\frac{\sin(2\theta)}{2}$,  and $\sin^2(\theta)=\tha-\tha\cos(2\theta)$, we get
\begin{align}
    &\frac{1}{N}\sum_{m+1\leq k\leq n}a_k(x;\alpha,\beta)^2=\frac{1}{2N}+\frac{1}{2N}\sum_{m+1\leq k\leq n}a_k(2x;2\alpha,\beta),\label{pf-sum-B^11-goes-to-0-in-prob-1}\\
    &\frac{1}{N}\sum_{m+1\leq k\leq n}a_k(x;\alpha,\beta) b_k(x;\alpha,\beta)= \frac{1}{2N}\sum_{m+1\leq k\leq n}b_k(2x;2\alpha,\beta),\label{pf-sum-B^12-goes-to-0-in-prob-1}\\
    &\frac{1}{N}\sum_{m+1\leq k\leq n}b_k(x;\alpha,\beta)^2 =\frac{1}{2N}-\frac{1}{2N}\sum_{m+1\leq k\leq n}a_k(2x;2\alpha,\beta).\label{pf-sum-B^22-goes-to-0-in-prob-1}
\end{align}
    Note that 
\begin{align}
&\frac{1}{\sqrt{N}}\sum_{m+1\leq k\leq n}z_k(2x;2\alpha,\beta)=X_N(2x;2\alpha,\beta;t)-X_N(2x;2\alpha,\beta;s)\label{pf-sum-M-goes-to-0-in-prob-2},\\
&\frac{1}{\sqrt{N}}\sum_{m+1\leq k\leq n}a_k(2x;2\alpha,\beta)=X_N^1(2x;2\alpha,\beta;t)-X_N^1(2x;2\alpha,\beta;s)\label{pf-sum-B^11-goes-to-0-in-prob-2},\\
&\frac{1}{\sqrt{N}}\sum_{m+1\leq k\leq n}b_k(2x;2\alpha,\beta)=X_N^2(2x;2\alpha,\beta;t)-X_N^2(2x;2\alpha,\beta;s)\label{pf-sum-B^12-goes-to-0-in-prob-2}.
\end{align}
Let $h$ denote the density of $\lambda$. Then, since $x$ is distributed according to $\lambda$, then $\tilde x=2x$ is distributed according to the probability measure $\tilde\lambda$, with density $\tilde h(\tilde x)=\ha h(\tilde x/2)$. Theorem \ref{thm-existence-of-theta-process} applied to $\tilde\lambda$ and $(2\alpha,\beta)$ implies that the sums in \eqref{pf-sum-M-goes-to-0-in-prob-2}, \eqref{pf-sum-B^11-goes-to-0-in-prob-2}, and \eqref{pf-sum-B^12-goes-to-0-in-prob-2} tend, in distribution, to $X(t)-X(s)$, $\Re(X(t)-X(s))$, and $\Im(X(t)-X(s))$ respectively,  as $N\to\infty$. Therefore, due to  overnormalization, the sums in \eqref{pf-sum-M-goes-to-0-in-prob-1}--
\eqref{pf-sum-B^22-goes-to-0-in-prob-1} tends to zero in distribution. Convergence to zero in distribution then implies convergence to zero in probability. Similarly, equations \eqref{statement-sum-B^11-goes-to-0-in-prob-sup}-\eqref{statement-sum-B^22-goes-to-0-in-prob-sup} follow from Proposition \ref{prop:tail-bound-Holder} and the overnormalization. 
\end{proof}

\begin{remark}\label{rk-what-to-with-L_N-and-J_N}
    Taking into account \eqref{XX_N-at-times_m/N_n/N+O-term}, \eqref{rewriting_X_N^11}--\eqref{rewriting_X_N^22}, and Lemma \ref{lem-B_N^ij_and_M_N-tend-to-0-in-probability}, in order to show that $\mathbf{X}_N(s,t)=(X_N(t)-X_N(s),\mathbb{X}_N(s,t))$ has a limit in distribution on $\mathscr C_g^{\gamma}$ for $\gamma<\frac{1}{2}$ as $N\ti$, it is enough to show that the random variables $\mathcal{J}_N(x;\alpha,\beta;s,t)$ and $\mathcal{L}_N(x;\alpha,\beta;s,t)$,  defined in \eqref{H^pm} and \eqref{def-L_N}, have a \emph{joint} limiting distribution as $N\ti$ as functions in $C^{2\gamma}(\Delta^{(0,T)}_2,\mathbb R^2\oplus \mathbb R^{2\times 2})$.  To this end, we will write  both $\mathcal{J}_N(x;\alpha,\beta,s,t)$ and $\mathcal{L}_N(x;\alpha,\beta,s,t)$ in terms of a pair of horocycle lifts in $\GamG\times\GamG$ and prove that such pair becomes equidistributed as $N\ti$ according to a measure $\hat\mu_{\GamG\times\GamG}$. 
\end{remark}
    The strategy outlined in the previous remark will actually allow to do more: we can fix $k\geq1$ and fix $\underline s=(s_1,\ldots,s_k)$, $\underline{t}=(t_1,\ldots, t_k)$ with $s_\ell<t_\ell$ for $1\leq \ell\leq k$ and show that the $k$-tuple of $\R^2\oplus\R^{2\times 2}$-valued random variables $(\mathbf{X}_N(x;\alpha,\beta;s_\ell,t_\ell))_{1\leq \ell\leq k}$ (where $x$ is $\lambda$-random) has a limiting distribution as $N\ti$. This will give the convergence of finite-dimensional distributions of the to those of a certain random $\R^2\times\R^{2\times2}$-valued function $\mathbb{X}$ on the simplex.

The limiting distribution of $\mathcal{L}_N$ alone comes directly from from Theorem \ref{thm-existence-of-theta-process}. We have that, for $x$ randomly distributed according to $\lambda$ and $(\alpha,\beta)\notin\Q^2$, the random variables  $\mathcal{L}_N(x;\alpha,\beta;s,t)$ converge in distribution to $X(t)-X(s)$ as $N\ti$. 
One key step in the proof of theorem \ref{thm-existence-of-theta-process} given  \cite{Cellarosi-Marklof} is to  rewrite $X_N(x;t)$ as a $t$-dependent function of a (single) horocycle lift as
\begin{align}\label{X_N(t)-as-Theta_1_(0,t]}  X_N(x;\alpha,\beta;t)=\Theta_{\mathbf{1}_{(0,t]}}\!\left((I;\sve{\alpha+\beta x}{0},0)\Psi^x\Phi^{2\log N}\right).
\end{align}
Therefore, we can write
\begin{align}\label{L_N(t)-as-Theta_1_(s,t]}
\mathcal{L}_N(x;\alpha,\beta;s,t)=\Theta_{\mathbf{1}_{(s,t]}}\!\left((I;\sve{\alpha+\beta x}{0},0)\Psi^x\Phi^{2\log N}\right). 
\end{align}
Note that in \eqref{X_N(t)-as-Theta_1_(0,t]}-\eqref{L_N(t)-as-Theta_1_(s,t]} we are deliberately ignoring terms that go to zero in probability as $N\to\infty$, uniformly in $s,t$ (such as the ones coming from the linear interpolation).
Let us focus on the sum $\mathcal{J}_N$. We have, by \eqref{def-big-Theta}, 
\begin{align}
    \mathcal{J}_N(x;\alpha,\beta;s,t)
    &=\frac{1}{N}\sum_{\lfloor s N\rfloor+1\leq k<\ell\leq \lfloor tN\rfloor}\e{-(\tha k^2+\beta k)x -k\alpha}\e{(\tha \ell^2+\beta\ell)x+\ell \alpha}\nonumber\\
    &={\Theta}_{T_{(s,t]}}^{(2)}\left(\left(-x+\tfrac{i}{N^2},0;\sve{-\alpha-\beta x}{0},0\right),\left(x+\tfrac{i}{N^2},0;\sve{\alpha+\beta x}{0},0\right)\right),\nonumber\\
    &={\Theta}_{T_{(s,t]}}^{(2)}\left( \left(I;\sve{-\alpha-\beta x}{0},0\right)\Psi^{-x}\Phi^{2\log N},\left(I;\sve{\alpha+\beta x}{0},0\right)\Psi^x\Phi^{2\log N}\right)
    \label{writing-J_N-using-horocycles}
\end{align}
where $T_{(s,t]}=\mathbf{1}_{\mathfrak{T}_{(s,t]}}$, and $\mathfrak{T}_{(s,t]}=\{(x,y)\in\R^2:\:s<x<y\leq t\}$ is the triangular region in the plane with vertices $(s,s),(s,t),(t,t)$. 
Note that only part of the perimeter of the triangle is included in this region:  the horizontal  segment $\{(x,t)\in\R^2:\:s<x<t\}$ is part of $\mathfrak{T}_{(s,t]}$, while  the vertical edge, the diagonal edge, and all three vertices of the triangle are not.
We will also write
\begin{align}\label{T_open-triangle-s,t}
    T_{(s,t)}=\mathbf{1}_{\mathfrak{T}_{(s,t)}},
\end{align}
where $\mathfrak{T}_{(s,t)}=\operatorname{int}(\mathfrak T_{(s,t]})=\{(x,y)\in\R^2:\:s<x<y< t\}$.
We will establish tightness for $\mathcal{J}_N$ in H\"{o}lder spaces, as well as the existence of its limiting finite-dimensional distributions. In both of these tasks, we may replace $\Theta^{(2)}_{T_{(s,t]}}$ with $\Theta^{(2)}_{T_{(s,t)}}$ since the function $T_{(s,t]}-T_{(s,t)}=\mathbf{1}_{(s,t)\times\{t\}}$ is supported on a line segment, see Lemma \ref{lemma:lines-dont-matter}.

\section{\texorpdfstring{Decomposing $\Theta^{(2)}_{{T}_{(0,1)}}$ }{}}\label{section:defining-theta-T}
In this section, we decompose $\Theta^{(2)}_{{T}_{(0,1)}}$ in a similar way as the decomposition of $\Theta_\chi$, using the linearity of $F\mapsto \Theta^{(2)}_F$. However, we decompose the triangle into seven pieces instead of two --- one for each corner, one for each edge and, a smooth remainder.

\subsection{General geometric decomposition of indicator functions}\label{subsection-dyadic-decompositions}
Let $f_0:\mathbb R\to \mathbb R$ be a smooth nonnegative function so that $f_0(x)=0$ for $x\leq 0$, $f_0(x)=1$ for $x\geq 1$, and so that 
\begin{equation}\label{eq:f_0-interval-def}
    f_0(x)+f_0(1-x)=1
\end{equation}
for $x\in (0,1)$. Define for $0<c_1<c_2<c_3<1$ with $p:=\frac{c_3-c_2}{c_2-c_1}>1$ the function
\begin{equation}\label{fc1c2c3}
    f_{c_1,c_2,c_3}(x):=
    \begin{cases}
        0&\text{ if } x\leq c_1\\
        f_0\left(\frac{1}{c_2-c_1}x-\frac{c_1}{c_2-c_1}\right)&\text{ if } c_1 \leq x\leq c_2\\
        f_0\left(\frac{c_3}{c_3-c_2}-\frac{1}{c_3-c_2}x\right)&\text{ if }c_2\leq x\leq c_3\\
        0&\text{ if }x\geq c_3,
    \end{cases}
\end{equation}
For instance, the function $\Delta$ from Section \ref{section-Theta_chi} can be written as $f_{\frac{1}{6},\frac{1}{3},\frac{2}{3}}$. We also define
the function
\begin{equation}\label{Fc1c2c3}
  F_{c_1,c_2,c_3}(x):=  \sum_{j=0}^\infty f_{c_1,c_2,c_3}\left(\left(\tfrac{c_3-c_2}{c_2-c_1}\right)^jx\right)=\sum_{j=0}^\infty f_{c_1,c_2,c_3}\left(p^jx\right).
\end{equation}
 See Figure \ref{fig:fF1121613} for one such function which we will use later.
\begin{figure}[h]
    \centering
    \hspace{-.8cm}
\includegraphics[width=15cm]{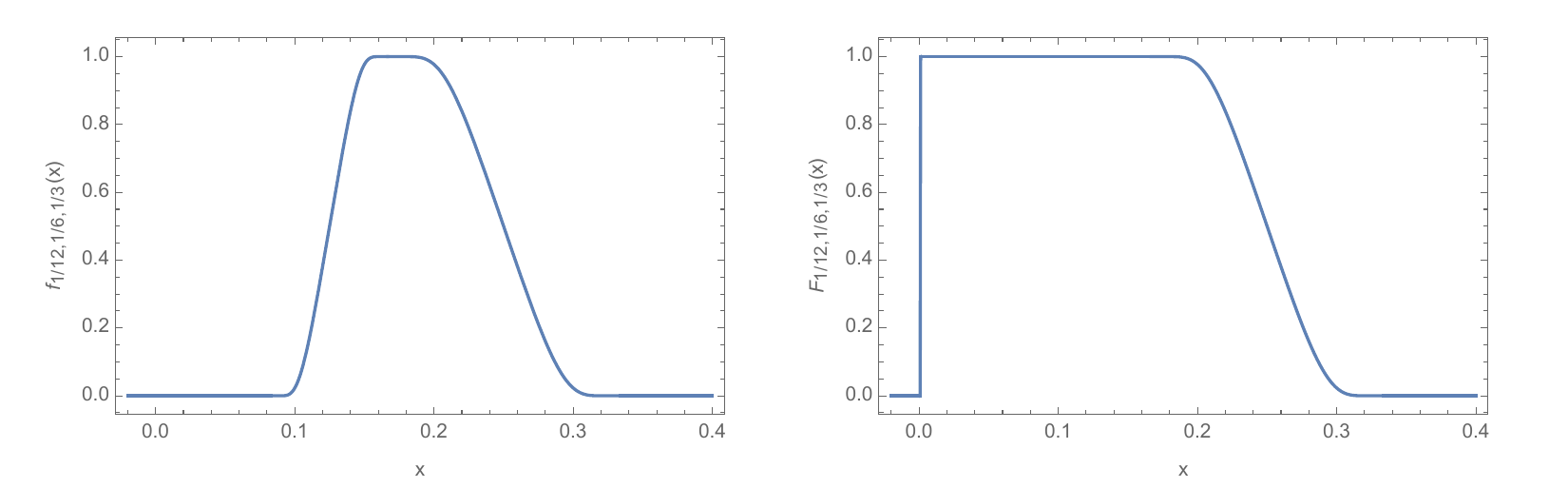}
    \caption{The functions $x\mapsto f_{\frac{1}{12},\frac{1}{6},\frac{1}{3}}(x)$ and $x\mapsto F_{\frac{1}{12},\frac{1}{6},\frac{1}{3}}(x)$  as in \eqref{fc1c2c3} and \eqref{Fc1c2c3}, obtained from the function 
$f_0(x)=\frac{h(x)}{h(x)+h(1-x)}$, where $h(x)=0$ for $x\leq 0$ and $h(x)=e^{-1/x}$ for $x>0$. These functions will be used Section \ref{subsection:Triangle}. In this case $p=2$.}
    \label{fig:fF1121613}
\end{figure}
Note that for all $j\in \mathbb N$ we have
\begin{equation*}
    f_{c_1,c_2,c_3}(p^j x):=
    \begin{cases}
        0&\text{ if } x\leq p^{-j}c_1\\
        f_0\left(\frac{1}{c_2-c_1}xp^j-\frac{c_1}{c_2-c_1}\right)&\text{ if } c_1p^{-j} \leq x\leq c_2p^{-j}\\
        f_0\left(\frac{c_3}{c_3-c_2}-\frac{1}{c_3-c_2}xp^j\right)&\text{ if }c_2p^{-j}\leq x\leq p^{-j} c_3\\
        0&\text{ if }x\geq c_3p^{-j},
    \end{cases}
\end{equation*}
Let $x\in (0,c_2)$. Then there is some $j_x\geq 0$ so that $x\in (c_2 p^{-j_x-1},c_2 p^{-j_x})$. Therefore 
\begin{equation*}
     f_{c_1,c_2,c_3}(p^j x)=
     \begin{cases}
         f_0\left(\frac{1}{c_2-c_1}xp^j-\frac{c_1}{c_2-c_1}\right)&\text{ if }j=j_x\\
         f_0\left(\frac{c_3}{c_3-c_2}-\frac{1}{c_3-c_2}xp^j\right)&\text{ if }j=j_x+1\\
         0&\text{ else. }
     \end{cases}
\end{equation*}
Therefore
\begin{align*}
        &F_{c_1,c_2,c_3}(x)=f_{c_1,c_2,c_3}(p^{j_x} x)+f_{c_1,c_2,c_3}(p^{j_x+1} x)\\
        &=f_0\left(\tfrac{1}{c_2-c_1}xp^{j_x}-\tfrac{c_1}{c_2-c_1}\right)+f_0\left(\tfrac{c_3}{c_3-c_2}-\tfrac{1}{c_3-c_2}xp^{j_x+1}\right)\\
        &=f_0\left(\tfrac{1}{c_2-c_1}xp^{j_x}-\tfrac{c_1}{c_2-c_1}\right)+f_0\left(\tfrac{c_3}{c_3-c_2}-\tfrac{1}{c_3-c_2}xp^{j_x}\tfrac{c_3-c_2}{c_2-c_1}\right)\\
        &=f_0\left(\tfrac{1}{c_2-c_1}xp^{j_x}-\tfrac{c_1}{c_2-c_1}\right)+f_0\left(\tfrac{c_3}{c_3-c_2}-\frac{1}{c_2-c_1}xp^{j_x}\right)\\
        &=f_0\left(\tfrac{1}{c_2-c_1}xp^{j_x}-\tfrac{c_1}{c_2-c_1}\right)+f_0\left(1-\left(\tfrac{1}{c_2-c_1}xp^{j_x}-\tfrac{c_1}{c_2-c_1}\right)\right)\\
        &=1,
\end{align*}
where in the last line we used relation \eqref{eq:f_0-interval-def}. If $x\in (c_2,c_3)$ we have that $f_{c_1,c_2,c_3}(p^j x)=0$ if $j>0$ but $f_{c_1,c_2,c_3}(x)\neq 0$. If $x\geq c_3$ then $f_{c_1,c_2,c_3}(p^jx)=0$ for all $j$. In summary we have
\begin{equation*}
    F_{c_1,c_2,c_3}(x)=
    \begin{cases}
        0&\text{ if }x\not\in (0,c_3)\\
        1&\text{ if }x\in (0,c_2)\\
        f_{c_1,c_2,c_3}\left(x\right)&\text{ if }x\in (c_2,c_3).
    \end{cases}
\end{equation*}
Functions of the form \eqref{Fc1c2c3} only have a jump discontinuity at $0$ and are building blocks for indicators of intervals since $F_{c_1,c_2,c_3}(x)+F_{c_1,c_2,c_3}(c_2+c_3-x)=\mathbf{1}_{(0,c_2+c_3)}(x)$. For instance, in \eqref{partition-of-unity}, we have written the indicator of $(0,1)$ as the sum of two series, namely $\chi_L(x)=F_{\frac{1}{6},\frac{1}{3},\frac{2}{3}}(x)$ and $\chi_R(x)=F_{\frac{1}{6},\frac{1}{3},\frac{2}{3}}(1-x)$.

\subsection{Triangle}\label{subsection:Triangle}
We will write the indicator of the triangle as a sum of $3$ ``corner'' functions (each of which will be written as the affine image of a product of two infinite series),  $3$ ``edge'' functions (each of which will be written as the affine image of an infinite series times a smooth function), and a smooth one. First, using the notation from Section \ref{subsection-dyadic-decompositions},  we define a template corner function as
\begin{equation}\label{def:TCor}
    {T}_{\operatorname{Cor}}(x,y):=F_{\frac{1}{12}, \frac{1}{6}, \frac{1}{3}}(x)\:F_{\frac{1}{12}, \frac{1}{6}, \frac{1}{3}}(y).
\end{equation}

We note that
\begin{equation*}
     {T}_{\operatorname{Cor}}(x,y)=
     \begin{cases}
         0&\text{ if }(x,y)\not \in \left(0,\frac{1}{3}\right)^2\\
         1&\text{ if }(x,y)\in \left(0,\frac{1}{6}\right)^2\\
         \text{smooth}&\text{ else, }
     \end{cases}
\end{equation*}
see Figure \ref{fig:template-functions} (left panels).
We also define the template line function as 
\begin{equation}\label{def:TLine}
    {T}_{\operatorname{Line}}(x,y):=\left(F_{\frac{1}{12}, \frac{1}{6}, \frac{1}{3}}(\tfrac{1}{2}-x)+\mathbf{1}_{\{\frac{1}{2}\}}(x)+F_{\frac{1}{12}, \frac{1}{6}, \frac{1}{3}}(-\tfrac{1}{2}+x)\right)F_{\frac{1}{24}, \frac{1}{12}, \frac{1}{6}}(y).
\end{equation}
We note that for $0<y<\frac{1}{12}$ we have that
\begin{equation*}
      {T}_{\operatorname{Line}}(x,y)=
      \begin{cases}
          0&\text{ for }{x\leq \frac{1}{6}}\\
          f_{\frac{1}{12}, \frac{1}{6}, \frac{1}{3}}(\frac{1}{2}-x)&\text{ for }{\frac{1}{6}\leq x\leq \frac{1}{3}}\\
          1&\text{ for } \frac{1}{3}\leq x\leq \frac{2}{3}\\
          f_{\frac{1}{12}, \frac{1}{6}, \frac{1}{3}}(x-\frac{1}{2})&\text{ for }{\frac{2}{3}\leq x\leq \frac{5}{6}}\\
          0&\text{ for }{x\geq \frac{5}{6}}.
      \end{cases}
\end{equation*}
Also, if $y>\frac{1}{6}$ or $y<0$ then $ {T}_{\operatorname{Line}}(x,y)=0$, see Figure \ref{fig:template-functions} (right panels). 

\begin{figure}[h!]
\hspace{-.8cm}\includegraphics[width=13cm]{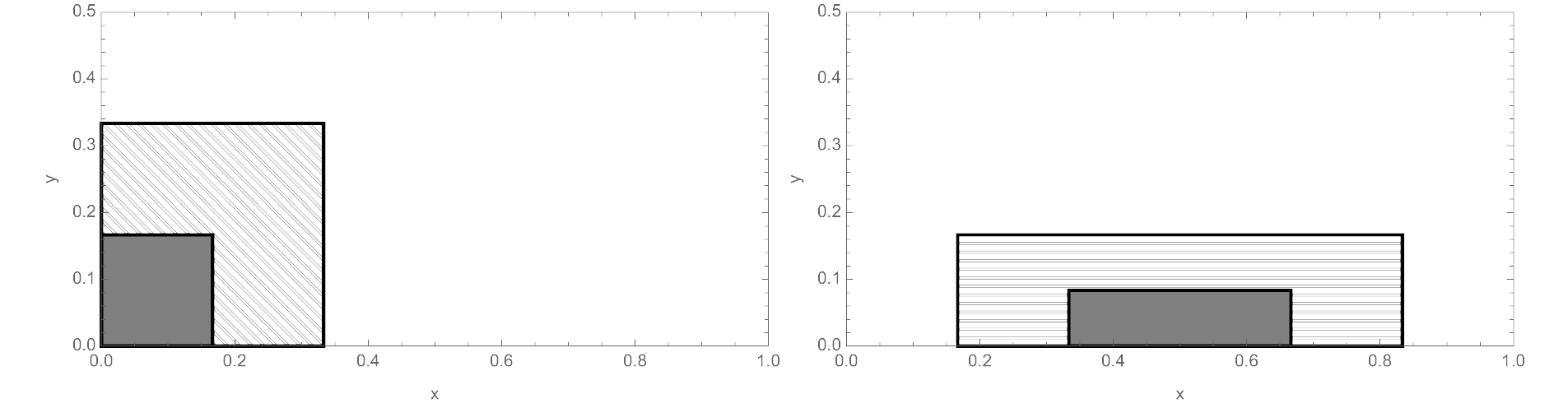}
\includegraphics[width=6cm]{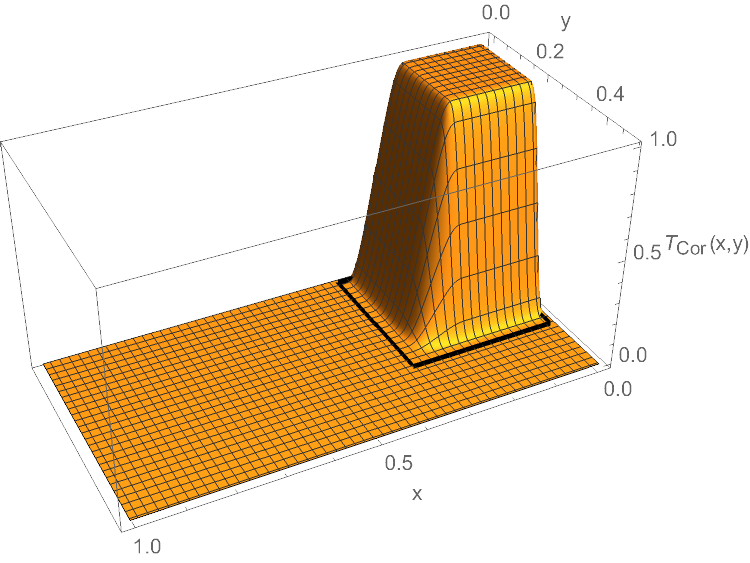}
\includegraphics[width=6cm]{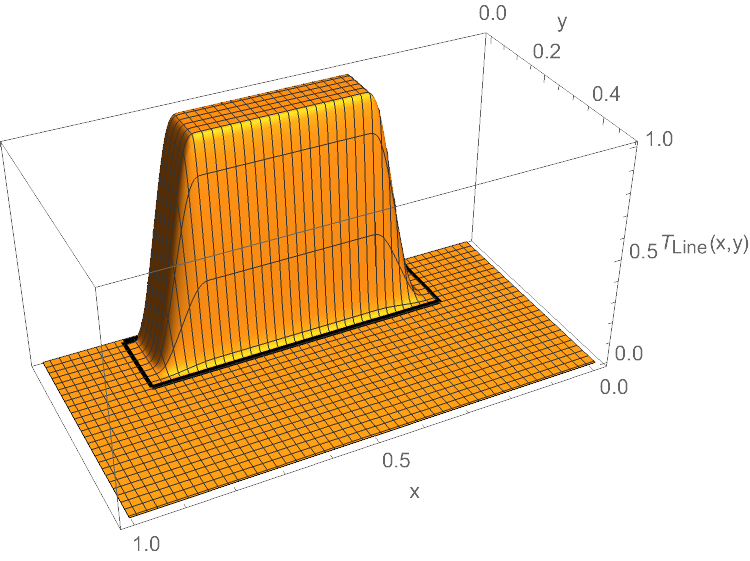}
    \caption{Top panels: The coloured and shaded area represent the support of the template corner function \eqref{def:TCor} (left) and of the template line function \eqref{def:TLine} (right). 
     In both cases, the solid colour represents the region where the function is constant equal to 1.
    Bottom panels: The graphs of  $(x,y)\mapsto T_{\operatorname{Cor}}(x,y)$ and of $(x,y)\mapsto T_{\operatorname{Line}}(x,y)$. The corner template  function has jump discontinuities along $\{(x,0):\: 0\leq x< \frac{1}{4}\}\cup\{(0,y):\: 0\leq y<
 \frac{1}{4}\}$ and is smooth otherwise. The line template  function has jump discontinuities along the line segment $\{(x,0):\: \frac{1}{4}< x< \frac{3}{4}\}$ and is smooth otherwise. The boundary of the support of each function from the top panels is also indicated in the bottom panels.} 
    \label{fig:template-functions}
\end{figure}

We now map the template corner function \eqref{def:TCor} to each of the corners of the open triangle $\operatorname{int}( \mathfrak{T}_{(0,1)})$. At the right-angle vertex $(0,1)$ we consider 
$$(w_1,w_2)\mapsto \mathfrak{C}_{0,1}(w_1,w_2):=T_{\operatorname{Cor}}(w_1,1-w_2),$$
whose support is shown in blue in Figure \ref{fig:corners-and-sides-diagrams}. At the acute-angle vertices $(0,0)$ and $(1,1)$ we define
\begin{align*}
       (w_1,w_2)&\mapsto\mathfrak{C}_{0,0}(w_1,w_2):= T_{\operatorname{Cor}}(w_1,w_2-w_1),\\
    (w_1,w_2)&\mapsto \mathfrak{C}_{1,1}(w_1,w_2):=T_{\operatorname{Cor}}(1-w_2,w_2-w_1),
\end{align*}
respectively. Their supports are shown in cyan and green, respectively, in Figure \ref{fig:corners-and-sides-diagrams}. Observe that the supports of the three corner functions have trivial pairwise intersections.
We also map the template line function \eqref{def:TLine} to the three edges of the triangle. To this end, we consider the functions
\begin{align*}
    (w_1,w_2)&\mapsto \mathfrak{L}_{\operatorname{h}}(w_1,w_2):=T_{\operatorname{Line}}(w_1,1-w_2)\\
    (w_1,w_2)&\mapsto \mathfrak{L}_{\operatorname{v}}(w_1,w_2):= T_{\operatorname{Line}}(w_2,w_1)\\
    (w_1,w_2)&\mapsto \mathfrak{L}_{\operatorname{d}}(w_1,w_2):=T_{\operatorname{Line}}(w_1,w_2-w_1),
\end{align*}
supported along the horizontal, vertical, and diagonal edge, respectively. Their supports are shown, respectively, in red, magenta, and yellow in  Figure \ref{fig:corners-and-sides-diagrams}.  Observe that the only nontrivial pairwise intersection of the supports of these line functions is comes from $\mathfrak{L}_{\operatorname{h}}$ and $\mathfrak{L}_{\operatorname{d}}$.
By construction, if we add the three corner functions and the three line functions, we obtain a function that is identically $0$ in the complement of  $\mathfrak{T}_{(0,1)}$ and, crucially, limits to $1$ as we approach the boundary of the triangle from its interior. In other words, the function 
\begin{align}\label{sum-6-functions}
\mathfrak{C}_{0,0}+\mathfrak{C}_{0,1}+\mathfrak{C}_{1,1}+ \mathfrak{L}_h+\mathfrak{L}_v+\mathfrak{L}_d
\end{align} is only discontinuous along the perimeter of the triangle, where it jumps from $0$ to $1$.
\begin{figure}
    \centering
    \includegraphics[width=10cm]{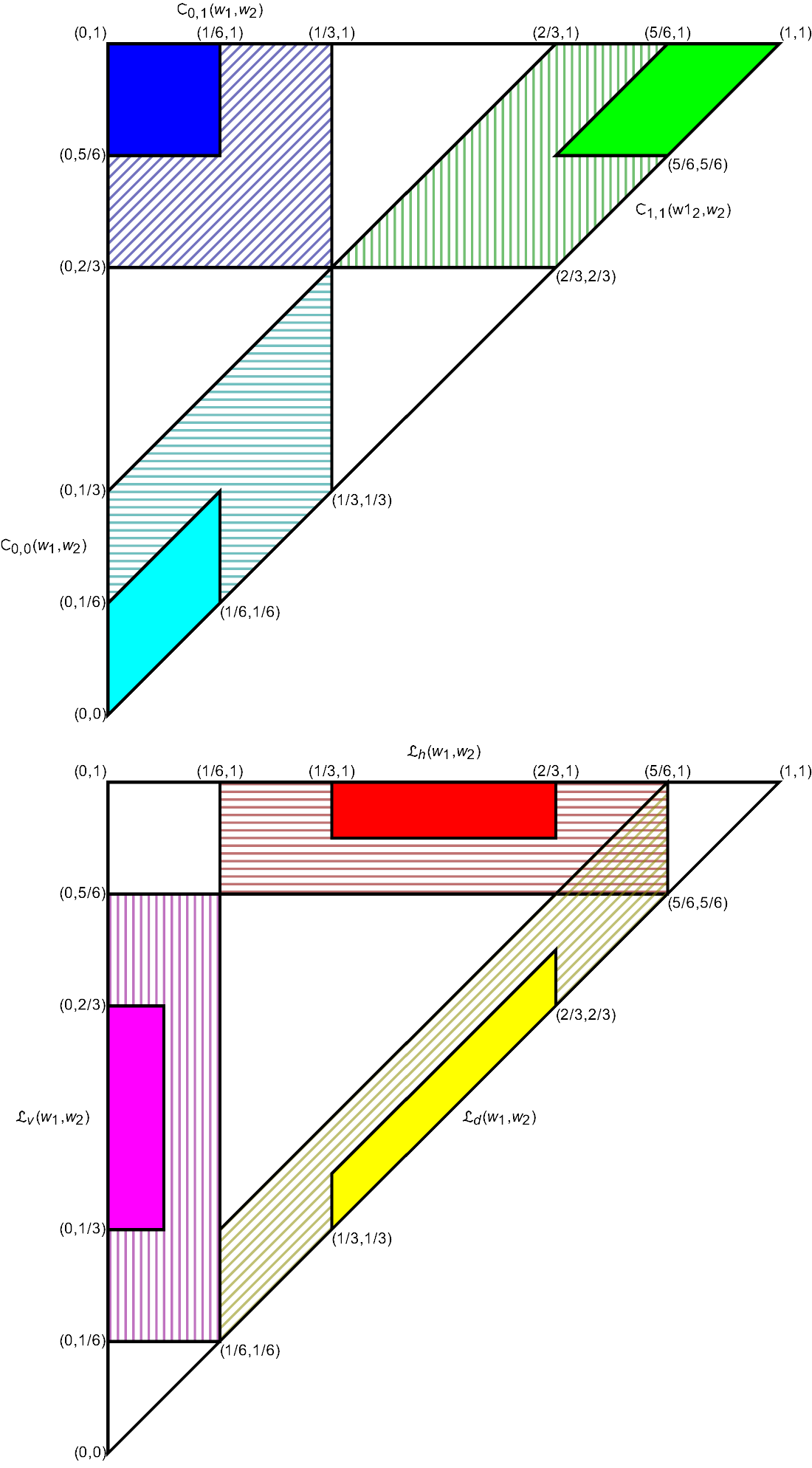}
    \caption{Top panel: the supports of three corner functions $\mathfrak{C}_{0,1}, \mathfrak{C}_{0,0}, \mathfrak{C}_{1,1}$. Bottom panel: the supports of the three edge functions $\mathfrak{L}_h, \mathfrak{L}_v, \mathfrak{L}_d$. In both panels, the solid colors indicate the regions where the corresponding functions are identically $1$}
    \label{fig:corners-and-sides-diagrams}
\end{figure}
We illustrate the way the various corner and line functions ``fit together'' in Figures \ref{fig:fit-at-the-corners-1}-\ref{fig:fit-at-the-corners-2}.
\begin{figure}
    \begin{center}
    \includegraphics[width=8cm]{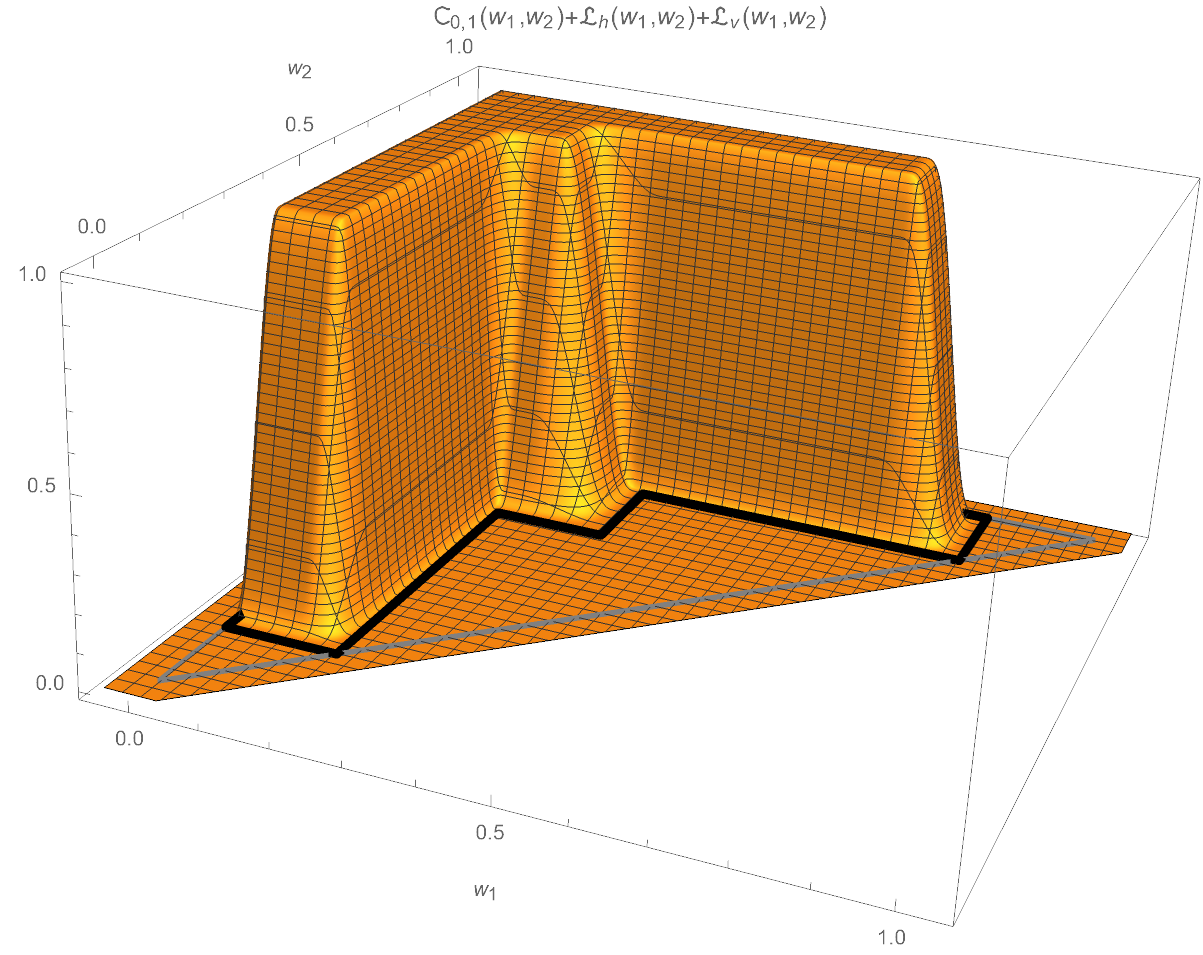}
    \end{center}
    \caption{The sum $\mathfrak{C}_{0,1}+\mathfrak{L}_h+\mathfrak{L}_v$. Note that this function is identically $1$ on $(0,\frac{1}{6}]\times[\frac{5}{6},1)\cup [\frac{1}{6},\frac{2}{3}]\times[\frac{11}{12},1)\cup(0,\frac{1}{12}]\times[\frac{1}{3},\frac{5}{6}]$ and nonzero only on $(0,\frac{1}{3})\times(\frac{3}{2},1)\cup (\frac{1}{6},\frac{5}{6})\times(\frac{5}{6},1)\cup(0,\frac{1}{6})\times(\frac{1}{6},\frac{5}{6})$. }
    \label{fig:fit-at-the-corners-1}
\end{figure}
\begin{figure}[h!]
    \begin{center}
    \includegraphics[width=7cm]{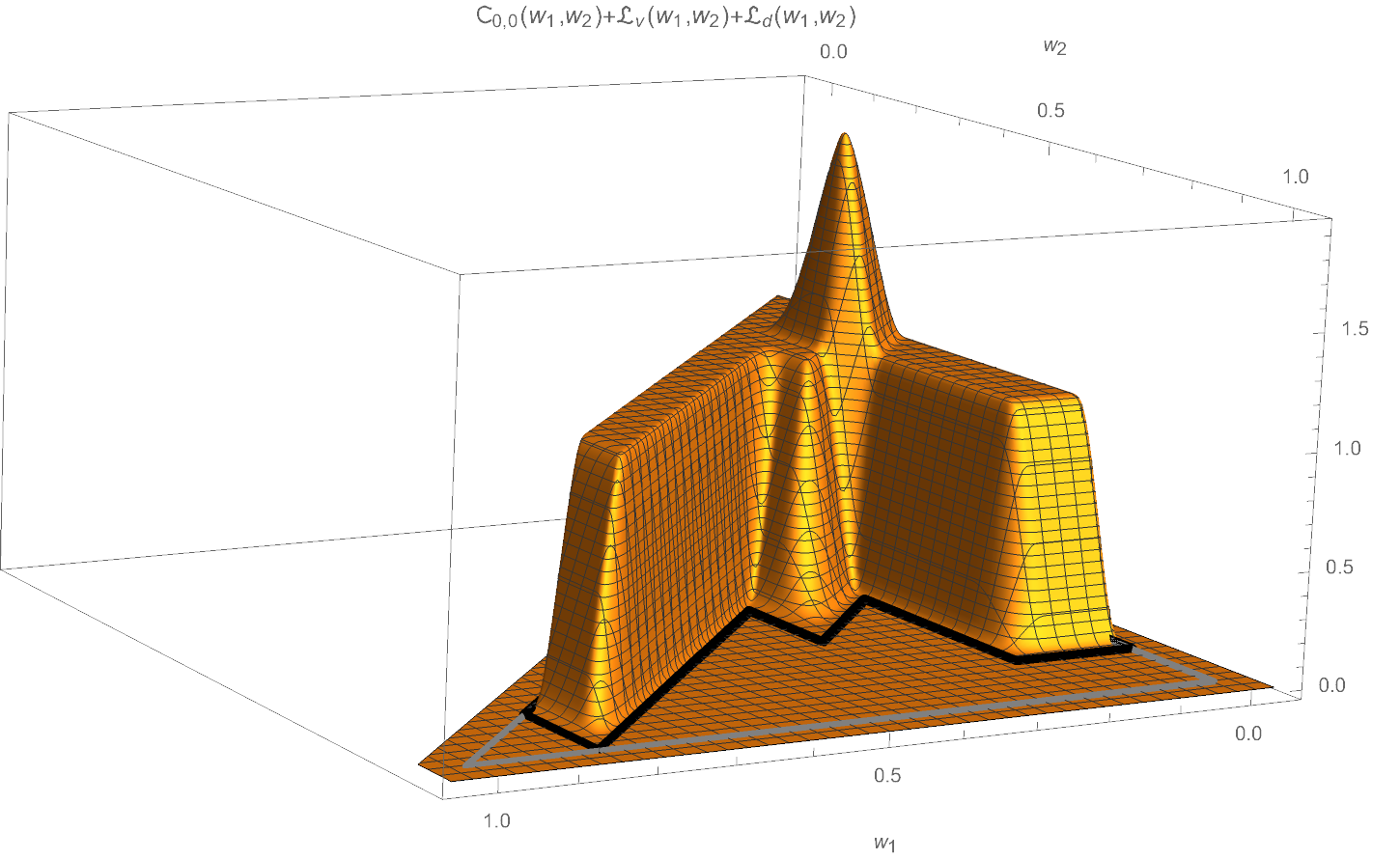}
    \includegraphics[width=7cm]{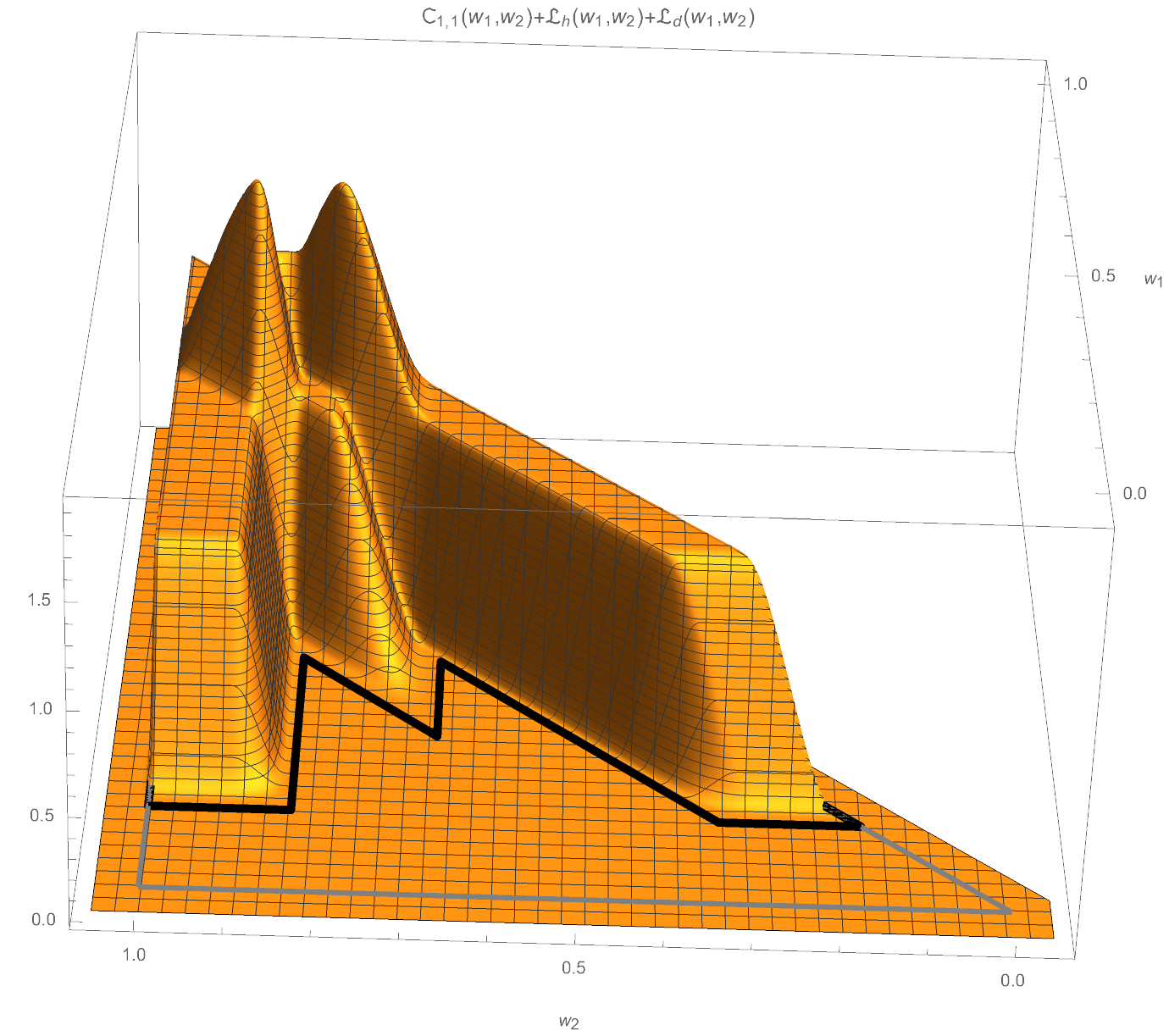}
    \end{center}
    \caption{Left panel: $\mathfrak{C}_{0,0}+\mathfrak{L}_v+\mathfrak{L}_d$. The `bump' that goes above $1$ is due to the overlap of the supports of $\mathfrak{C}_{0,0}$ and $\mathfrak{L}_v$ (in cyan and magenta, respectively, in Figure \ref{fig:corners-and-sides-diagrams})  Right panel: the sum $\mathfrak{C}_{1,1}+\mathfrak{L}_h+\mathfrak{L}_d$. 
    The two `bumps' that go above one are due to the overlap of the supports of the three functions (in green, red, and yellow, respectively, in Figure \ref{fig:corners-and-sides-diagrams}). 
    The important feature in both panels is the fact that approaching the boundary of the triangle from its interior, each sum --in a neighbourhood of its corner-- limits to $1$. In other words, the contribution of the `bumps' vanishes at the boundary of the triangle. 
    The reason the two sums $\mathfrak{C}_{0,0}+\mathfrak{L}_v+\mathfrak{L}_d$ and $\mathfrak{C}_{1,1}+\mathfrak{L}_h+\mathfrak{L}_d$ cannot be obtain from one another by a rotation is that $\mathfrak{L}_d$ is obtained from the template line function \eqref{def:TLine} by means of a vertical shear.}
    \label{fig:fit-at-the-corners-2}
\end{figure} 
Everywhere in the interior of the triangle, the sum \eqref{sum-6-functions} is smooth. An important consequence of this fact is that we can find a seventh function $\mathfrak{F}_\text{smooth}$ which is smooth everywhere on $\R^2$, with $\mathfrak{F}_\text{smooth}(w_1,w_2)=0$ if $(w_1,w_2)\notin\mathfrak{T}_{(0,1)}$, such that 
\begin{align}
\label{eq:decomposition-of-triangle}
    T_{(0,1)}=\mathfrak{C}_{0,0}+\mathfrak{C}_{0,1}+\mathfrak{C}_{1,1}+ \mathfrak{L}_h+\mathfrak{L}_v+\mathfrak{L}_d+\mathfrak{F}_{\text{smooth}}.
    \end{align}
\begin{remark}\label{remark:where-line-segments-come-from}
    Expanding \eqref{def:TLine}, we can write the template line function in terms of two `rectangular' corner functions, 
$(x,y)\mapsto F_{\frac{1}{12},\frac{1}{6},\frac{1}{3}}(\frac{1}{2}-x) F_{\frac{1}{24},\frac{1}{12},\frac{1}{6}}(y)=T_{\operatorname{Cor}}(\frac{1}{2}-x,2y)$ and $(x,y)\mapsto F_{\frac{1}{12},\frac{1}{6},\frac{1}{3}}(-\frac{1}{2}+x)F_{\frac{1}{24},\frac{1}{6},\frac{1}{3}}(y)=T_{\operatorname{Cor}}(\frac{1}{2}-x,2y)$, and the function $(x,y)\mapsto T_{\operatorname{Segm}}(x,y):=\mathbf{1}_{\{\frac{1}{2}\}}(x)F_{\frac{1}{24},\frac{1}{12},\frac{1}{6}}(y)$,  supported on the vertical line segment $\{(\frac{1}{2},y):\: 0\leq y\leq \frac{1}{6}\}$. Similarly, we can write each line function $\mathfrak{L}_h$, $\mathfrak{L}_v$, $\mathfrak{L}_d$ in terms of (possibly sheared) images of these rectangular corner functions and functions supported on vertical or horizontal line segments. 
The next lemma will allow us to ignore the contribution of specific functions supported along such line segments, as well the horizontal line segment supporting $T_{(0,1]}-T_{(0,1)}$.

\end{remark}

\begin{lemma}\label{lemma:lines-dont-matter}
    Let $\lambda$ be a probability measure on $\R$ which is absolutely continuous with respect to the Lebesgue measure. Let $x$ be randomly distributed according to $\lambda$. Let $F$ be one of the following functions supported on horizontal or vertical line segments:
    \begin{itemize}
        \item $\mathbf{1}_{(0,1)\times \{1\}}$,
        \item $(w_1,w_2)\mapsto T_{\operatorname{Segm}}(w_2,w_1)$,
        \item $(w_1,w_2)\mapsto T_{\operatorname{Segm}}(w_1,1-w_2)$,
        \item $(w_1,w_2)\mapsto T_{\operatorname{Segm}}(w_1,w_2-w_1)$
    \end{itemize}
    Then 
    \begin{align}\label{lemma:lines-dont-matter-statement-tightness}
    \lambda\left(\left\{x\in \mathbb R: \left| \Theta_{F}^{(2)}((-x+iy_1,0;\vecxi_1,\zeta_1),(x+iy_2,0;\vecxi_2,\zeta_2))\right|>R\right\}\right)&\ll \frac{1}{(1+R)^4}.
    \end{align}
uniformly in $0<y_1,y_2\leq 1$ and $(\vecxi_1,\zeta_1),(\vecxi_2,\zeta_2)\in\Hei$.      Moreover, for every $\alpha,\beta,\zeta_1,\zeta_2\in\R$ we have
    \begin{align}\label{lemma:lines-dont-matter-statement-fdd}
    \lim_{N\to\infty} \lambda \bigg(\bigg\{x\in \R:  \left|\Theta^{(2)}_F\!\left((-x+\tfrac{i}{N^2},0;\sve{-\alpha-\beta x}{0},\zeta_1),(x+\tfrac{i}{N^2},0;\sve{\alpha+\beta x}{0},\zeta_2)\right)\right|
    >\delta\bigg\}\bigg)=0
\end{align}
for every $\delta>0$.
Additionally, let $0<\varepsilon<\frac{1}{4}$. Then for all $y_1(x),y_2(x)$ with $0<y_0\leq  y_1(x),y_2(x)\leq 1$, we have uniformly in $(\vecxi_1(x),\zeta_1(x)),(\vecxi_2(x),\zeta_2(x))\in\Hei$ the estimate
\begin{equation}\label{eq:lines-dont-matter-random-y}
     \lambda\left(\left\{x\in \mathbb R: \left| \Theta_{F}^{(2)}((-x+iy_1(x),0;\vecxi_1(x),\zeta_1(x)),(x+iy_2(x),0;\vecxi_2(x),\zeta_2(x)))\right|>\frac{R}{y_1^\varepsilon(x)y_2^\varepsilon(x)}\right\}\right)\ll \frac{1}{(1+R)^4}.
\end{equation}
\end{lemma}
\begin{proof}
    If $F(w_1,w_2)=\mathbf{1}_{(0,1)\times \{1\}}(w_1,w_2)=\chi(w_1)\delta_{1}(w_2)$, then 
\begin{align}\label{Theta^(2)_product}
        \Theta^{(2)}_F((-x+iy_1, 0;\vecxi_1,\zeta_1),(x+iy_2,0;\vecxi_2,\zeta_2))
        =
        \,\Theta_\chi(-x+iy_1,0;\vecxi_1,\zeta)\Theta_{\delta_{1}}(x+iy_2,0;\vecxi_2,\zeta_2)
    \end{align}
    and
    \begin{align*}
        \Theta_{\delta_{1}}(x+iy_2,0;\vecxi_2,\zeta_2)=
        \begin{cases}
        y_2^{\frac{1}{4}}
        \e{\zeta_2
        +\frac{1}{2}
        \frac{x}{y_2}+\frac{\xi_{2,1}}{y_2^{1/2}}
        }&\mbox{if $\frac{y_2^{1/2}\xi_{2,2}+1}{y^{1/2}}\in\Z$},\\
        0&\mbox{otherwise;}
        \end{cases}
        \end{align*}
        is bounded in absolute value by $1$ for all $x\in\R$, $0<y_2\leq1$, and $(\vecxi_2,\zeta_2)\in\Hei$.
        Therefore Proposition \ref{prop:3.17-simplified} implies
        \begin{align}
            &\lambda\left(\left\{x\in\R: \left| \Theta_{F}^{(2)}\!\left((-x+iy,0;\vecxi_1,\zeta_1),(x+iy,0;\vecxi_2,\zeta_2)\right)\right|>R\right\}\right)
            \leq \lambda\left(\left\{x\in\R:\left| \Theta_{\chi}\!\left(-x+iy,0;\vecxi_1,\zeta_1\right)\right|>R\right\}\right)\nonumber\\
            &\ll \frac{1}{(1+R)^4}\nonumber
        \end{align}
        uniformly in $0<y_1,y_2\leq1$ and $(\vecxi_1,\zeta_1),(\vecxi_2,\zeta_2)\in\Hei$. 
        If $F(w_1,w_2)=T_{\operatorname{Segm}}(w_2,w_1)=F_{\frac{1}{24},\frac{1}{12},\frac{1}{6}}(w_1)\delta_{\frac{1}{2}}(w_2)=\chi_L(4w_1)\delta_{\frac{1}{2}}(w_2)$ the same argument works, but we apply Proposition \ref{prop:3.17-simplified} to $\Theta_{\chi_L(4\cdot)}$. 
        For the other two functions, the argument is again the same, after swapping the role of $w_1$ and $w_2$, and each time applying Proposition \ref{prop:3.17-simplified} to a slightly modified version of $\chi_L$ or $\chi_R$. This proves \eqref{lemma:lines-dont-matter-statement-tightness} for all four functions.

        Now let $\varepsilon>0$ be given. If $F(w_1,w_2)=\mathbf{1}_{(0,1)\times \{1\}}(w_1,w_2)=\chi(w_1)\delta_{1}(w_2)$, then using \eqref{Theta^(2)_product} with $y_1=y_2=N^{-2}$, $\vecxi_1=\sve{-\alpha-\beta x}{0}$, and $\vecxi_2=\sve{\alpha+\beta x}{0}$, we can bound the measure of the set in the left-hand-side of \eqref{lemma:lines-dont-matter-statement-fdd} by 
        \begin{align*}
           \lambda\!\left(\left\{x\in\R:\: \left|\tfrac{1}{\sqrt{N}}\Theta_{\chi}(-x+\frac{i}{N^2},0;\sve{-\alpha-\beta x}{0},\zeta_1)\right|>\varepsilon\right\}\right).
        \end{align*}
        The overnormalization by $\sqrt{N}$ and Lemma 4.7 in \cite{Cellarosi-Marklof} yield \eqref{lemma:lines-dont-matter-statement-fdd}. The argument for the other three functions is similar, each time applying Lemma 4.7 in \cite{Cellarosi-Marklof} to a slightly different compactly supported, Riemann integrable function on $\R$. 
        The proof of equation \eqref{eq:lines-dont-matter-random-y} follows similarly, using Proposition \ref{proposition:tail-with-random-y}. 
\end{proof}
\begin{remark}\label{remark:line-segments-for-other-triangles}
    Lemma \ref{lemma:lines-dont-matter} is written, for simplicity, with the triangle $\mathfrak{T}_{(0,1)}$ in mind. Minor changes allow us to use it in arbitrary triangles $\mathfrak{T}_{(s,t)}$ as well.
\end{remark}

\section{\texorpdfstring{Tightness of $\mathbf X_N$}{}}\label{section:tightness}
Recall that  $\mathbf{X}_N:\Delta_2^{(0,T)}\to\R^2\oplus\R^{2\times 2}$, defined in  \eqref{mathbfX_N(s,t)}, is a random function for each $N\geq1$. Also recall the space of geometric rough paths introduced in Definition \ref{def:geometric-rough-paths}. 
Recall from Section \ref{section:rough-path-to-theta}, equations \eqref{rewriting_X_N^11}-\eqref{rewriting_X_N^22} and Lemma \ref{lemma:grid-approximation-with-continuous-time} that 
\begin{align*}
    \mathbb{X}_N^{1,1}\!\left(\tfrac{m}{N},\tfrac{n}{N}\right)&=\Re(\tfrac{1}{4}\mathcal{L}_N^2+\tha\mathcal{J}_N)+R_{1,1}\\
    \mathbb{X}_N^{1,2}\!\left(\tfrac{m}{N},\tfrac{n}{N}\right)&=\Im(\tfrac{1}{4}\mathcal{L}_N^2+\tha\mathcal{J}_N)+R_{1,2}\\
    \mathbb{X}_N^{2,1}\!\left(\tfrac{m}{N},\tfrac{n}{N}\right)&=\Re(\tfrac{1}{4i}\mathcal{L}_N^2-\tfrac{1}{2i}\mathcal{J}_N)+R_{2,1}\\
    \mathbb{X}_N^{2,2}\!\left(\tfrac{m}{N},\tfrac{n}{N}\right)&=\Im(\tfrac{1}{4i}\mathcal{L}_N^2-\tfrac{1}{2i}\mathcal{J}_N)+R_{2,2},
\end{align*}
where $R_{i,j}$ go to $0$ in probability on $C^{2\gamma}(\Delta_2^{(0,T)},\mathbb R^{2\times 2})$. By Slutsky's theorem, in order to show convergence of $\mathbb X_N$ on $C^{2\gamma}(\Delta_2^{(0,T)},\mathbb R^{2\times 2})$ we just need to show convergence in distribution of
\begin{align*}
    \bar{\mathbb{X}}_N^{1,1}\!\left(\tfrac{m}{N},\tfrac{n}{N}\right)&:=\Re(\tfrac{1}{4}\mathcal{L}_N^2+\tha\mathcal{J}_N)\\
    \bar{\mathbb{X}}_N^{1,2}\!\left(\tfrac{m}{N},\tfrac{n}{N}\right)&:=\Im(\tfrac{1}{4}\mathcal{L}_N^2+\tha\mathcal{J}_N)\\
    \bar{\mathbb{X}}_N^{2,1}\!\left(\tfrac{m}{N},\tfrac{n}{N}\right)&:=\Re(\tfrac{1}{4i}\mathcal{L}_N^2-\tfrac{1}{2i}\mathcal{J}_N)\\
    \bar{\mathbb{X}}_N^{2,2}\!\left(\tfrac{m}{N},\tfrac{n}{N}\right)&:=\Im(\tfrac{1}{4i}\mathcal{L}_N^2-\tfrac{1}{2i}\mathcal{J}_N)
\end{align*}
as $N\to\infty$.
Also recall Lemma \ref{lemma:grid-approximation-with-continuous-time} which shows that studying ${\mathbb X}_N$ on grid points is sufficient. The main result of this section is the following, which is one of the two main ingredients needed to prove Theorem \ref{theorem:main}.

\begin{theorem}\label{theorem:tightness-second-order}
For all $\gamma\in (0,\frac{1}{2})$ the laws of $\mathbf{X}_N=(X_N,\bar{\mathbb{X}}_N)$ are are tight on $\mathscr C_g^{\gamma}$.   
\end{theorem}

To prove this theorem, we need tightness for $X_N$ (already established in Proposition \ref{prop:1st-order-tightness}) \emph{and} for $\bar{\mathbb{X}}_N$. 
In light of 
\eqref{def-L_N}, tightness for $\mathcal{L}_N$ is the same as tightness for $X_N$ (done in Proposition  \ref{prop:1st-order-tightness}). Therefore we focus on establishing tightness for $\mathcal{J}_N$ and, thanks to \eqref{writing-J_N-using-horocycles}, we must consider the rank-2 theta function $\Theta^{(2)}_{T_{(s,t]}}$ evaluated along a specific pair of horocycle lifts. 
Thanks to the first part of Lemma \ref{lemma:lines-dont-matter} and Remark \ref{remark:line-segments-for-other-triangles}, we may instead consider $\Theta^{(2)}_{T_{(s,t)}}$. We will first focus on the case $s=0, t=1$.

Using  linearity,  
we aim to write the theta function $\Theta_{T_{(0,1)}}^{(2)}$ in terms of $7$ theta functions, one for each term in \eqref{eq:decomposition-of-triangle}. There are $3$ ``sheared" components, $\mathfrak C_{0,0},\mathfrak C_{1,1}$ and $\mathfrak L_d$ corresponding to the two acute corners and the hypotenuse of the triangle (in cyan, green, and yellow respectively in Figure \ref{fig:corners-and-sides-diagrams}), and the rest are non-sheared. The non-sheared components are easier to handle, but the sheared components require mapping $G\times G$ into a larger group.

\begin{remark}
    Unfortunately, as we write the $T_{(0,1)}$ as a sum of $7$ pieces, as in Section \ref{section:defining-theta-T} we will have an abundance of repetitive tail bounds. In order to prove tightness in the same way as Proposition \ref{prop:1st-order-tightness} we must prove tail bounds for each of these terms with random $y$ such as in Proposition \ref{proposition:tail-with-random-y}. Additionally, we prove tail bounds for non-random $y$ which we make use of in Section \ref{section:fdds}. For expository reasons, we include these tail bounds together even though we do not use tail bounds for  deterministic $y$ in this section. As the tail bounds in this section are repetitive we will often omit details in proofs. 
\end{remark}

We first recall the function $H$ defined in equation (3.125) in \cite{Cellarosi-Marklof}
\begin{equation}\label{def-height-function-H}
    H(x+iy):=\sum_{\gamma \in \Gamma_\infty \setminus \operatorname{PSL}(2,\mathbb Z)}y_\gamma^{1/2}\chi_{[1/2,\infty)}(y_\gamma^{1/2}),
\end{equation}
where $\Gamma_\infty=\left\{\begin{pmatrix}
    1&&x\\
    0&&1
\end{pmatrix}: x\in \mathbb Z\right\}$ is the stabilizer of the cusp and $y_\gamma:=\Im(\gamma (x+iy))$,  with $\gamma$ acting by M\"obius transformations. The height function $H$ can be extended to a function on $G$ by defining $H\left(x+iy;\phi,\boldsymbol \xi,\zeta\right):=H(x+iy).$
\subsection{Estimates for non-sheared components}\label{subsection-estimates-for-nonsheared}
In this subsection, we establish tail bounds for smooth $f$ and for functions $$F\in \{\mathfrak{C}_{0,1},\mathfrak{L}_h,\mathfrak{L}_v\}.$$ In these three cases, $\Theta_F^{(2)}$ is really just a product of two rank $1$ theta functions. We follow the same strategy as the proof of Proposition 3.17 in \cite{Cellarosi-Marklof} and the proof for the tightness in Proposition \ref{prop:1st-order-tightness}.  We start with a bound for regular $F$, which will be building blocks for our non-sheared and sheared components. 
\begin{lemma}\label{lemma:HH-bound-for-unshear}
Let $F\in \mathcal S_{\eta_1,\eta_2}(\mathbb R^2)$ and consider any $(g_1,g_2)\in G\times G$, where $g_j=(x_j+iy_j,\phi_j;\sve{\xi_{j,1}}{\xi_{j,2}},\zeta_j)$ for $j=1,2$. We then have the estimate 
    \begin{equation}\label{eq:lemma:HH-bound-for-unshear}
        |\Theta_F^{(2)}(g_1,g_2)|\ll H(x_1+iy_1)H(x_2+iy_2),
    \end{equation}
    where the implied constant depends only on $f$.
\end{lemma}
\begin{proof}
    Lemma \ref{lemma:growth-in-cusp-second-order} implies that if $y_1,y_2\geq \frac{1}{2}$ we have 
    \begin{equation*}
        |\Theta_F^{(2)}(g_1,g_2)|\ll y_1^{\frac{1}{2}} y_2^{\frac{1}{2}},
    \end{equation*}
    where the constant only depends on $f$. Hence this is true for all $(g_1,g_2)$ in the fundamental domain $\mathcal F_\Gamma \times \mathcal F_\Gamma$. For arbitrary $(g_1,g_2)\in G\times G$, from the defintion of fundamental domain there exist some $(\gamma_1,\gamma_2)\in \Gamma\times \Gamma$ so that $(g_1',g_2')=(\gamma_1,\gamma_2) (g_1,g_2)\in \mathcal F_\Gamma \times \mathcal F_\Gamma$. Therefore recalling the  invariance of $\Theta_F^{(2)}$ under $\Gamma\times \Gamma$  by Proposition \ref{prop:invariance-under-gamma-gamma}, we have 
    $$|\Theta_F^{(2)}(g_1,g_2)|=|\Theta_F^{(2)}(g_1',g_2')|\ll {y_1'}^{\frac{1}{2}} {y_2'}^{\frac{1}{2}}\leq H(x_1+iy_1)H(x_2+iy_2).$$
\end{proof}
\begin{proposition}\label{prop:bound-for-second-order-smooth}
    Let $\lambda$ be a Borel probability measure on $\mathbb R$ absolutely continuous with respect to Lebesgue measure. Let $F\in \mathcal S_{\eta_1,\eta_2}(\R^2)$ with $\eta_1,\eta_2>1.$ Then for any $(g_1,g_2)\in  G\times  G$ with $-x_1=x_2=x$ we have, uniformly in $y_1,y_2\leq 1$, that 
 \begin{equation*}
        \lambda\left(\left\{x\in \mathbb R: \left| \Theta_{F}^{(2)}(g_1,g_2)\right|>R\right\}\right)\ll \frac{1}{(1+R)^2}.
    \end{equation*}
\end{proposition}
\begin{proof}
    Lemma \ref{lemma:HH-bound-for-unshear} implies that $\lambda\left(\left\{x\in \mathbb R: \left| \Theta_{F}^{(2)}(g_1,g_2)\right|>R\right\}\right)$ is bounded above by
    $$
    \lambda\left(\left\{x\in \mathbb R: H(-x+iy_1)H(x+iy_2)>\tfrac{R}{C}\right\}\right),$$ where $C$ is the constant implied by \eqref{eq:lemma:HH-bound-for-unshear}.
    A union bound then implies that 
   \begin{align*} \lambda\left(\left\{x\in \mathbb R: \left| \Theta_{F}^{(2)}(g_1,g_2)\right|>
   R
   \right\}\right)&\leq \lambda\left(\left\{x\in \mathbb R: H(-x+i y_1)>\sqrt{\tfrac{R}{C}}\right\}\right)\\
   &~+\lambda\left(\left\{x\in \mathbb R: H(x+i y_2)>\sqrt{\tfrac{R}{C}}\right\}\right).\end{align*}
   We now bound each of the previous terms using  the strategy of the proof of Proposition 3.20 in \cite{Cellarosi-Marklof}.
   \begin{align*}
       \int_{\R} \mathbf{1}_{\left\{H(-x+i y_1)<\sqrt{\frac{R}{C}}\right\}}\,\de\lambda&=\int_{\R/\Z}\mathbf{1}_{\left\{H(-x+i y_1)<\sqrt{\frac{R}{C}}\right\}}\de \lambda_{\Z}\\
       &\ll \int_{\R/\Z}\mathbf{1}_{\left\{H(-x+i y_1)<\sqrt{\frac{R}{C}}\right\}}\de x,
   \end{align*}
   where $\lambda_{\Z}$ is the push forward of $\lambda$ to $\R/\Z$ via the quotient map.  Lemma 3.19 in \cite{Cellarosi-Marklof} implies that for $R\geq C$ we have
   \begin{align*}
       \int_{\R/\Z}\mathbf{1}_{\left\{H(-x+i y_1)<\sqrt{\frac{R}{C}}\right\}}\de x\leq 2 C^2 R^{-2},
   \end{align*} 
   uniformly in $y_1\leq 1$.
   The same bound holds for $\int_\R \mathbf{1}_{\left\{H(x+i y_2)<\sqrt{R/C}\right\}}\,\de\lambda$, uniformly in $y_2\leq1$, finally implying the desired bound $\ll (1+R)^{-2}$.
\end{proof}

\begin{proposition}\label{prop:bound-for-second-order-smooth-random-y}
    Let $\lambda$ be a Borel probability measure on $\mathbb R$ absolutely continuous with respect to Lebesgue measure. Let $f\in \mathcal S_{\eta_1,\eta_2}(\mathbb R^2)$ with $\eta_1,\eta_2>1.$ Let $(g_1,g_2)\in  G\times  G$ be so that $-x_1=x_2=x$, $\vecxi_1(x)$, $\vecxi_2(x)$ and $y_1(x)$ $y_2(x)$ are functions of $x$ so that $0<y_0\leq y_1(x),y_2(x)\leq 1.$ Let $\varepsilon>0$ be small. Then uniformly in $y_0$ we have
 \begin{equation*}
        \lambda\left(\left\{x\in \mathbb R: \left| \Theta_{f}^{(2)}(g_1,g_2)\right|>\frac{R}{(y_1(x)y_2(x))^
        \varepsilon}\right\}\right)\ll \frac{1}{(1+R)^2}.
    \end{equation*}
\end{proposition}
\begin{proof}
    Lemma \ref{lemma:HH-bound-for-unshear} implies that 
    \begin{align*}
    \lambda&\left(\left\{x\in \mathbb R: \left| \Theta_{f}^{(2)}(g_1,g_2)\right|>\frac{R}{(y_1(x)y_2(x))^
        \varepsilon}\right\}\right)\\&\leq \lambda\left(\left\{x\in \mathbb R: H(-x+iy_1(x))H(x+iy_2(x))>\frac{R}{C(y_1^\varepsilon(x)y_2^\varepsilon(x)}\right\}\right), \end{align*}
    where $C$ is the constant implied in \eqref{eq:lemma:HH-bound-for-unshear}.
    A union bound implies that 
   \begin{align*} \lambda\left(\left\{x\in \mathbb R: \left| \Theta_{f}^{(2)}(g_1,g_2)\right|>\frac{R}{C(y_1(x)y_2(x))^\varepsilon}\right\}\right)
   &\leq \lambda\left(\left\{x\in \mathbb R: (y_1(x))^\varepsilon\, H(-x+i y_1(x))>\sqrt{\tfrac{R}{C}}\right\}\right)\\
   &~+\lambda\left(\left\{x\in \mathbb R: (y_2(x))^\varepsilon\, H(x+i y_1(x))>\sqrt{\tfrac{R}{C}}\right\}\right).
   \end{align*}
   Each of the two terms is estimated using the same argument used in the proof of Lemma \ref{lemma:tail-bound-smooth-f-random-y} (which uses Lemma \ref{lemma:integral-over-gamma-R4}), yielding the desired bound $\ll (1+R)^{-2}$.
\end{proof}
\begin{proposition}\label{prop:bound-unshear-building-blocks}
    Let $\lambda$ be a Borel probability measure on $\mathbb R$ absolutely continuous with respect to Lebesgue measure.  Then for any $(g_1,g_2)\in  G\times  G$ with $-x_1=x_2=x$ we have uniformly in $y_1,y_2\leq 1$ that 
    \begin{align*}
    \lambda\left(\left\{x\in \mathbb R: \left| \Theta_{\mathfrak{C}_{0,1}}^{(2)}(g_1,g_2)\right|>R\right\}\right)&\ll \frac{1}{(1+R)^2}\\
    \lambda\left(\left\{x\in \mathbb R: \left| \Theta_{\mathfrak{L}_h}^{(2)}(g_1,g_2)\right|>R\right\}\right)&\ll \frac{1}{(1+R)^2}\\
 \lambda\left(\left\{x\in \mathbb R: \left| \Theta_{\mathfrak L_v}^{(2)}(g_1,g_2)\right|>R\right\}\right)&\ll \frac{1}{(1+R)^2}
    \end{align*}
\end{proposition}
\begin{proof}
    Recalling \eqref{def:TCor}, since  $\mathfrak C_{0,1}(w_1,w_2)=F_1(w_1)F_2(w_2)$ with $F_1(w)=F_{\frac{1}{12},\frac{1}{6},\frac{1}{3}}(w_1)=\chi_L(2w_1)$ and $F_1(w_2)=F_{\frac{1}{12},\frac{1}{6},\frac{1}{3}}(1-w_2)=
    \chi_R(2w-1)$, we have that $\Theta_{\mathfrak C_{0,1}}^{(2)}(g_1,g_2)=\Theta_{F_1}(g_1)\Theta_{F_2}(g_2)$ 
    and using a union bound we have 
    \begin{align*}
        \lambda\left(\left\{x\in \mathbb R: \left| \Theta_{\mathfrak C_{0,1}}^{(2)}(g_1,g_2)\right|> R\right\}\right)&\leq \lambda\left(\left\{x\in \mathbb R: \left| \Theta_{F_1}(g_1)\right|>\sqrt R\right\}\right)\\
        &~+\lambda\left(\left\{x\in \mathbb R: \left| \Theta_{F_2}(g_2)\right|>\sqrt R\right\}\right).
    \end{align*}
    Proposition \ref{prop:3.17-simplified} (applied to  slightly modified $\chi_L$ and $\chi_R$) implies that each or the terms above is $\ll (1+\sqrt{R})^{-4}\ll (1+R)^{-2}$.
    
    Recalling \eqref{def:TLine} and Remark \ref{remark:where-line-segments-come-from}, we write $\mathfrak{L}_h(w_1,w_2)=
    F_1(w_1)F_2(w_2)+F_3(w_1)F_4(w_2)+F_5(w_1,w_2)$, where $F_1(w_1)=F_{\frac{1}{12},\frac{1}{6},\frac{1}{3}}(-\frac{1}{2}+w_1)=\chi_L(-1+2w_1)$, $F_2(w_2)=F_4(w_2)=F_{\frac{1}{24},\frac{1}{12},\frac{1}{6}}(1-w_2)
    =\chi_R(4w_2-3)$, $F_3(w_1)=F_{\frac{1}{12},\frac{1}{6},\frac{1}{3}}(\frac{1}{2}-w_1)
    =\chi_R(2w_1)$, and $F_5(w_1,w_2)=T_{\operatorname{Segm}}(w_1,1-w_2)$.
    Successive union bounds give
    \begin{align*}
        \lambda\left(\left\{x\in \mathbb R: \left| \Theta_{\mathfrak L_{h}}^{(2)}(g_1,g_2)\right|> R\right\}\right)\leq &\sum_{\ell\in\{1,3\}} \lambda\left(\left\{x\in \mathbb R: \left| \Theta_{F_\ell}(g_1)\right|>\sqrt{\tfrac{R}{3}}\right\}\right)\\
        &+\sum_{\ell\in\{2,4\}} \lambda\left(\left\{x\in \mathbb R: \left| \Theta_{F_\ell}(g_2)\right|>\sqrt{\tfrac{R}{3}}\right\}\right)\\
        &+ \lambda\left(\left\{x\in \mathbb R: \left| \Theta_{F_5}(g_1,g_2)\right|>\tfrac{R}{3}\right\}\right).
    \end{align*}
    Proposition \ref{prop:3.17-simplified} (again applied to slightly modified $\chi_L$ and $\chi_R$) yields the bound $\ll \left(1+\sqrt{R/2}\right)^{-4}\ll(1+R)^{-2}$ for each of the four sums above, while the fifth term is $\ll(1+R)^{-4}$ by the first part of Lemma \ref{lemma:lines-dont-matter}. Combining the five bounds we obtain the tail bound for $\Theta^{(2)}_{\mathfrak{L}_h}$.

    The tail bound for $\Theta^{(2)}_{\mathfrak{L}_v}$ since it is very similar to that of $\Theta^{(2)}_{\mathfrak{L}_h}$, using different $F_1,\ldots F_4$ (each handled by Proposition \ref{prop:3.17-simplified}) and $F_5$ (handled by the first part of Lemma \ref{lemma:lines-dont-matter}).
    \end{proof}
\begin{proposition}\label{prop:bound-unshear-building-blocks-random-y}
    Let $\lambda$ be a Borel probability measure on $\mathbb R$ absolutely continuous with respect to Lebesgue measure.  Then for any $g\in  G\times  G$ with $-x_1=x_2=x$, $\vecxi_1(x),$ $\vecxi_2(x)$ and $y_1(x),$ $y_2(x)$ with $0<y_0\leq y_1(x),y_2(x)\leq 1$, and $\varepsilon>0$ small, uniformly in $y_0$ that 
    \begin{align*}
    \lambda\left(\left\{x\in \mathbb R: \left| \Theta_{\mathfrak C_{0,1}}^{(2)}(g_1,g_2)\right|>\frac{R}{(y_1(x)y_2(x))^\varepsilon}\right\}\right)&\ll \frac{1}{(1+R)^2},\\
    \lambda\left(\left\{x\in \mathbb R: \left| \Theta_{\mathfrak L_h}^{(2)}(g_1,g_2)\right|>\frac{R}{(y_1(x)y_2(x))^\varepsilon}\right\}\right)&\ll \frac{1}{(1+R)^2},\\
 \lambda\left(\left\{x\in \mathbb R: \left| \Theta_{\mathfrak L_v}^{(2)}(g_1,g_2)\right|>\frac{R}{(y_1(x)y_2(x))^\varepsilon}\right\}\right)&\ll \frac{1}{(1+R)^2}.
    \end{align*}
\end{proposition}
\begin{proof} 
    We argue as in the proof of Proposition \ref{prop:bound-unshear-building-blocks}. Writing $\mathfrak C_{0,1}(w_1,w_2)=F_1(w_1)F_2(w_2)$, a union bound yields
    \begin{align*}
        \lambda\left(\left\{x\in \mathbb R: \left| \Theta_{\mathfrak{C}_{0,1}}^{(2)}(g_1,g_2)\right|> \frac{R}{(y_1(x)y_2(x))^\varepsilon}\right\}\right)
        &\leq \lambda\left(\left\{x\in \mathbb R: y_1^\varepsilon(x)\left| \Theta_{F_1}(g_1)\right|>\sqrt R\right\}\right)\\
        &~+\lambda\left(\left\{x\in \mathbb R: y_2^\varepsilon(x)\left| \Theta_{F_2}(g_2)\right|>\sqrt R\right\}\right).
    \end{align*}
    Applying Proposition \ref{proposition:tail-with-random-y} to slightly modified $\chi_L$ and $\chi_R$, each of the two terms above is bounded by $\ll_\varepsilon (1+\sqrt{R})^{-4}\ll_\varepsilon(1+R)^{-2}$.
    The argument for $\mathfrak L_h$ and $\mathfrak L_h$ is similar, but involves a union bound with five terms as in the proof of  Proposition \ref{prop:bound-unshear-building-blocks}.
\end{proof}

\subsection{The rank-2 theta function of Maklof-Welsh. 
}\label{subsection:theta-theta-tilde}
In order to handle the ``sheared" components such as $\mathfrak C_{1,1}$, it is helpful to map  $G\times G$ to a bigger group. In \cite{Marklof-Welsh-I,Marklof-Welsh-II} the authors consider the theta function $\Tilde{\Theta}_f^{(2)}:\mathbb H(\mathbb R^2)\rtimes \operatorname{Sp}(2,\R)\to \mathbb C$ (we use the tilde to distinguish it from our $\Theta^{(2)}$, although in their notation there is no tilde). Here $f\in \mathcal S(\mathbb R^2)$ is in the classical Schwarz space. We can map $G\times G$ to $\mathbb H(\R^2)\rtimes \operatorname{Sp}(2,\R)$ in a canonical way. In coordinates, using the notation from Section 4.1 of \cite{Marklof-Welsh-I}, we map $g_j=(x_j+iy_j,\phi_j;
    \sve{\xi_{j1}}{\xi_{j2}},\zeta_j)$ for $j=1,2$ to 
    \begin{align*}
    &X(g_1,g_2)=\begin{pmatrix}
        x_1&&0\\
        0&&x_2
    \end{pmatrix},\hspace{.5cm}
    &Y(g_1,g_2)=\begin{pmatrix}
        y_1&&0\\
        0&&y_2
    \end{pmatrix},\hspace{.5cm}
    &Q(g_1,g_2)=\begin{pmatrix}
        e^{i\phi_1}&&0\\
        0&&e^{i\phi_2}
    \end{pmatrix},\hspace{.5cm}\\
    &t(g_2,g_2)=-\zeta_1-\zeta_2,\hspace{.5cm}
    &\mathbf{x}(g_2,g_2)=(-\xi_{21},-\xi_{22}), \hspace{1.4cm}
    &\mathbf y(g_2,g_2)=(\xi_{11},\xi_{12}).
    \end{align*}
We emphasize that this is not quite an embedding due to the lack of injectivity coming from $t(g_1,g_2)$, and from the fact that $G$ is defined in terms of the universal cover of $\sltr$, while no cover of $\operatorname{Sp}(2,\R)$ is considered in \cite{Marklof-Welsh-I, Marklof-Welsh-II}.  Elements in $H(\R^2)\rtimes \operatorname{Sp}(2,\R)$ are of the form $(h,g)$, where $h=(\mathbf{x},\mathbf{y},t)\in\mathbb{H}(\R^2)$,  $\mathbf{x},\mathbf{y}$ are 2-dimensional row vectors, and $g$ is a $4\times4$ symplectic matrix. Since $\mathbf{x}, \mathbf{y},t$ all depend on $g_1,g_2\in G$, so does $h$.
Note that if $(M_1,M_2)\in \operatorname{SL}(2,\mathbb R)\times \operatorname{SL}(2,\mathbb R)$ with 
\begin{align*}
    M_1=\begin{pmatrix}
        a_1&&b_1\\
        c_1&&d_1
    \end{pmatrix},\hspace{1cm}
    M_2=\begin{pmatrix}
        a_2&&b_2\\
        c_2&&d_2
    \end{pmatrix},
\end{align*}
and we set
\begin{align*}
    A&=\begin{pmatrix}
        a_1&&0\\
        0&&a_2
    \end{pmatrix},  
    \hspace{.5cm}
    B=\begin{pmatrix}
        b_1&&0\\
        0&&b_2
    \end{pmatrix},
    \hspace{.5cm}
    C=\begin{pmatrix}
        c_1&&0\\
        0&&c_2
    \end{pmatrix}, 
    \hspace{.5cm}
    D=\begin{pmatrix}
        d_1&&0\\
        0&&d_2
    \end{pmatrix},
\end{align*}
then the matrix 
\begin{align*}
    g=\begin{pmatrix}
        A&&B\\
        C&&D
    \end{pmatrix}\in\operatorname{Sp}(2,\mathbb R),
\end{align*} 
is block diagonal. For these choices of $X,Y,Q,y,\mathbf x, \mathbf y$ with $\phi_1=\phi_2=0$ and for smooth $F$ we have that
\begin{equation*}
    \Theta_F^{(2)}(g_1,g_2)=\Tilde{\Theta}_F^{(2)}(h(g_1,g_2),g(g_1,g_2)).
\end{equation*}
Indeed, this follows from equation (4.3) in \cite{Marklof-Welsh-I} and from Proposition 2.1 in \cite{Marklof-Welsh-I} with the given $Q$. Additionally, by Theorem 4.1 in \cite{Marklof-Welsh-I} there is a lattice $\tilde \Gamma \subset \mathbb H(\mathbb R^2)\rtimes \operatorname{Sp}(2,\mathbb R)$ so that for all $\tilde \gamma\in \tilde \Gamma$ and we have that 
\begin{equation*}
    |\Tilde{\Theta}_F^{(2)}(\tilde \gamma(h,g))|= |\Tilde{\Theta}_F^{(2)}(h,g)|,
\end{equation*}
making the function $|\Tilde{\Theta}_{F}^{(2)}|:\tilde \Gamma \setminus(\mathbb H(\R^2)\rtimes \operatorname{Sp}(2,\R))\to \R_{\geq 0}$ well defined\footnote{However, $\Tilde{\Theta}_{F}^{(2)}$ is not well defined from the quotient  $\tilde \Gamma \setminus(\mathbb H(\R^2)\rtimes \operatorname{Sp}(2,\R))$ to $\C$ as there is a nontrivial automorphy factor of modulus $1$  in Theorem 4.1 in \cite{Marklof-Welsh-I}}. In addition, we note that for the above choices of $X,Y,Q,t,\mathbf x, \mathbf y$ with any $\phi_1,\phi_2\in \mathbb R$ and for smooth $F$ the functions $|\Theta_{F}^{(2)}|:(\Gamma \times \Gamma)\setminus G\times G\to \mathbb R_{\geq 0}$ and $|\Tilde{\Theta}_{F}^{(2)}|:\tilde \Gamma \setminus(\mathbb H(\mathbb R^2)\rtimes \operatorname{Sp}(2, \mathbb R))\to \mathbb R_{\geq 0}$ are equal. It is also sometimes convenient to write $Y=UV~{}^tU$ for diagonal $$V=\begin{pmatrix}
    v_1&&0\\
    0&&v_2
\end{pmatrix},$$
where $v_1\geq v_2>0$ if $Y$ is nondiagonal. There are two slightly different choices of fundamental domain for these higher rank theta functions - one given in \cite{Marklof-Welsh-I} and one given in \cite{Marklof-Welsh-II}. We will denote by $\mathcal D$ the fundamental domain given on page 7 of \cite{Marklof-Welsh-II}. Thankfully, the fundamental domain $\mathcal F_\Gamma \times \mathcal F_\Gamma$ maps to $\mathcal D.$
\begin{lemma}\label{lemma:embedding-of-fundamental-domains}
    Let $(g_1,g_2)\in \mathcal F_\Gamma \times \mathcal F_\Gamma$ be so that $\zeta_1=-\zeta_2$ and so that $y_1\geq y_2$. Then under the above mappings, $(h(g_1,g_2),g(g_1,g_2))\in \mathcal D.$
\end{lemma}
\begin{proof}
    If $(g_1,g_2)\in \mathcal F_\Gamma \times \mathcal F_\gamma$ then $x_1,x_2,\xi_{1,1},\xi_{1,2},\xi_{2,1},\xi_{2,2}\in [-\frac{1}{2},\frac{1}{2}]$ and $t=0$. Therefore $h(g_1,g_2)\in \Tilde{\mathcal D}$ (as defined in equation (3.4) in \cite{Marklof-Welsh-II}). We also need to check that $g(g_1,g_2)$ is of the form given in equation (3.6) in \cite{Marklof-Welsh-II}. We can just let $\mathbf r=0$, $X_1=x_2$, $\mathbf s=0,$ $t_1=x_1$, $v_1=y_1,Y_1=y_2$. Finally, we need to check that if $\gamma\in \operatorname{Sp}(2,\mathbb Z)$ then $v_1(\gamma g)\leq v_1(g)$. To this end, using equation (2.10) (also see equation (3.8)) in \cite{Marklof-Welsh-II} we see that
    $$v_1(\gamma g)^{-1}=\mathbf c Y ~{}^t\mathbf c+(\mathbf c X+\mathbf d) Y^{-1} ~{}^t(\mathbf c X+\mathbf d),$$
    where $\mathbf c=(c_1,c_2), \mathbf d=(d_1,d_2)$ are the first rows of $C,D$ in $\gamma =\begin{pmatrix}
        A&&B\\
        C&&D
    \end{pmatrix}.$
    Therefore
    \begin{align*}
        v_1(\gamma g)^{-1}&=c_1^2y_1+c_2^2y_2+\frac{(c_1x_1+d_1)^2}{y_1}+\frac{(c_2x_2+d_2)^2}{y_2}\\
        &=\frac{c_1^2y_1^2+(c_1x_1+d_1)^2}{y_1}+\frac{c_2^2y_2^2+(c_2x_2+d_2)^2}{y_2}\\
        &=\frac{1}{(y_1)_{\gamma_1}}+\frac{1}{(y_2)_{\gamma_2}},
    \end{align*}
    where $\gamma_1,\gamma_2\in \operatorname{SL}(2,\mathbb Z)$. As $(g_1,g_2)\in \mathcal F_\Gamma \times \mathcal F_\Gamma$ we have that $(y_1)_{\gamma_1}\leq y_1$ and hence
    \begin{align*}
        v_1(\gamma g)^{-1}&\geq \frac{1}{y_1}=v_1(g)^{-1}.
    \end{align*}
\end{proof}

\subsection{Estimates for sheared components}\label{subsection:estimates-for-sheared-blocks}
In this section, we give tail estimates for the ``sheared" components in \eqref{eq:decomposition-of-triangle}, namely $\mathfrak{C}_{0,0}$, $\mathfrak{C}_{1,1}$, and $\mathfrak{L}_d$, analogously to the estimates we provided in Section  \ref{subsection-estimates-for-nonsheared} for the non-sheared components. By Proposition 2.1 in \cite{Marklof-Welsh-I}, for smooth $F$ and for $\tau,\eta>0$ we have that
\begin{align*}
    |\Theta_{(w_1,w_2)\mapsto F(\tau w_1,\eta(w_2-w_1))}^{(2)}(g_1,g_2)|&=\tau\eta|\Tilde \Theta_{F}^{(2)}\left(\left(h(g_1,g_2),g(g_1,g_2)\right),(\operatorname{id},P)\right)|,
\end{align*}
where  
$$P=\begin{pmatrix}
    \Omega&&0\\
    0&&{}^t\Omega^{-1}
\end{pmatrix}$$
and
$$\Omega=\begin{pmatrix}
        \tau&&-\eta\\
        0&&\eta
    \end{pmatrix}.$$
We will use the following lemma to bound the ``sheared" components. Recall the height function in \eqref{def-height-function-H}.
\begin{lemma}\label{lemma:H-bound-for-sheared}
    Let $F\in \mathcal S(\mathbb R^2)$ and let 
    \begin{align*}
    X=\begin{pmatrix}
        x_1&&0\\
        0&&x,
    \end{pmatrix}, \hspace{.5cm} Y=\begin{pmatrix}
        y_1^{\frac{1}{2}}&&0\\
        0&&y_2^{\frac{1}{2}},
    \end{pmatrix}
    \end{align*}
    and $k(Q)=I$. Additionally let $$\Omega=\begin{pmatrix}
        \tau&&-\eta\\
        0&&\eta
    \end{pmatrix},$$
    with $\tau,  \eta> 0$. Define the matrix 
    $$P=\begin{pmatrix}
        \Omega&&0\\
        0&&{}^t\Omega^{-1}
    \end{pmatrix}$$
    and consider $(h_0,P)\in \mathbb H(\mathbb R^2)\rtimes \operatorname{Sp}(2,\mathbb R).$ Then, uniform in $\tau$ and $\eta$, we have the estimate 
    \begin{align}\label{Bound-tildeTheta^2-in-terms-of-HH}
        \left|\Tilde{\Theta}_F^{(2)}(( h,g)(h_0,P))\right|\ll_F H\!\left(x_1+i\tau^2 y_1\right)\,H\!\left(x_2+i\eta^2 y_2\right).
    \end{align} 
\end{lemma}
\begin{proof}
Without loss of generality, we may assume that $\sup_{\gamma \in \operatorname{SL}(2,\mathbb Z)}(\tau y_1)_\gamma \geq \sup_{\gamma \in \operatorname{SL}(2,\mathbb Z)}(\eta y_1)_\gamma.$ Otherwise we just consider reverse the embedding $(h(g_2,g_1),g(g_2,g_1))$ and  $$\tilde \Omega=\begin{pmatrix}
    \eta &&0\\
    -\eta&&\tau
\end{pmatrix}.$$
With this assumption in hand we may write $$\Omega=\begin{pmatrix}
    \tau&&0\\
    0&&\eta
\end{pmatrix}\begin{pmatrix}
    1&&-\frac{\eta}{\tau}\\
    0&&1
\end{pmatrix}.$$
Given our choice of $X,Y,k(Q)$ we have that
\begin{align*}
    gP&=\begin{pmatrix}
        I&&X\\
        0&&I
    \end{pmatrix}
    \begin{pmatrix}
        Y^{\frac{1}{2}}&&0\\
        0&&{}^tY^{-\frac{1}{2}}
    \end{pmatrix}
    \begin{pmatrix}
        \begin{pmatrix}
            \tau&&0\\
            0&&\eta
        \end{pmatrix}&&0\\
        0&&\begin{pmatrix}
            \frac{1}{\tau}&&0\\
            0&&\frac{1}{\eta}
        \end{pmatrix}
    \end{pmatrix} 
       \begin{pmatrix}
        \begin{pmatrix}
            1&&-\frac{\eta}{\tau}\\
            0&&1
        \end{pmatrix}&&0\\
        0&&\begin{pmatrix}
            1&&0\\
            \frac{\eta}{\tau}&&1
        \end{pmatrix}
    \end{pmatrix}\\
    &=\begin{pmatrix}
        I&&X\\
        0&&I
    \end{pmatrix}
    \begin{pmatrix}
        Y_{\tau,\eta}^{\frac{1}{2}}&&0\\
        0&&{}^tY_{\tau,\eta}^{-\frac{1}{2}}
    \end{pmatrix}
       \begin{pmatrix}
        \begin{pmatrix}
            1&&-\frac{\eta}{\tau}\\
            0&&1
        \end{pmatrix}&&0\\
        0&&\begin{pmatrix}
            1&&0\\
            \frac{\eta}{\tau}&&1
        \end{pmatrix}
    \end{pmatrix}=g_{\tau, \eta} P',
\end{align*}
where $$Y^{\frac{1}{2}}_{\tau,\eta}:=\begin{pmatrix}
    (\tau^2 y_1)^{\frac{1}{2}}&&0\\
    0&&(\eta^2 y_2)^{\frac{1}{2}}
\end{pmatrix},\hspace{1cm}g_{\tau,\eta}:=\begin{pmatrix}
        I&&X\\
        0&&I
    \end{pmatrix}
    \begin{pmatrix}
        Y_{\tau,\eta}^{\frac{1}{2}}&&0\\
        0&&{}^tY_{\tau,\eta}^{-\frac{1}{2}}
    \end{pmatrix},$$
and
$$P':=\begin{pmatrix}
        \begin{pmatrix}
            1&&-\frac{\eta}{\tau}\\
            0&&1
        \end{pmatrix}&&0\\
        0&&\begin{pmatrix}
            1&&0\\
            \frac{\eta}{\tau}&&1
        \end{pmatrix}
    \end{pmatrix}.$$
Then recalling the product law in the semidirect product  (equation (2.4) in \cite{Marklof-Welsh-II}), we have 
\begin{align*}
     \left|\Tilde{\Theta}_F^{(2)}(( h,g)(h_0,P))\right|&=\left|\Tilde{\Theta}_F^{(2)}( h',g_{\tau,\eta}P')\right|\\
     &=\left|\Tilde{\Theta}_F^{(2)}(( h',g_{\tau,\eta})(\operatorname{id},P'))\right|
\end{align*}
where $h'=hh_0^{g^{-1}}.$
Now recall the height function $D$ defined in equation (4.3) in \cite{Marklof-Welsh-II}, as well as Corollary 4.1 in \cite{Marklof-Welsh-II}. We have
\begin{align*}
     \left|\Tilde{\Theta}_f^{(2)}( h,g)(h_0,P)\right|\ll_f \left(D\!\left(\tilde \Gamma(( h',g_{\tau,\eta})(\operatorname{id},P')\right)\right)^{\frac{1}{4}}.
\end{align*}

 If $\tau\geq \eta$ then the matrix
\begin{equation*}
    \begin{pmatrix}
        1&&-\frac{\eta}{\tau}\\
        0&&1
    \end{pmatrix}
\end{equation*}
is within $1$ of the identity matrix uniformly over $\tau,\eta$. Hence, Lemma 4.4 in \cite{Marklof-Welsh-II} implies that
$$D\!\left(\tilde \Gamma(( h',g_{\tau,\eta})(\operatorname{id},P')\right)\ll D\!\left(\tilde \Gamma( h',g_{\tau,\eta})\right).$$
If $\eta>\tau$ then the matrix 
$$K:=\begin{pmatrix}
    1&&0\\
    0&&\frac{\tau}{\eta}
\end{pmatrix}$$
is within $1$ of the identity matrix, uniformly over $\tau,\eta$. Therefore Lemma 4.4 in \cite{Marklof-Welsh-II} implies that
\begin{align*}
   D\!\left(\tilde \Gamma((h',g_{\tau,\eta})(\operatorname{id},P'))\right)&\ll D\!\left(\tilde \Gamma \left((h',g_{\tau,\eta})(\operatorname{id},P')\left(\operatorname{id},\begin{pmatrix}
    K&&0\\
    0&&{}^tK^{-1}
\end{pmatrix}\right)\right)\right)\\
&=D\!\left(\tilde \Gamma \left((h',g_{\tau,\eta})\left(\operatorname{id},\begin{pmatrix}
    \hat K&&0\\
    0&&{}^t\hat K^{-1}
\end{pmatrix}\right)\right)\right),
\end{align*}
where 
$$\hat K=\begin{pmatrix}
    1&&-\frac{\tau}{\eta}\\
    0&&\frac{\tau}{\eta}
\end{pmatrix}$$
is within $1$ of the identity matrix uniformly in $\tau,\eta$. Therefore
$$D\!\left(\tilde \Gamma (h',g_{\tau,\eta})(\operatorname{id},P')\right)\ll D\!\left(\tilde \Gamma (h',g_{\tau,\eta})\right).$$
Putting together the cases when $\tau<\eta$ and $\tau\geq \eta$ gives the uniform estimate 
$$  \left|\Tilde{\Theta}_f^{(2)}(( h,g)(h_0,P))\right|\ll_f \left(D\!\left(\tilde \Gamma (h',g_{\tau,\eta})\right)\right)^{\frac{1}{4}}.$$
Note that $g_{\tau,\eta}$ is block diagonal. Let $\gamma_1,\gamma_2\in \operatorname{SL}(2,\mathbb Z)$ be so that $\gamma_1$ takes $x_1+i\tau^2 y_1$ to the fundamental domain in $\operatorname{SL}(2,\mathbb R)$ and so that $\gamma_2$ takes $x_2+i\eta^2 y_1$ to the fundamental domain in $\operatorname{SL}(2,\mathbb R)$. By Lemma \ref{lemma:embedding-of-fundamental-domains} and the assumption that $\sup_{\gamma \in \operatorname{SL}(2,\mathbb Z)}(\tau y_1)_\gamma \geq \sup_{\gamma \in \operatorname{SL}(2,\mathbb Z)}(\eta y_1)_\gamma,$ the resulting $(\gamma_1,\gamma_2) g_{\tau,\eta}$ is in the fundamental domain of $\operatorname{Sp}(2,\mathbb R).$ Therefore
\begin{align*}
    \left(D\!\left(\tilde \Gamma (h',g_{\tau,\eta})\right)\right)^{\frac{1}{4}}&\ll(\tau^2 y_1)_{\gamma_1}^{\frac{1}{4}}(\eta^2 y_2)_{\gamma_2}^{\frac{1}{4}}\leq H\!\left(x_1+i\tau^2 y_1\right)\,H\!\left(x_2+i\eta^2 y_2\right).
\end{align*}
\end{proof}

\begin{proposition}\label{prop:smooth-f-sheared-lambda-bound}
         Let $\lambda$ be a Borel probability measure on $\mathbb R$ which is absolutely continuous with respect to Lebesgue measure. Let $F\in \mathcal S(\mathbb R^2)$ and let $(h,g)$, $(h_0,P)$ be as in Lemma \ref{lemma:H-bound-for-sheared} with $-x_1=x_2=x$. Let $R>0$. Then  uniformly over all variables and $0<\tau^2 y_{1},\eta^2 y_{2}\leq 1$ we have
    \begin{equation*}
        \lambda\left(\left\{x\in \mathbb R: \left|\Tilde \Theta_{F}^{(2)}(( h,g)(h_0,P))\right|>R\right\}\right)\ll \frac{1}{(1+R)^2}.
    \end{equation*}
\end{proposition}
\begin{proof}
   Lemma \ref{lemma:H-bound-for-sheared} implies that that  
   \begin{align*}\lambda&\left(\left\{x\in \mathbb R: \left|\Tilde \Theta_{F}^{(2)}(( h,g)(h_0,P))\right|>R\right\}\right)\leq \lambda\left(\left\{x\in \mathbb R: H(x+i\tau^2 y_1)H(-x+i\eta^2 y_2)>\tfrac{R}{C}\right\}\right),\end{align*}
   where $C$ is the constant implied in \eqref{Bound-tildeTheta^2-in-terms-of-HH}A union bound implies that 
   \begin{align*} \lambda\bigg(\bigg\{x\in \mathbb R: \left|\Tilde \Theta_{F}^{(2)}(( h,g)(h_0,P))\right|>\tfrac{R}{C}\bigg\}\bigg)&\leq \lambda\left(\left\{x\in \mathbb R: H\!\left(-x+i\tau^2 y_1\right)>\sqrt{\tfrac{R}{C}}\right\}\right)\\
   &~+\lambda\left(\left\{x\in \mathbb R: H(x+i\eta^2 y_1)>\sqrt{\tfrac{R}{C}}\right\}\right).\end{align*}
   Lemma 3.19 in \cite{Cellarosi-Marklof}, along with the proof of Proposition 3.20 in \cite{Cellarosi-Marklof}, conclude.
\end{proof}
\begin{proposition}\label{prop:smooth-f-sheared-lambda-bound-random-y}
     Let $\lambda$ be a Borel probability measure on $\mathbb R$ which is absolutely continuous with respect to Lebesgue measure. Let $F\in \mathcal S(\mathbb R^2)$ and let $(h,g)$, $(h_0,P)$ be as in Lemma \ref{lemma:H-bound-for-sheared} with $-x_1=x_2=x$. Let $R>0$ and let $y_1(x),y_2(x),\boldsymbol \xi(x),$ be so that $0<y_0\leq \tau^2 y_{1}(x),\eta^2 y_{2}(x)\leq 1$. Let $\varepsilon>0$ be small. Then uniformly over all variables and over $y_0$ we have
    \begin{equation*}
        \lambda\left(\left\{x\in \mathbb R: \left|\Tilde \Theta_{F}^{(2)}(( h,g)(h_0,P))\right|>\frac{R}{(\tau^2 y_1(x)\eta^2 y_2(x))^{\varepsilon}}\right\}\right)\ll \frac{1}{(1+R)^2}.
    \end{equation*}
\end{proposition}
\begin{proof}
   Lemma \ref{lemma:H-bound-for-sheared} implies that that  
   \begin{align*} &\lambda\left(\left\{x\in \mathbb R: \left|\Tilde \Theta_{F}^{(2)}(( h,g)(h_0,P))\right|>\frac{R}{(\tau ^2 y_1(x)\eta^2 y_2(x))^{\varepsilon}}\right\}\right)\\
   &\leq \lambda\left(\left\{x\in \mathbb R: H(-x+i\tau^2 y_1(x))H(x+i\eta^2 y_2(x))>\frac{R}{C}\right\}\right),\end{align*}
   where $C$ is the constant implied by \eqref{Bound-tildeTheta^2-in-terms-of-HH}.
   A union bound implies that 
   \begin{align*} 
   &\lambda\left(\left\{x\in \mathbb R: \left|\Tilde \Theta_{F}^{(2)}(( h,g)(h_0,P))\right|>\frac{R}{(C(\tau^2 y_1(x)\eta^2 y_2(x))^{\varepsilon}})\right\}\right)\\
   &\leq\: \lambda\left(\left\{x\in \mathbb R: (\tau^2 y_1(x))^\varepsilon H(-x+i\tau^2 y_1(x))>\sqrt{\tfrac{R}{C}}\right\}\right)+\lambda\left(\left\{x\in \mathbb R: (\eta^2 y_2(x))^\varepsilon H(x+i\eta^2 y_2(x))>\sqrt{\tfrac{R}{C}}\right\}\right).\end{align*}
    Therefore Lemma \ref{lemma:integral-over-gamma-R4} and the proof of Lemma \ref{lemma:tail-bound-smooth-f-random-y} conclude.
\end{proof}
\begin{proposition}\label{prop:bound-shear-building-blocks}
    Let $\lambda$ be a Borel probability measure on $\mathbb R$ absolutely continuous with respect to Lebesgue measure.  Then for any $(g_1,g_2)\in  G\times  G$ with $-x_1=x_2=x$, we have uniformly in $y_1,y_2\leq 1$ that 
    \begin{align*}
    \lambda\left(\left\{x\in \mathbb R: \left| \Theta_{\mathfrak{C}_{0,0}}^{(2)}(g_1,g_2)\right|>R\right\}\right)&\ll \frac{1}{(1+R)^2}\\
    \lambda\left(\left\{x\in \mathbb R: \left| \Theta_{\mathfrak C_{1,1}}^{(2)}(g_1,g_2)\right|>R\right\}\right)&\ll \frac{1}{(1+R)^2}\\
 \lambda\left(\left\{x\in \mathbb R: \left| \Theta_{\mathfrak L_d}^{(2)}(g_1,g_2)\right|>R\right\}\right)&\ll \frac{1}{(1+R)^2}
    \end{align*}
\end{proposition}
\begin{proof}
    Using \eqref{def:TCor}, we write
    \begin{equation}\label{eq:F-decompose-as-sum-of-fg}
    \mathfrak{C}_{0,0}(w_1,w_2)=\sum_{j,k=0}^\infty f(2^j w_1)f(2^k(w_2-w_1)).
    \end{equation}
    where $f=f_{\frac{1}{12},\frac{1}{6},\frac{1}{3}}$ is smooth. 
    By linearity we have that 
    \begin{equation*}
    \Theta_{\mathfrak{C}_{0,0}}^{(2)}(g_1,g_2)=\sum_{j,k=0}^\infty \Theta_{(w_1,w_2)\mapsto f(2^jw_1)f(2^k (w_2-w_1))}^{(2)}(g_1,g_2).\end{equation*}
    Using Proposition 2.1 in \cite{Marklof-Welsh-I} and the mapping in Subsection \ref{subsection:theta-theta-tilde} we have that 
    \begin{equation*}
    \Theta_{\mathfrak{C}_{0,0}}^{(2)}(g_1,g_2)=\sum_{j,k=0}^\infty 2^{-\frac{j}{2}}2^{-\frac{k}{2}}\tilde \Theta_{(w_1,w_2)\mapsto f(w_1)f(w_2)}^{(2)}\left(\left(h(g_1,g_2),g(g_1,g_2)\right)(\operatorname{id},P_{j,k})\right),\end{equation*}
    where $P_{j,k}$ is as in Lemma \ref{lemma:H-bound-for-sheared} with $\tau=2^j$ and $\eta=2^k$. Let $J=\lceil \log_2(y_1^{-1/2})\rceil,$ and $K=\lceil \log_2(y_2^{-1/2})\rceil$. In this case, we have that 
     \begin{equation}\label{eq:decomposition-for-theta-2-F}
     \Theta_{\mathfrak{C}_{0,0}}^{(2)}(g_1,g_2)=\sum_{j=0}^J\sum_{k=0}^K 2^{-\frac{j}{2}}2^{-\frac{k}{2}}\tilde \Theta_{(w_1,w_2)\mapsto f(w_1)f(w_2)}^{(2)}\left(\left(h(g_1,g_2),g(g_1,g_2)\right)(\operatorname{id},P_{j,k})\right).\end{equation}
     Let $C_J,C_K>0$ be such that $\delta_j=\frac{C_J}{j^2}$ and $\delta'_k=\frac{C_K}{k^2}$ satisfy $\sum_{j=0}^J \delta_j=\sum_{k=0}^K \delta'_k=1$. In this case, we may use a union bound and Proposition \ref{prop:smooth-f-sheared-lambda-bound} to get that
     \begin{align}\label{eq:argument-in-sheared-tail-bound}
         \lambda&\left(\left\{x\in \mathbb R:\:\left|\Theta_{\mathfrak{C}_{0,0}}^{(2)}(g_1,g_2)\right|>R\right\}\right)\\
         &\leq \sum_{j=0}^{J}\sum_{k=1}^{K}\lambda\bigg(\bigg\{x:\in \mathbb R: |\tilde \Theta_{(w_1,w_2)\mapsto f(w_1)f(w_2)}^{(2)}(h(g_1,g_2),g(g_1,g_2)P_{j,k})|>\frac{R 2^{\frac{j}{2}} 2^{\frac{k}{2}}}{C_JC_K}\bigg\}\bigg)\nonumber\\
         &\ll  \sum_{j=0}^{J}\sum_{k=0}^{K} \frac{1}{(1+2^{\frac{j}{2}}2^{\frac{k}{2}}\delta_j\delta_k' R)^2}\nonumber \\
        &\ll \frac{1}{(1+R)^2} \sum_{j=0}^{J}\sum_{k=1}^{K} 2^{-j}2^{-k}\delta_j^2 \delta_k'^2\nonumber\\
        &\ll \frac{1}{(1+R)^2}.\label{eq:argument-in-sheared-tail-bound-end}
     \end{align}
    From \eqref{def:TLine} and Remark \ref{remark:where-line-segments-come-from} we write $\mathfrak{L}_d(w_1,w_2)$ as a sum of the following three functions: 
    \begin{align}
    \mathfrak{L}_{d,1}(w_1,w_2):=&F_{\frac{1}{12},\frac{1}{6},\frac{1}{3}}(\tfrac{1}{2}-w_1)F_{\frac{1}{24},\frac{1}{12},\frac{1}{6}}(w_2-w_1)=\sum_{j,k=0}^{\infty}f_{1}(2^j(\tfrac{1}{2}-w_1))f_{2}(2^k(w_2-w_1)),\label{L_d-piece1}\\
\mathfrak{L}_{d,2}(w_1,w_2):=&T_{\operatorname{Segm}}(w_1,w_2-w_1)\label{L_d-piece2}\\
    \mathfrak{L}_{d,3}(w_1,w_2):=&F_{\frac{1}{12},\frac{1}{6},\frac{1}{3}}(-\tfrac{1}{2}+w_1)F_{\frac{1}{24},\frac{1}{12},\frac{1}{6}}(w_2-w_1)=\sum_{j,k=0}^{\infty}f_{1}(2^j(-\tfrac{1}{2}+w_1))f_{2}(2^k(w_2-w_1)),\label{L_d-piece3}
    \end{align}
    where $f_1=f_{\frac{1}{12},\frac{1}{6},\frac{1}{3}}$ and $f_2=f_{\frac{1}{24},\frac{1}{12},\frac{1}{6}}$ are smooth. Let us first focus on \eqref{L_d-piece3}.
\begin{align}
    \Theta^{(2)}_{\mathfrak{L}_{d,3}}(g_1,g_2)&=\sum_{j,k=0}^{\infty} \Theta^{(2)}_{(w_1,w_2)\mapsto f_1(2^j(-\frac{1}{2}+w_1))f_2(2^k(w_2-w_1))}(g_1,g_2)\nonumber\\
    &=\sum_{j,k=0}^{\infty}2^{-\frac{j}{2}}2^{-\frac{k}{2}}\tilde\Theta^{(2)}_{(w_1,w_2)\mapsto f_1(w_1)f_2(w_2)}\left((h(g_1,g_2),g(g_1,g_2)),(u,P_{j,k})\right),\label{eq:decomposition-for-theta-L_d_3}
\end{align}
where $u=((-\tfrac{1}{2},-\tfrac{1}{2}), (0,0),0)\in\mathbb{H}(\R^2)$ and $P_{j,k}$ is as above. 
Letting $J$, $K$, $(\delta_j)_{1\leq j\leq J}$, $(\delta_k)_{1\leq k\leq K}$, $C_J$, and $C_K$  as before, a union bound and Proposition  \ref{prop:smooth-f-sheared-lambda-bound} gives

    \begin{align*}
         \lambda&\left(\left\{x\in \mathbb R:\:\left|\Theta_{\mathfrak{L}_{d,3}}^{(2)}(g_1,g_2)\right|>\tfrac{R}{3}\right\}\right)\\
         &\leq \sum_{j=0}^{J}\sum_{k=1}^{K}\lambda\bigg(\bigg\{x:\in \mathbb R: |\tilde \Theta_{(w_1,w_2)\mapsto f_1(w_1)f_2(w_2)}^{(2)}\left(\left(h(g_1,g_2),g(g_1,g_2)\right)(u,P_{j,k})\right)|>\frac{R 2^{\frac{j}{2}} 2^{\frac{k}{2}}}{3C_JC_K}\bigg\}\bigg)\nonumber\\
         &\ll  \sum_{j=0}^{J}\sum_{k=0}^{K} \frac{1}{(1+2^{\frac{j}{2}}2^{\frac{k}{2}}\delta_j\delta_k' R/3)^2}\nonumber \\
        &\ll \frac{1}{(1+R)^2}.
     \end{align*}

The theta function corresponding to \eqref{L_d-piece1} is handled after by observing that $f_1(2^j(\frac{1}{2}-w_1))=f_1^-(2^j(-\frac{1}{2}+w_1))$, where $f_1^-(w)=f_1(-w)$, giving the tail bound $\lambda\left(\left\{x\in \mathbb R:\:\left|\Theta_{\mathfrak{L}_{d,1}}^{(2)}(g_1,g_2)\right|>\tfrac{R}{3}\right\}\right)\ll(1+R)^{-2}$. Finally, the first part of Lemma \ref{lemma:lines-dont-matter} provides the 
the tail bound $\lambda\left(\left\{x\in \mathbb R:\:\left|\Theta_{\mathfrak{L}_{d,2}}^{(2)}(g_1,g_2)\right|>\tfrac{R}{3}\right\}\right)\ll(1+R)^{-4}$. A union bound gives the desired tail bound for $\Theta_{\mathfrak{L}_d}^{(2)}$. The tail bound for $\Theta_{\mathfrak{C}_{1,1}}^{(2)}$ is similar to that for $\Theta_{\mathfrak{L}_{d,1}}^{(2)}$, with slight changes to the group elements $(u,P_{j,k})\in\mathbb{H}(\R^2)\rtimes \operatorname{Sp}(2,\R)$

    %
\end{proof}

\begin{proposition}\label{prop:bound-shear-building-blocks-random-y}
 Let $\lambda$ be a Borel probability measure on $\mathbb R$ which is absolutely continuous with respect to Lebesgue measure. Let $f\in \mathcal S(\mathbb R^2)$ and let $(h,g)$, $P$ be as in Lemma \ref{lemma:H-bound-for-sheared} with $-x_1=x_2=x$. Let $R>0$ and let $y_1(x),y_2(x), \vecxi_1(x), \vecxi_2(x)$ be so that $0<y_0\leq \tau^2 y_{1}(x),\eta^2 y_{2}(x)\leq 1$. Let $\varepsilon>0$ be small. Then uniformly over all variables and over $y_0$ we have
     \begin{align*}
    \lambda\left(\left\{x\in \mathbb R: \left| \Theta_{\mathfrak{C}_{0,0}}^{(2)}(g_1,g_2)\right|>
\frac{R}{(\tau^2 y_1(x)\eta^2 y_2(x))^{\varepsilon}}\right\}\right)&\ll \frac{1}{(1+R)^2}\\
    \lambda\left(\left\{x\in \mathbb R: \left| \Theta_{\mathfrak C_{1,1}}^{(2)}(g_1,g_2)\right|>\frac{R}{(\tau^2 y_1(x)\eta^2 y_2(x))^{\varepsilon}}\right\}\right)&\ll \frac{1}{(1+R)^2}\\
 \lambda\left(\left\{x\in \mathbb R: \left| \Theta_{\mathfrak L_d}^{(2)}(g_1,g_2)\right|>\frac{R}{(\tau^2 y_1(x))\eta^2 y_2(x))^{\varepsilon}}\right\}\right)&\ll \frac{1}{(1+R)^2}
    \end{align*}
\end{proposition}
\begin{proof}
For each of the two corner functions, the proof is a slightly more elaborate version of the proof of Proposition  \ref{proposition:tail-with-random-y}, involving double sums instead of single sums and $\tilde \Theta^{(2)}$ instead of $\Theta$.
    We start as in the proof of Proposition \ref{prop:bound-shear-building-blocks}. Each theta function is written as a sum of theta functions of the form $\tilde \Theta^{(2)}_{F}$, where $F(w_1,w_2)=f_1(w_1)f_2(w_2)$, as in \eqref{eq:decomposition-for-theta-2-F}.
    We use a union bound over $j,k$ and for each term Lemma \ref{lemma:H-bound-for-sheared} relates tail bounds for such theta functions to tail bounds for product of height functions. A further union bound concludes. For the line function, we first split it a sum of three functions as in \eqref{L_d-piece1}--\eqref{L_d-piece3}. The argument outlined above applies for $\Theta^{(2)}_{\mathfrak{L}_{d,\ell}}$ for $\ell=1,3$, while for $\Theta_{\mathfrak{L}_{d,2}}$ we use equation \eqref{eq:lines-dont-matter-random-y} in Lemma \ref{lemma:lines-dont-matter}. Finally, a union bound concludes.
\end{proof}
\subsection{\texorpdfstring{Tail estimates for $\Theta_{T_{(0,1]}}^{(2)}$}{}}
We are finally able to prove tail bounds for $\Theta_{T_{(0,1]}}^{(2)}.$
\begin{proposition}\label{prop:final-tightness-for-T_(0,1]}
       Let $\lambda$ be a Borel probability measure on $\mathbb R$ absolutely continuous with respect to Lebesgue measure.  Then for any $(g_1,g_2)\in  G\times  G$ with $-x_1=x_2=x$, we have uniformly in $y_1,y_2\leq 1$ that 
    \begin{equation}\label{eq:tail-bound-T}
        \lambda\left(\left\{x\in \mathbb R: \left|\Theta_{T_{(0,1]}}^{(2)}(g_1,g_2)\right|>R\right\}\right)\ll \frac{1}{(1+R)^2}.
    \end{equation}
\end{proposition}
\begin{proof}
    By linearity, we use \eqref{eq:decomposition-of-triangle} to decompose $\Theta^{(2)}_{T_{(0,1]}}=\sum_{\ell=1}^8\Theta^{(2)}_{F_\ell}$, 
    where $$(F_1,\ldots, F_8)=(\mathbf{1}_{(0,1)\times \{1\}},\mathfrak{C}_{0,0},\mathfrak{C}_{0,1},\mathfrak{C}_{1,1}, \mathfrak{L}_h,\mathfrak{L}_v,\mathfrak{L}_d,\mathfrak{F}_{\text{smooth}}).$$
    We use the union bound
    \begin{align*}
          \lambda\left(\left\{x\in \mathbb R: \left|\Theta_{T_{(0,1]}}^{(2)}(g_1,g_2)\right|>R\right\}\right)\leq\sum_{\ell=1}^8   \lambda\left(\left\{x\in \mathbb R: \left|\Theta_{F_\ell}^{(2)}(g_1,g_2)\right|>\tfrac{R}{8}\right\}\right)
    \end{align*}
    For $\ell=1$ we apply the first part of Lemma \ref{lemma:lines-dont-matter}; for $\ell=2,4,7$ we apply Proposition \ref{prop:bound-shear-building-blocks}; for $\ell=3,5,6$ we apply  Proposition \ref{prop:bound-unshear-building-blocks}; for $\ell=8$ we apply Proposition \eqref{prop:bound-for-second-order-smooth}.
We obtain the bound $\ll (1+R)^{-4}+7(1+R)^{-2}\ll(1+R)^{-2}$, as claimed.
\end{proof}

The following proposition is the key technical ingredient needed in the proof of Theorem  \ref{theorem:tightness-second-order}.
\begin{proposition}\label{prop:tail-random-y-second-order}
    Let $\lambda$ be a Borel probability measure on $\mathbb R$ which is absolutely continuous with respect to Lebesgue measure. Let $R>0$ and let $\varepsilon>0$ be small. Then for all $y_1(x),y_2(x)$ with $0<y_0\leq y_1(x),y_2(x)\leq 1$ and for all $-x_1=x_2=x$,  $ \vecxi_1(x), \zeta_1(x),\vecxi_2(x), \zeta_2(x)$ we have the bound
    \begin{equation}\label{eq:tail-bound-T-random-y}
        \lambda\left(\left\{x\in \mathbb R: \left|\Theta_{ T_{(0,1)}}^2(g_1,g_2)\right|>\frac{R}{(y_1(x)y_2(x))^\varepsilon}\right\}\right)\ll \frac{1}{(1+R)^2}
    \end{equation}
    independent of $y_0$.
\end{proposition}
\begin{proof}
    We decompose $\Theta^2_{T_{(0,1]}}$ and apply a union  as in the proof of Proposition \ref{prop:final-tightness-for-T_(0,1]}. We get
    \begin{align*}
          \lambda\left(\left\{x\in \mathbb R: \left|\Theta_{T_{(0,1]}}^{(2)}(g_1,g_2)\right|>\frac{R}{(y_1(x)y_2(x))^\varepsilon}\right\}\right)&\leq  \sum_{\ell=1}^8 \lambda\left(\left\{x\in \mathbb R: \left|\Theta_{F_\ell}^{(2)}(g_1,g_2)\right|>\frac{R}{8(y_1(x)y_2(x))^\varepsilon}\right\}\right).
    \end{align*}
   We then conclude by combining the bounds obtained from \eqref{eq:lines-dont-matter-random-y} and Propositions \ref{prop:bound-shear-building-blocks-random-y}, \ref{prop:bound-unshear-building-blocks-random-y}, and \ref{prop:bound-for-second-order-smooth-random-y}.
\end{proof}

\begin{remark}
    For the indicator $T_{(s,t]}$ of a general triangle, we may write 
\begin{align}\label{Theta-2-general-triangle}
    &\Theta^{(2)}_{T_{(s,t]}}\left(\left(x_1+iy_1;0,\begin{pmatrix}
        \xi_{1,1}\\
        \xi_{1,2}
\end{pmatrix},\zeta_1\right),\left(x_2+iy_2;0,\begin{pmatrix}
        \xi_{2,1}\\
        \xi_{2,2}
    \end{pmatrix},\zeta_2\right)\right)\nonumber\\
    &= (t-s)\Theta^{(2)}_{T_{(0,1]}}\bigg(\bigg(x_1+i\frac{y_1}{(t-s)^2};0,\begin{pmatrix}
        \xi_{1,1}\\
        \xi_{1,2}+s
\end{pmatrix},\zeta_1\bigg),\left(x_2+i\frac{y_2}{(t-s)^2};0,\begin{pmatrix}
        \xi_{2,1}\\
        \xi_{2,2}+s
    \end{pmatrix},\zeta_2\right)\bigg).
\end{align}
\end{remark}

\begin{proposition}\label{prop:tail-bound-second-order-Holder}
 Let $R>0$, let $0<\gamma<\frac{1}{2}$, and let $\lambda$ be a probability measure on $\mathbb R$ absolutely continuous with respect to Lebesgue measure. Let $(g_1,g_2)\in G\times G$ be with $-x_1=x_2=x$. Then uniformly over $y_1,y_2\leq $ we have the tail bound
    \begin{equation*}
        \lambda\left(\left\{x\in \mathbb R:\sup_{s\neq t\in [0,1]}\frac{| \Theta_{T(s,t)}^{(2)}(g_1,g_2)|}{|t-s|^{2\gamma}}>R\right\}\right)\ll \frac{1}{(1+R)^2},
    \end{equation*}
    uniformly over $N$, but where the implied constant is dependent on $\gamma$.    
\end{proposition}
\begin{proof}
The proof is the same as the proof of Proposition \ref{prop:1st-order-tightness}, using equation \eqref{Theta-2-general-triangle} and Proposition \ref{prop:tail-random-y-second-order} instead of Proposition \ref{proposition:tail-with-random-y}.
\end{proof}

\begin{proof}[Proof of Theorem \ref{theorem:tightness-second-order}]
Let $\gamma\in(0,\frac{1}{2})$ and pick $\gamma'\in (\gamma, \frac{1}{2})$. Recall from \cite{Friz-Hairer-Book}, Exercise 2.12, that the embedding $\mathscr C_g^{\gamma'}\hookrightarrow \mathscr C_g^{\gamma}$ is compact. Propositions \ref{prop:tail-bound-Holder} and \ref{prop:tail-bound-second-order-Holder} conclude (see also Proposition \ref{prop:1st-order-tightness}).
\end{proof}

\section{\texorpdfstring{Limit theorems and Convergence of Finite Dimensional Distributions of $\mathbb X_N$}{}}\label{section:fdds}
Let  $\hat\mu_{\GamG\times\GamG}$ be the probability measure on $\GamG\times\GamG$ given by
    \begin{align}
        d\hat\mu(\Gamma g_1,\Gamma g_2)=&
        \frac{\de x_1\de y_1 \de\phi_1\de\xi_{11}\de \xi_{12}\de\zeta_1\de x_2\de y_2\de \phi_2\de\xi_{21}\de \xi_{22}\de\zeta_2}{y_1^2y_2^2}\nonumber
        \\&\delta(x_1+x_2)\delta(y_1-y_2)\delta(\phi_1+\phi_2)\delta(\xi_{11}+\xi_{21})\delta(\xi_{12}-\xi_{22})\delta(\zeta_1+\zeta_2)\label{def-hat-mu-2}
    \end{align}
    where $g_j=(x_j+iy_j,\phi_j;
    \sve{\xi_{j1}}{\xi_{j2}},\zeta_j)$ for $j=1,2$. Note that $\GamG\times\GamG$ is 12-dimensional and $\hat\mu$ is supported on the 6-dimensional submanifold of $\GamG$ given by the equations $x_1=-x_2$, $y_1=y_2$, $\phi_1=-\phi_2$, $\xi_{11}=-\xi_{21}$, $\xi_{12}=\xi_{22}$, $\zeta_1=-\zeta_2$. 
\begin{theorem}[Joint equidistribution theorem]\label{thm-joint-equidistribution}
    Let $\vecxi\notin\Q^2$, and let $\lambda$ be a Borel probability measure on $\R$ which is absolutely continuous with respect to the Lebesgue measure. For every bounded continuous function $F:\GamG\times\GamG\to\R$ we have 
    \begin{align*}
    &\lim_{\tau\ti}\int_{\R}F\!\left(\Gamma \left(I;\sve{\xi_1}{\xi_2},0\right)\Psi^{x}\Phi^{\tau},\Gamma \left(I;\sve{-\xi_1}{\xi_2},0\right)\Psi^{-x}\Phi^{\tau}\right)\de\lambda(x)=\\
    &=\int_{\GamG\times\GamG}F(g_1,g_2)\,\de\hat\mu_{\GamG\times\GamG}(g_1,g_2).
    \end{align*}
\end{theorem}
\begin{proof}
    Since $\vecxi\notin\Q^2$, we know that, separately, each of the horocycle lifts $\Gamma \left(I;\sve{\pm\xi_1}{\xi_2},0\right)\Psi^{\pm x}\Phi^{\tau}$
    becomes equidistributed with respect to the Haar measure $\mu_{\GamG}$ as $t\ti$. We claim that the joint limiting distribution is given by the probability measure $\hat\mu$ defined in \eqref{def-hat-mu-2}, whose marginals are both equal to $\mu_{\GamG}$. 
    
    Let $x,\tau$ be given, and suppose $\gamma_{+}=\gamma_{+}(x,\tau)\in\Gamma$ is the unique lattice element such that 
    \begin{align}
        \gamma_{+}\left(I;\sve{\xi_1}{\xi_2},0\right)\Psi^{x}\Phi^{\tau}\in \mathscr{F}_\Gamma.
    \end{align}
    We need to find $\gamma_{-}=\gamma_{-}(x,\tau)$ such that 
    \begin{align}
        \gamma_{-}\left(I;\sve{-\xi_1}{\xi_2},0\right)\Psi^{-x}\Phi^{\tau}\in \mathscr{F}_\Gamma.
    \end{align}
    If we write $\gamma_{+}$ as a word in the generators, i.e. $\gamma_{+}=\gamma_+^{(1)}\gamma_+^{(2)}\cdots\gamma_+^{(L)}$ with $\gamma_{+}^{(i)}\in\{\gamma_1^{\pm1},\ldots,\gamma_5^{\pm1}\}$, then we claim that $\gamma_{-}=\gamma_{-}^{(1)}\gamma_{-}^{(2)}\cdots\gamma_{-}^{(L)}$ where  $\gamma_{-}^{(i)}=(\gamma_{+}^{(i)})^{-1}$ if $\gamma_{+}^{(i)}\in\{\gamma_1^{\pm1},\gamma_2^{\pm1},\gamma_3^{\pm1},\gamma_5^{\pm1}\}$ and $\gamma_{-}^{(i)}=\gamma_{+}^{(i)}$ if $\gamma_{+}^{(i)}=\gamma_4^{\pm1}$.  Note that in coordinates \eqref{coordinates-on-G} we have $\left(I;\sve{\pm\xi_1}{\xi_2},0\right)\Psi^{\pm x}\Phi^{\tau}=\left(\pm x+ie^{-\tau},0;\sve{\pm \xi_1}{\xi_2},0\right)$. 
    To show our claim, it is enough to check each generator of $\Gamma$. Recall \eqref{generators-1}-\eqref{generators-5} and  compute
    \begin{align*}
            &\gamma_1\left(x+iy,\phi;\sve{\xi_1}{\xi_2},\zeta\right)=\left(\frac{-x}{x^2+y^2}+i\frac{y}{x^2+y^2}
        ,\phi+\arg(x+iy);\sve{-\xi_2}{\xi_1},\zeta+\tfrac{1}{8}\right)\\
        &\gamma_1^{-1}\left(-x+iy,-\phi;\sve{-\xi_1}{\xi_2},-\zeta\right)\\
        &~=\left(\frac{x}{x^2+y^2}+i\frac{y}{x^2+y^2}
        ,-\phi-\arg\!\left(\frac{1}{-(-x+iy)}\right);\sve{\xi_2}{\xi_1},-\zeta-\tfrac{1}{8}\right),\nonumber\\
        &\gamma_2\left(x+iy,\phi;\sve{\xi_1}{\xi_2},\zeta\right)=\left(x+1+iy
        ,\phi;\sve{1/2+\xi_1+\xi_2}{\xi_2},\zeta+\tha\xi_2\right)\\ 
        &\gamma_2^{-1}\left(-x+iy,-\phi;\sve{-\xi_1}{\xi_2},\zeta\right)=\left(-x-1+iy,-\phi;\sve{-1/2-\xi_1-\xi_2}{\xi_2},-\zeta-\tha\xi_2\right). \\
        &\gamma_3\left(x+iy,\phi;\sve{\xi_1}{\xi_2},\zeta\right)=\left(x+iy,\phi;\sve{\xi_1+1}{\xi_2},\zeta+\tha\xi_2\right)\\
        &\gamma_{3}^{-1}\left(-x+iy,-\phi;\sve{-\xi_1}{\xi_2},\zeta\right)=\left(-x+iy,-\phi;\sve{-\xi_1-1}{\xi_2},-\zeta-\tha\xi_2\right),\\
        &\gamma_4\left(x+iy,\phi;\sve{\xi_1}{\xi_2},\zeta\right)=\left(x+iy,\phi;\sve{\xi_1}{\xi_2+1},\zeta-\tha\xi_1\right)\\
        &\gamma_{4}\left(-x+iy,-\phi;\sve{-\xi_1}{\xi_2},-\zeta\right)=\left(-x+iy,-\phi;\sve{-\xi_1}{\xi_2+1},-\zeta+\tha\xi_1\right),\\
        &\gamma_5\left(x+iy,\phi;\sve{\xi_1}{\xi_2},\zeta\right)=\left(x+iy,\phi;\sve{\xi_1}{\xi_2},\zeta+1\right)\\
        &\gamma_{5}^{-1}\left(-x+iy,-\phi;\sve{-\xi_1}{\xi_2},-\zeta\right)=\left(-x+iy,-\phi;\sve{-\xi_2}{\xi_1},-\zeta-1\right),
    \end{align*}
    and similarly for the action of $\gamma_1^{-1},\ldots,\gamma_5^{-1}$ on $\left(x+iy,\phi;\sve{\xi_1}{\xi_2},\zeta\right)$.
    We see that when we apply each of the generators $\gamma\in\{\gamma_1^{\pm1},\gamma_2^{\pm1},\gamma_3^{\pm1},\gamma_5^{\pm1}\}$ to $\left(x+iy,\phi;\sve{\xi_1}{\xi_2},\zeta\right)$ (in order to translate it to the fundamental domain $\mathscr{F}_\Gamma$) and the corresponding generator $\gamma^{-1}$ to $\left(-x+iy,-\phi;\sve{-\xi_1}{\xi_2},-\zeta\right)$ we obtain two points\footnote{We need to omit sets of $\lambda$-measure zero corresponding to the boundary of the fundamental domain $\mathscr{F}_{\Gamma}$. For instance, under the action of $\gamma_1$ we could get $\frac{-x}{x^2+y^2}=\pm\ha$, or $\phi+\arg(x+iy)=\pm\tfrac{\pi}{2}$.} in the fundamental domain whose $(x,\phi,\xi_1,\zeta)$-coordinates are opposite and whose $(y,\xi_2)$-coordinates are the same.  The same happens when we apply $\gamma_4$ to both points $\left(\pm x+iy,\pm\phi;\sve{\pm\xi_1}{\xi_2},\pm\zeta\right)$. Our claim is therefore proven.

    For every $\tau$  the group elements $\gamma_{\pm}(x,\tau)$ needed to translate $(\pm x+ie^{-\tau},0;\sve{\pm \alpha}{-\beta},0)$  into the fundamental domain $\mathscr{F}_\Gamma$ will be related as discussed above. Therefore the point $\gamma_{-}\left(I;\sve{-\xi_1}{\xi_2},0\right)\Psi^{-x}\Phi^{\tau}$ is determined by the point $\gamma_{+}\left(I;\sve{\xi_1}{\xi_2},0\right)\Psi^{x}\Phi^{\tau}$ via the transformation $\mathscr{F}_\Gamma\to\mathscr{F}_\Gamma$, 
    \begin{align*}
    \left(x+iy,\phi;\sve{\xi}{\xi_2},\zeta\right)\mapsto\left(-x+iy,-\phi;\sve{-\xi_1}{\xi_2},-\zeta\right).
\end{align*}
Since equidistribution of the each horocycle lift as $\tau\ti$ occurs with respect to the measure $\mu_{\GamG}$, we obtain that the joint equidistribution holds with respect to the measure $\hat\mu_{\GamG\times\GamG}$.
f
\end{proof}
\begin{theorem}\label{thm-joint-equidistribution-family}
    Let $\lambda$ be a Borel probability  measure on $\R$ which is absolutely continuous with respect to the Lebesgue measure. Let $F:\R\times\GamG\times\GamG\to\R$ be a bounded continuous function, end let  $F_\tau:\R\times\GamG\times\GamG\to\R$ be a family of uniformly bounded, continuous functions so that $B_\tau\to B$ uniformly on compacta as $\tau\ti$. Let $\vecxi\notin\Q^2$. Then we have 
    \begin{align*}
        &\lim_{\tau\ti}\int_{\R}F_\tau\!\left(x, \left([I,0];\sve{\xi_1}{\xi_2},0\right)\Psi^{x}\Phi^\tau,\left([I,0];\sve{-\xi_1}{\xi_2},0\right)\Psi^{-x}\Phi^\tau\right)\de\lambda(x)=\\
        &=\int_{\R\times\GamG\times\GamG} F(x,\Gamma g_1,\Gamma g_2)\,\de\lambda(x)\,\de\hat\mu_{\GamG\times\GamG}(g_1,g_2)
    \end{align*}
\end{theorem}
\begin{proof}
    Using a standard argument (see e.g. Theorem 5.3 in \cite{Marklof-Strombergsson-Annals-2010}), the result follows directly from  Theorem \ref{thm-joint-equidistribution}.
\end{proof}

If $\underline{f}=(f_1,\ldots,f_k)\in(\LtR)^k$ and $\underline{T}={T_1,\ldots,T_k}\in(\mathrm{L}^2(\R^2))^k$, we use the notations 
\begin{align*}
    \Theta_{\underline{f}}(g)&=\left(\Theta_{f_1}(g),\ldots,\Theta_{f_k}(g)\right),\\
    \Theta^{(2)}_{\underline{T}}(g_1,g_2)&=\left(\Theta^{(2)}_{T_1}(g_1,g_2),\ldots,\Theta^{(2)}_{T_k}(g_1,g_2)\right).
\end{align*}
For fixed $\alpha,\beta$ we will also write the elements of $G$
\begin{align*}
    g^\tau_{\alpha,\beta}(x)&=\left(I;\sve{\alpha+\beta x}{0},0\right)\Psi^x\Phi^\tau=\left(x+i e^{-\tau},0;\sve{\alpha+\beta x}{0},0\right)\\
    \overline{g}^\tau_{\alpha,\beta}(x)&=\left(I;\sve{-\alpha-\beta x}{0},0\right)\Psi^{-x}\Phi^\tau=\left(-x+i e^{-\tau},0;\sve{-\alpha-\beta x}{0},0\right),
    \end{align*} where $x,\tau\in\R$.

\begin{theorem}\label{thm-limiting-distribution-of-finitely-regular-thetas}
    Fix $m\geq1$ and fix $2m$ regular functions $f_1,\ldots, f_m\in\SEtaR$ with $\eta>1$ and $T_1,\ldots,T_m\in\mathcal{S}_{\eta_1,\eta_2}(\R^2)$ with $\eta_1,\eta_2>1$.
    Let $\lambda$ be a Borel probability  measure on $\R$ which is absolutely continuous with respect to the Lebesgue measure. Let $x$ be randomly distributed according to $\lambda$. Let $(\alpha,\beta)\notin\Q^2$. Then the $\C^{2m}$-valued random variables
    \begin{align*}
\left(\Theta_{\underline{f}}\left(\Gamma g^\tau_{\alpha,\beta}(x)\right),\Theta_{\underline{T}}^{(2)}\left(\Gamma \overline{g}^\tau_{\alpha,\beta}(x),\Gamma g^\tau_{\alpha,\beta}(x)\right)\right)_{\ell=1}^m
    \end{align*}
    have a limit in distribution as $\tau\ti$, where  $\underline f=(f_1,\ldots,f_m)$ and  $\underline{T}=(T_1,\ldots,T_m)$. More precisely,  there exists a probability measure $\mathbb{P}^{\mathrm{reg}}_{\underline{f},\underline{T}}$
        on $\C^{2m}$ such that for every bounded continuous function $B:\C^{2m}\to\R$, we have
    \begin{align}
\lim_{\tau\ti}\int_{\R}B\Big{(}\Theta_{\underline{f}}\left(\Gamma g^\tau_{\alpha,\beta}(x)\right),\Theta_{\underline{T}}^{(2)}\left(\Gamma\overline g^\tau_{\alpha,\beta}(x),\Gamma g^\tau_{\alpha,\beta}(x)\right)
\Big{)}\de\lambda(x)=
\int_{\C^{2m}}B\,\de\mathbb{P}^{\mathrm{reg}}_{\underline{f},\underline{T}}.\label{statement-joint-equidistribution-2m-regular-functions-1}
    \end{align}
    The limiting law $\mathbb{P}^{\mathrm{reg}}_{\underline{f},\underline{T}}$ does not depend on $\alpha,\beta, \lambda$. 
\end{theorem}
\begin{proof}
Using the group law on $G$, we can write 
\begin{align}\label{transforming-the-horocycle-lifts}
    &\left(I;\sve{\pm\alpha\pm\beta x}{0},0\right)\Psi^{\pm x}\Phi^{\tau}=\left(I;\sve{\pm\alpha}{-\beta},0\right)\Psi^{\pm x}\Phi^{\tau}\left(I;\sve{0}{e^{-\tau/2}\beta},\mp \tha\beta(\alpha+x\beta)\right).
\end{align}
Now we can use the representation-theoretical definitions of $\Theta_{f_\ell}$ and $\Theta^{(2)}_{T_\ell}$ in terms of  $R:G\to\mathcal{U}(\LtR)$ and of $R^{(2)}:G\times G\to \mathcal{U}(\mathrm{L}^2(\R^2))$ respectively. For each $1\leq\ell\leq m$, the identity \eqref{transforming-the-horocycle-lifts} implies
\begin{align}\label{Theta_f_with_transformed_lift-1}
\Theta_{f_\ell}\!\left(\Gamma g^\tau_{\alpha,\beta}(x)\right)=\Theta_{\tilde{f}_{\ell,\tau,x}}\!\left((I;\sve{\alpha}{-\beta},0)\Psi^{x}\Phi^{\tau}\right)
\end{align}
where $\tilde{f}_{\ell,x,\tau}=R\!\left(\left(I;\sve{0}{e^{-\tau/2}\beta},\tha\beta(\alpha+x\beta)\right)\right)f_\ell\in\SEtaR$ and similarly
\begin{align}
    &\Theta^{(2)}_{T_\ell}\!\left(\Gamma\overline g^\tau_{\alpha,\beta}, \Gamma g^\tau_{\alpha,\beta}\right)=
\Theta^{(2)}_{\tilde{T}_{\ell;\tau}}\!\left(\Gamma (I;\sve{-\alpha}{-\beta},0)\Psi^{-x}\Phi^{2\log N},\Gamma (I;\sve{\alpha}{-\beta},0)\Psi^{x}\Phi^{2\log N} \right),\label{Theta_f_with_transformed_lift-2}
\end{align}
where $\tilde T_{\ell;\tau}=R^{(2)}\!\left(\left(I;\sve{0}{e^{-\tau/2}\beta},-\tha\beta(\alpha+x\beta)\right),\left(I;\sve{0}{e^{-\tau/2}\beta}, \tha\beta(\alpha+x\beta)\right)\right){T}_\ell\in\mathcal{S}_{\eta_1,\eta_2}(\R^2)$. 
Using \eqref{representation-Heisenberg-2} and \eqref{representation-Heisenberg-3} we can write
\begin{align}
\tilde{f}_{\ell,x;\tau}(w)&=\e{\tha\beta(
\alpha+x\beta)}f_\ell(w-e^{-\tau/2}\beta),\label{function-tilde_f_tau_x}\\
    \tilde{T}_{\ell;\tau}(w_1,w_2)&=T\!\left(w_1-e^{-\tau/2}\beta,w_2-e^{-\tau/2}\beta\right).\label{function-tilde_T_tau}
\end{align} 
Let us set 
\begin{align}F_\tau(x,\Gamma g_1,\Gamma g_2)&=B\!\left(\Theta_{(\tilde{f}_{1,x;\tau},\ldots,\tilde{f}_{m,x;\tau})}(\Gamma g_1),\Theta^{(2)}_{(\tilde{T}_{1;\tau},\ldots, \tilde{T}_{m;\tau})}(\Gamma g_2,\Gamma g_1)\right)\label{def-F_tau_x-in-limit-theorem}\\
    F(x,\Gamma g_1,\Gamma g_2)&=B\!\left(\Theta_{(\tilde{f}_{1,x},\ldots,\tilde{f}_{m,x})}(\Gamma g_1),\Theta^{(2)}_{(T_{1},\ldots,T_{m})}(\Gamma g_2,\Gamma g_1)\right)\label{def-F_x-in-limit-theorem}
\end{align}
where $\tilde{f}_{\ell,x}(w)=\e{\tha\beta(\alpha+x\beta)}f_\ell(w)$. Since \eqref{function-tilde_f_tau_x}-\eqref{function-tilde_T_tau} imply that $\tilde{f}_{\ell,x;\tau}\to\tilde{f}_{\ell,x}$ and $\tilde{T}_{\ell;\tau}\to T_\ell$ as $\tau\ti$ for each $1\leq \ell\leq m$, we obtain that $F_\tau\to F$ on compacta as $\tau\ti$.
We now use \eqref{Theta_f_with_transformed_lift-1}-\eqref{Theta_f_with_transformed_lift-2} and \eqref{def-F_tau_x-in-limit-theorem} to  rewrite the limit in \eqref{statement-joint-equidistribution-2m-regular-functions-1} as
\begin{align}
    \lim_{\tau\to\infty}\int_{\R}F_{\tau}\!\left(x,\Gamma(I;\sve{\alpha}{-\beta},0)\Psi^x\Phi^\tau,\Gamma(I;\sve{-\alpha}{-\beta},0)\Psi^{-x}\Phi^{\tau}\right)\de\lambda(x).
\end{align}
We apply Theorem \ref{thm-joint-equidistribution-family} with $\vecxi=(\alpha,-\beta)\notin\Q^2$ and obtain that this limit equals
\begin{align}
    \int_{\R\times\GamG\times\GamG}F(x,\Gamma g_1,\Gamma g_2)\,\de\lambda(x)\,\de\hat\mu_{\GamG\times\GamG}(g_1,g_2).\label{integral-after-applying-limit-theorem-before-removing-x-dependence}
\end{align}
Note that the pre-factor $\e{\ha\beta(\alpha+x\beta)}$ in the definition of $\tilde{f}_{\ell,x}$ is the same for each $1\leq\ell\leq m$.
Recall that the measure $\hat\mu_{\GamG\times\GamG}$ is defined in \eqref{def-hat-mu-2} and projects onto the Haar measure $\mu_{\GamG}$ via $(\Gamma g_1,\Gamma g_2)\mapsto\Gamma g_1$. The invariance of $\mu_{\GamG}$ under multiplication by $(I;\sve{0}{0},-\ha\beta(\alpha+x\beta))$ allows us to remove the $x$-dependence from each $\Theta_{\tilde{f}_{\ell,x}}$ without introducing any further $x$-dependence any of the $\Theta_{\tilde{T}_\ell}$ (due to the fact that the component of $g_1$ in the $\zeta$-direction is opposite to that of $g_2$). Therefore \eqref{integral-after-applying-limit-theorem-before-removing-x-dependence} equals
\begin{align}
    \int_{\GamG\times\GamG}B\!\left(\Theta_{\underline{f}}(\Gamma g_1),\Theta^{(2)}_{\underline{T}}(\Gamma g_2,\Gamma g_1) \right)\de\hat\mu_{\GamG\times\GamG}\label{integral-after-applying-limit-theorem-after-removing-x-dependence}
\end{align}
To obtain \eqref{statement-joint-equidistribution-2m-regular-functions-1}, we simply push forward the measure $\hat\mu_{\GamG\times\GamG}$  via the function $H_{\underline{f},\underline{T}}:\GamG\times\GamG\to\C^{2m}$, 
\begin{align*}
    H_{\underline{f},\underline{T}}(\Gamma g_1,\Gamma g_2)=\left(\Theta_{\underline{f}}(\Gamma g_1),\Theta^{(2)}_{\underline{T}}(\Gamma g_2,\Gamma g_1) \right).
\end{align*}
That is, \eqref{integral-after-applying-limit-theorem-after-removing-x-dependence} equals the right-hand-side of \eqref{statement-joint-equidistribution-2m-regular-functions-1} with $\mathbb{P}^{\mathrm{reg}}_{\underline{f},\underline{T}}=(H_{\underline{f},\underline{T}})_{*}\hat\mu_{\GamG\times\GamG}$.
\end{proof}
We are now able to prove the second ingredient necessary for proving Theorem \ref{theorem:main}.
\begin{theorem}[Existence of finite dimensional limiting distributions of $\mathbf{X}_N$]\label{thm-existence-of-fdd-for-mathbfX}
Fix $k\geq1$ and fix $\underline{s}=(s_1,\ldots, s_k)$, $\underline{t}=(t_1, \ldots, t_k)\in[0,T]^k$ with $s_\ell<t_\ell$ for each $1\leq \ell\leq k$. Let $(\alpha,\beta)\notin\Q^2$. Let $\lambda$ be a Borel probability  measure on $\R$ which is absolutely continuous with respect to the Lebesgue measure. Let $x$ be randomly distributed according to $\lambda$. Then the $\R^2\times\R^{2\times2}$-valued random variables
\begin{align}
    \mathbf{X}_N(x;\alpha,\beta;s_1,t_1), \ldots, \mathbf{X}_N(x;\alpha,\beta;s_k,t_k)
\end{align}
have a joint limiting distribution as $N\ti$. More precisely, there exists a probability measure $\mathbb{P}_{s_1,t_1;\ldots;s_k,t_k}$ on $(\R^2\times\R^{2\times2})^k$ such that for every bounded continuous function $B:(\R^2\times\R^{2\times2})^k\to\R$, we have 
\begin{align}
    \lim_{N\to\infty}\int_{\R} B&\!\left(\mathbf{X}_N(x;\alpha,\beta; s_1,t_1),
    \ldots, \mathbf{X}_N(x;\alpha,\beta;s_k,t_k)\right)\de\lambda(x)\\
    &=\int_{(\R^2\times\R^{2\times 2})^k}B\,\de\mathbb{P}_{s_1,t_1;\ldots;s_k,t_k}.
\end{align}
The limiting law $\mathbb{P}_{s_1,t_1;\ldots;s_k,t_k}$ does not depend on $\alpha,\beta$ and on $\lambda$.
\end{theorem}
As we will see, Theorem \ref{thm-existence-of-fdd-for-mathbfX} follows from Theorem \ref{thm-limiting-distribution-of-finitely-regular-thetas} by means of an approximation argument which uses several lemmata. The strategy is the same as for the proof of Theorem 4.4 in \cite{Cellarosi-Marklof}.

\begin{lemma}\label{lemma:make-sharp-functions-smooth-epsilon}
Let $\lambda$ be a Borel probability measure on $\mathbb R$ absolutely continuous with respect to Lebesgue measure. Let $\varepsilon>0$. Then there exists $F_1,...,F_6\in \mathcal S(\mathbb R^2)$ so that for any $(g_1,g_2)\in  G\times  G$ with $-x_1=x_2=x$, we have 
 \begin{align*}
    \lambda\left(\left\{x\in \mathbb R: \left| \Theta_{\mathfrak C_{0,0}-F_1}^{(2)}(g_1,g_2)\right|>R\right\}\right)&< \frac{\varepsilon}{(1+R)^2}\\
    \lambda\left(\left\{x\in \mathbb R: \left| \Theta_{\mathfrak C_{1,1}-F_2}^{(2)}(g_1,g_2)\right|>R\right\}\right)&< \frac{\varepsilon}{(1+R)^2}\\
 \lambda\left(\left\{x\in \mathbb R: \left| \Theta_{\mathfrak C_{0,1}-F_3}^{(2)}(g_1,g_2)\right|>R\right\}\right)&<\frac{\varepsilon}{(1+R)^2}\\
  \lambda\left(\left\{x\in \mathbb R: \left| \Theta_{\mathfrak{L}_{h}-F_4}^{(2)}(g_1,g_2)\right|>R\right\}\right)&<
  \frac{\varepsilon}{(1+R)^2}\\
    \lambda\left(\left\{x\in \mathbb R: \left| \Theta_{\mathfrak L_v-F_5}^{(2)}(g_1,g_2)\right|>R\right\}\right)&<
    \frac{\varepsilon}{(1+R)^2}\\
 \lambda\left(\left\{x\in \mathbb R: \left| \Theta_{\mathfrak L_d-F_6}^{(2)}(g_1,g_2)\right|>R\right\}\right)&< \frac{\varepsilon}{(1+R)^2}
  \end{align*}
uniformly in $y_1,y_2\leq 1$. 
\end{lemma}
\begin{proof}
    For $F = \mathfrak{C}_{0,0}$ recall from equation \eqref{eq:F-decompose-as-sum-of-fg} that we can write
    $$ F(w_1,w_2)=\sum_{j,k=0}^\infty f(2^j w_1)f(2^k(w_2-w_1))$$
    for some smooth $f$. Also recall for this $F$ from equation \eqref{eq:decomposition-for-theta-2-F} that we can write
    $$ \Theta_{F}^{(2)}(g_1,g_2)=\sum_{j=0}^J\sum_{k=0}^K 2^{-j/2}2^{-k/2}\tilde \Theta_{(w_1,w_2)\mapsto f(w_1)f(w_2)}^{(2)}\left(\left(h(g_1,g_2),g(g_1,g_2)\right) (\operatorname{id},P_{j,k})\right),$$
    where $J=\lceil \log_2(y_1^{-1/2})\rceil,$ and $K=\lceil \log_2(y_2^{-1/2})\rceil$. Define $$F_N(w_1,w_2):=\sum_{j,k=0}^{N-1} f(2^j w_1)f(2^k(w_2-w_1)).$$
    In this case we have that
    \begin{align*}
        \Theta_{F-F_N}^{(2)}(g_1,g_2)&=\sum_{j=N}^J\sum_{k=N}^K 2^{-\frac{j}{2}}2^{-\frac{k}{2}}\tilde \Theta_{(w_1,w_2)\mapsto f(w_1)f(w_2)}^{(2)}\left(\left(h(g_1,g_2),g(g_1,g_2)\right)(\operatorname{id},P_{j,k})\right)\\
        &~+\sum_{j=0}^J\sum_{k=N}^K 2^{-\frac{j}{2}}2^{-\frac{k}{2}}\tilde \Theta_{(w_1,w_2)\mapsto f(w_1)f(w_2)}^{(2)}\left(\left(h(g_1,g_2),g(g_1,g_2)\right)(\operatorname{id},P_{j,k})\right)\\
        &~+\sum_{j=N}^J\sum_{k=0}^K 2^{-\frac{j}{2}}2^{-\frac{k}{2}}\tilde \Theta_{(w_1,w_2)\mapsto f(w_1)f(w_2)}^{(2)}\left(\left(h(g_1,g_2),g(g_1,g_2)\right)(\operatorname{id},P_{j,k})\right).
    \end{align*}
    Without loss of generality, we may assume that $J,K>N$. Indeed, if $N<K$ or $N<J$ then one or more of these sums will be $0$. We can apply a union bound to get that 
    \begin{align*}
        \lambda&\left(\left\{x\in \mathbb R:|        \Theta_{F-F_N}^{(2)}(g_1,g_2)|>R \right\}\right)\\
        &~~\leq \lambda\left(\left\{x\in \mathbb R:\left|\sum_{j=N}^J\sum_{k=N}^K 2^{-\frac{j}{2}}2^{-\frac{k}{2}}\tilde \Theta_{(w_1,w_2)\mapsto f(w_1)f(w_2)}^{(2)}\left(\left(h(g_1,g_2),g(g_1,g_2)\right)(\operatorname{id},P_{j,k})\right)\right|>\frac{R}{3}  \right\}\right)\\
        &~~+ \lambda\left(\left\{x\in \mathbb R:\left|\sum_{j=0}^J\sum_{k=N}^K 2^{-\frac{j}{2}}2^{-\frac{k}{2}}\tilde \Theta_{(w_1,w_2)\mapsto f(w_1)f(w_2)}^{(2)}\left(\left(h(g_1,g_2),g(g_1,g_2)\right)(\operatorname{id},P_{j,k})\right)\right|>\frac{R}{3} \right\}\right)\\
        &~~+ \lambda\left(\left\{x\in \mathbb R:\left|\sum_{j=N}^J\sum_{k=0}^K 2^{-\frac{j}{2}}2^{-\frac{k}{2}}\tilde \Theta_{(w_1,w_2)\mapsto f(w_1)f(w_2)}^{(2)}\left(\left(h(g_1,g_2),g(g_1,g_2)\right)(\operatorname{id},P_{j,k})\right)\right|>\frac{R}{3} \right\}\right).
    \end{align*}
    In this case, the argument used in inequalities \eqref{eq:argument-in-sheared-tail-bound}-\eqref{eq:argument-in-sheared-tail-bound-end} and the standard tail bounds for a geometric series gives that
    \begin{equation}
        \lambda\left(\left\{x\in \mathbb R:|        \Theta_{F-F_N}^{(2)}(g_1,g_2)|>R \right\}\right)\ll \frac{2^{-N}2^{-N}}{(1+R)^2}+\frac{2^{-N}}{(1+R)^2}+\frac{2^{-N}}{(1+R)^2}.
    \end{equation}
    Choosing $N$ large enough proves the claim for $F= \mathfrak{C}_{0,0}$. The proof for the rest of the building blocks follows similarly. 
\end{proof}

\begin{lemma}\label{lemma:approximation-of-triangle-smooth-fdds}
Let $\lambda$ be a Borel probability measure on $\mathbb R$ absolutely continuous with respect to Lebesgue measure.  Let for any $(g_1,g_2)\in  G\times  G$ with $-x_1=x_2=x$. For all $\varepsilon>0$, $\delta>0$ there exists $F\in \mathcal S(\mathbb R^2)$ so that 
    \begin{equation}\label{statement-Lemma-replace-T_(0,1]-with-smooth-F-for-small-price}
        \lambda\left(\left\{x\in \mathbb R: |\Theta_{T_{(0,1]}}^{(2)}(g_1,g_2)-\Theta_{F}^{(2)}(g_1,g_2)|>\delta\right\}\right)< \varepsilon,
    \end{equation}
    uniformly in $y_1,y_2\leq 1.$
\end{lemma}
\begin{proof}
Thanks to the second part of Lemma \ref{lemma:lines-dont-matter}, it is enough to show \eqref{statement-Lemma-replace-T_(0,1]-with-smooth-F-for-small-price} with $T_{(0,1]}$ replaced by $T_{(0,1)}$. Let $\varepsilon'<\frac{\varepsilon}{6(1+R)^2}$. Use Lemma \ref{lemma:make-sharp-functions-smooth-epsilon} with $\varepsilon'$ to get functions $F_1,...,F_6$. Define $F=F_1+...+F_6+\mathfrak{F}_{\text{smooth}}$ where $\mathfrak{F}_{\text{smooth}}$ is coming from  \eqref{eq:decomposition-of-triangle}. A union bound and Lemma \ref{lemma:make-sharp-functions-smooth-epsilon} concludes.
\end{proof}

\begin{proof}[Proof of Theorem \ref{thm-existence-of-fdd-for-mathbfX}]
Recall from equations \eqref{def-L_N} and \eqref{rewriting_X_N^11}-\eqref{rewriting_X_N^22} that we can write $\mathbf X_N(s,t):=(X(t)-X(s),\mathbb X_N(s,t))$ as a continuous function of $\Theta_{\chi_{(s,t]}}$ and of $\Theta^{(2)}_{T_{(s,t]}}$ plus some terms that go to $0$ in probability. By Slutsky's theorem and the continuous mapping theorem, convergence of the joint finite dimensional distributions of $\Theta_{\chi_{(s,t]}}$ and of $\Theta^{(2)}_{T_{(s,t]}}$ imply the convergence of the finite dimensional distributions of $\mathbf X_N$. 
Using the argument in the proof of Theorem 4.4 in \cite{Cellarosi-Marklof} with Theorem \ref{thm-limiting-distribution-of-finitely-regular-thetas} taking the role of Lemma 4.5 in \cite{Cellarosi-Marklof} and with the Lemma \ref{lemma:approximation-of-triangle-smooth-fdds} taking the role of Lemma 4.9 in \cite{Cellarosi-Marklof} concludes.
\end{proof}

\begin{proof}[Proof of the Main Theorem \ref{theorem:main}]
Theorem \ref{theorem:tightness-second-order} and Theorem \ref{thm-existence-of-fdd-for-mathbfX} prove the main Theorem \ref{theorem:main}, as outlined in Section \ref{section:strategy}.
\end{proof}

\section{L\'evy Area and Closing Remarks}
The spiral-like structures (``curlicues") in the sequence of partial sums of theta sums have been investigated extensively, see, e.g. \cite{Berry-Goldberg, Cellarosi-Curlicue, Cellarosi-Marklof, Coutsias-Kazarinoff1987, Moore-vanderPoorten, Sinai-curlicues}. Presence of curlicues at various scales can be understood dynamically via renormalization techniques that rely on approximate functional equations for theta sums, see \cite{Coutsias-Kazarinoff1998, Fedotov-Kopp2012, HL-1914, Mordell, Wilton}.
The well defined (in light of Theorem \ref{theorem:main}) L\'evy area $$L(t):=\int_0^t X^1(r)dX^2(r)-\int_0^t X^2(r) dX^1(r)$$ can be seen (in light of Green-Stokes theorem) as twice the signed area swept out by $(X^1,X^2)$ with respect to the line connecting $X(0)$ and $X(t)$. Curlicues happen when $L$ is highly oscillatory but does not grow much in magnitude. See Figure \ref{fig:levyarea}.

\begin{figure}[h!]
    \centering
    \includegraphics[width=14cm]{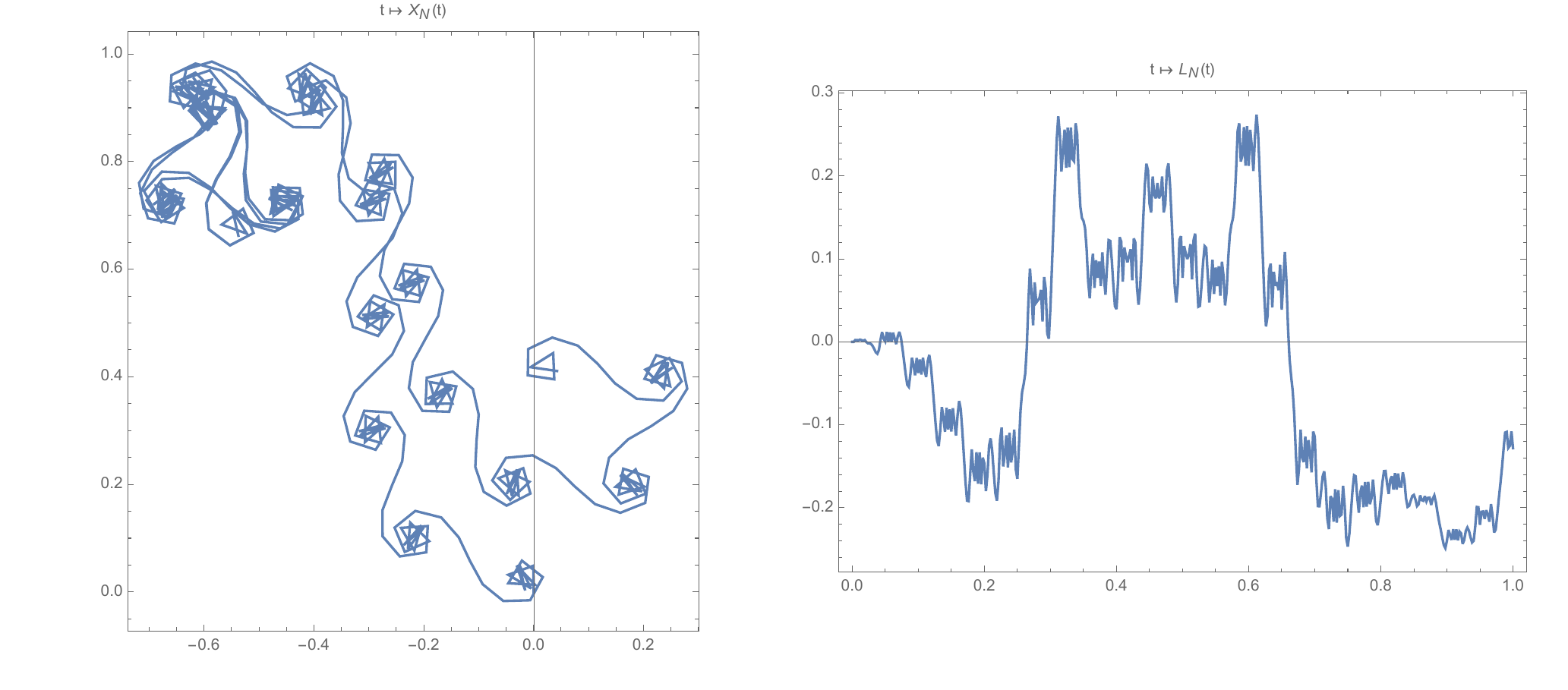}
    \includegraphics[width=14cm]{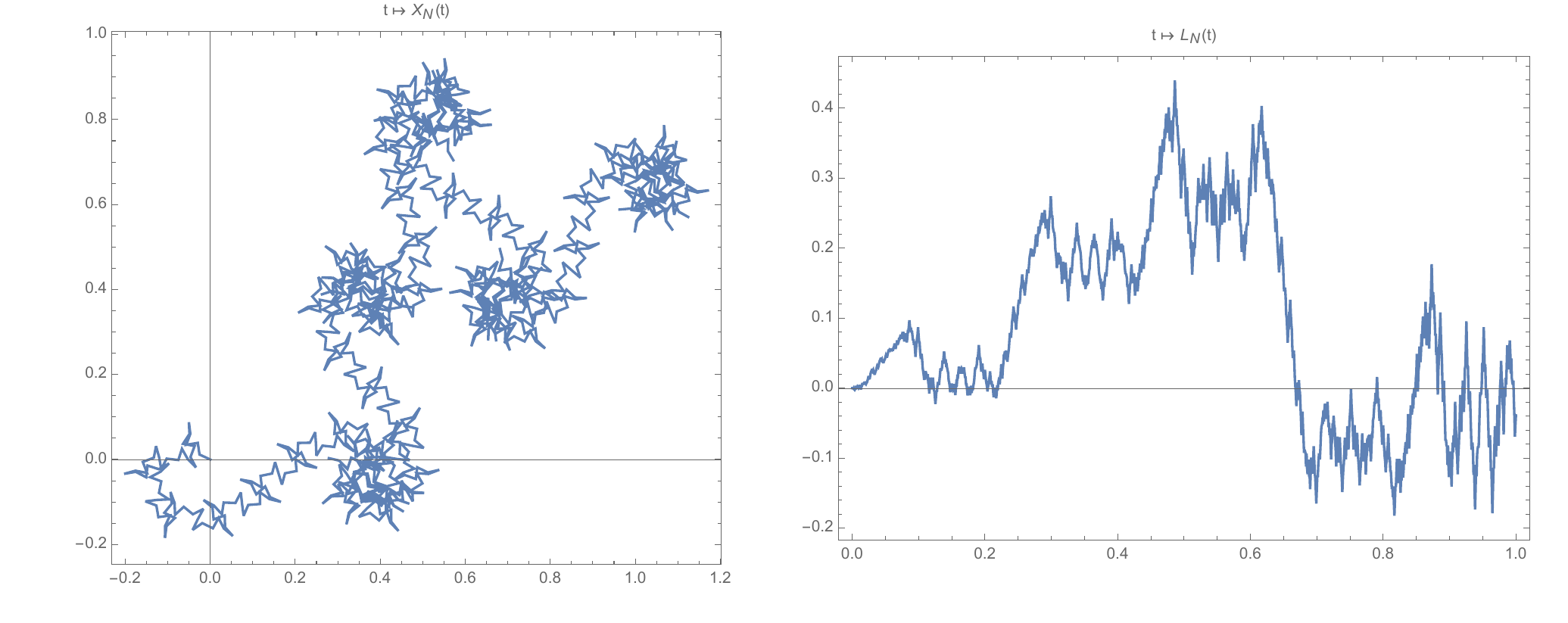}
    \caption{Left panels: two realizations of the random curve $[0,1]\ni t\mapsto X_{N}(t)$, where $\alpha=\sqrt{2}$, $\beta=0$. Right panels: the corresponding L\'evy area $[0,1]\ni t\mapsto L_N(t)$.
    In the top panels we have $x\approx 0.0528282$ and $N=420$. In the bottom panels we have $x\approx0.411534$ and $N=1,\!300$. }
    \label{fig:levyarea}
\end{figure}

The L\'evy area of $X_N$ at time $1$ by a standard trigonometric identity is 
\begin{align*}
    L_N(1)&:=\frac{1}{N}\sum_{1\leq n_1<n_2\leq N} \cos\left(2\pi\left(\frac{1}{2}n_1^2+\beta n_1\right)x+2\pi n_1\alpha\right)\sin\left(2\pi\left(\frac{1}{2}n_2^2+\beta n_2\right)x+2\pi n_2\alpha\right)\\
    &~- \frac{1}{N}\sum_{1\leq n_1<n_2\leq N}\sin\left(2\pi\left(\frac{1}{2}n_1^2+\beta n_1\right)x+2\pi n_1\alpha\right)\cos\left(2\pi\left(\frac{1}{2}n_2^2+\beta n_2\right)x+ 2\pi n_2\alpha\right)\\
    &=-\frac{1}{2N} \sum_{1\leq n_1<n_2\leq N}\sin\left(2\pi \left(\frac{1}{2}(n_1^2-n_2^2)+\beta (n_1-n_2)\right)x+2\pi(n_1-n_2)\alpha\right)\\
    &=\frac{1}{2}\Im\left(\Theta^{(2)}_{T_{(0,1]}}\left(\left(-x+\tfrac{i}{N^2},0;\sve{-\alpha-\beta x}{0},0\right),\left(x+\tfrac{i}{N^2},0;\sve{\alpha+\beta x}{0},0\right)\right)\right).
\end{align*}
Using the evenness of cosine and applying another trigonometric identity to the real part of $\Theta^{(2)}_{T_{(0,1]}}$ shows that
\begin{align*}
    \Re&\bigg(\Theta^{(2)}_{T_{(0,1]}}\left(\left(-x+\tfrac{i}{N^2},0;\sve{-\alpha-\beta x}{0},0\right),\left(x+\tfrac{i}{N^2},0;\sve{\alpha+\beta x}{0},0\right)\right)\bigg)\\
    &=\frac{1}{N} \sum_{1\leq n_1<n_2\leq N}\cos\left(2\pi \left(\frac{1}{2}(n_1^2-n_2^2)+\beta (n_1-n_2)\right)x+2\pi (n_1-n_2)\alpha\right)\\
    &=\left(\frac{1}{2N}\sum_{1\leq n_1,n_2\leq N}\cos\left(2\pi\left(\frac{1}{2}(n_1^2-n_2^2)+\beta (n_1-n_2)\right)x+2\pi(n_1-n_2)\alpha\right)\right)-\frac{1}{2}\\
    &=\left(\frac{1}{\sqrt N} \sum_{1\leq n\leq N}\cos\left(2\pi\left(\frac{1}{2}n^2+\beta n\right)x+2\pi n\alpha\right)\right)^2\\
    &\:\:\:\:\:\:+\left(\frac{1}{\sqrt N} \sum_{1\leq n\leq N}\sin\left(2\pi\left(\frac{1}{2}n^2+\beta n\right)x+2\pi n\alpha\right)\right)^2-\frac{1}{2}\\
    &=\|X_N(1)\|^2-\frac{1}{2}.
\end{align*}
Therefore, it is only the imaginary part of $\Theta^{(2)}_{T_{(0,1]}}$ that gives any extra information, as the real part can be written in terms of the lower ordered process. Figure \ref{fig:two-histograms} compares the empirical distributions of the imaginary and the real parts of $\Theta^{(2)}_{T_{(0,1]}}\left(\left(-x+\tfrac{i}{N^2},0;\sve{-\alpha-\beta x}{0},0\right),\left(x+\tfrac{i}{N^2},0;\sve{\alpha+\beta x}{0},0\right)\right)$.
\begin{figure}[b]
    \centering
    \includegraphics[width=16cm]{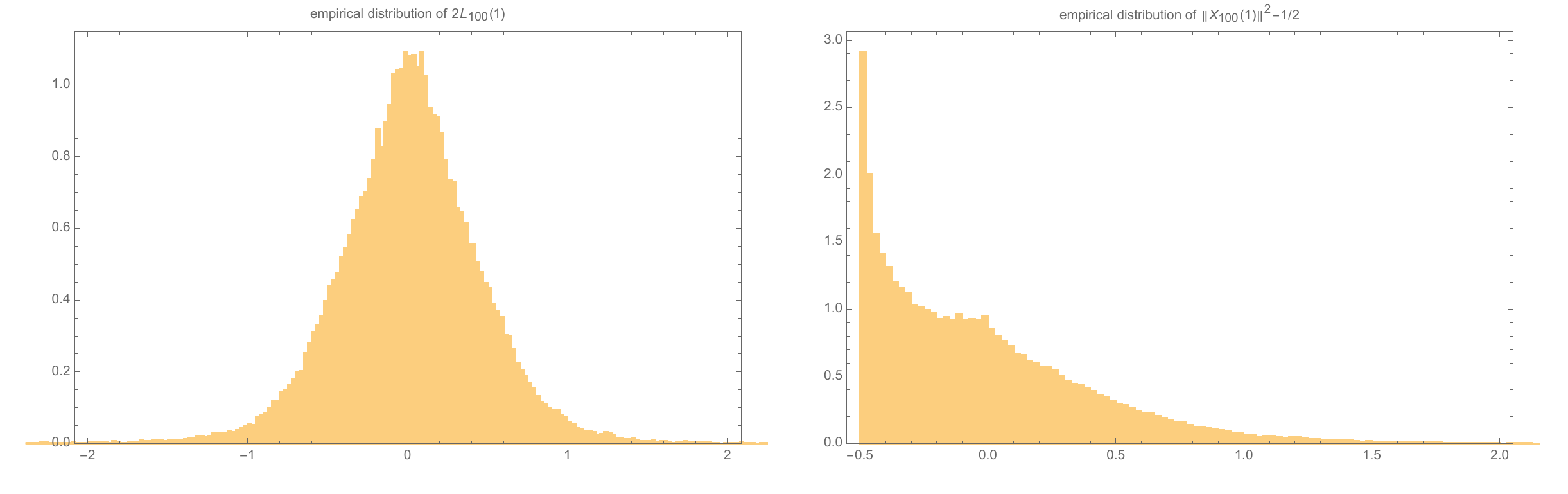}
    \caption{The empirical distribution of the imaginary part (left panel) and the real part (right panel) of $\Theta^{(2)}_{T_{(0,1]}}\left(\left(-x+\tfrac{i}{N^2},0;\sve{-\alpha-\beta x}{0},0\right),\left(x+\tfrac{i}{N^2},0;\sve{\alpha+\beta x}{0},0\right)\right)$, where $N=100$, $\alpha=\sqrt{2}$, $\beta=0$, and $x$ randomly uniformly distributed on $[0,1]$ (sample size is 300,000). The imaginary part equals twice the L\'{e}vy area $2L_N(1)$, while the real part equals $\|X_N\|^2-\frac{1}{2}$.}
    \label{fig:two-histograms}
\end{figure}

\begin{remark}
    The primary technical barrier in the construction of the rough path is the dyadic construction of $T_{(0,1)}$ given in Section \ref{section:defining-theta-T} and in the bounds for the ``sheared" building blocks in Subsections \ref{subsection:theta-theta-tilde} and \ref{subsection:estimates-for-sheared-blocks}. In \cite{Marklof-Welsh-I,Marklof-Welsh-II} the authors consider higher ranked theta sums over boxes. However the relationship between these two sums is unclear. We must control sums of the form
    $$\sum_{1\leq n_1<n_2\leq N} \sin(a_{n_1})\cos(a_{n_2})$$
    for a sequence of real numbers $a_n$. If we let $a_n=0$ for $n<N$ then
    $$\sum_{1\leq n_1<n_2\leq N} \sin(a_{n_1})\cos(a_{n_2})=0,$$
    but 
    $$\sum_{n_1,n_2=1}^N \sin(a_{n_1})\cos(a_{n_2})=\sin(a_N)(N-1+\cos(a_N)),$$
    which can take values between $-(N-1)$ and $N-1$. Therefore, the sum over a square in general tells us very little about the sum over a triangle.
\end{remark}

\begin{remark}
    The authors in \cite{Cellarosi-Marklof} show in Lemma 4.6 for compactly supported Riemann integrable $f$ that $$\lim_{y\to 0^+} E_\lambda [|\Theta_f(g)|^2]\leq 2 \|f\|_{L^2}.$$
    To do this, the authors had to count solutions to the diophantine equation 
    $$n_1^2-n_2^2+c(n_1-n_2)=0.$$
    When trying to do the same argument for $\Theta^{(2)}_f$ one arrives at the diophantine equation
    $$n_1^2-n_2^2+n_3^2-n_4^2+c(n_1-n_2+n_3-n_4)=0,$$
    which is significantly harder to count solutions to. In the case $c=0$ for example, to get the kind of decay necessary for the proof of Lemma 4.6 in \cite{Cellarosi-Marklof} one has to count solutions to 
    $$n_1^2-n_2^2+n_3^2-n_4^2=k\in \mathbb Z.$$
    Counting solutions of this form in a way that would let the argument of Lemma 4.6 in \cite{Cellarosi-Marklof} work seems quite intractable. The lack of a nice abstract approximation result is a one reason why constructing the triangle $T_{(0,1)}$ by hand dyadically was necessary. 
\end{remark}
\begin{remark}
    The dyadic building blocks $T_{\operatorname{Cor}}$ and $T_{\operatorname{Line}}$ given in Section \ref{section:defining-theta-T} can be combined in more general ways to construct other polygons. In addition, due to the generality of the parameters $c_1,c_2,c_3$ one can choose the support of corner and line building blocks to construct smaller polygons or polygons with smaller angles. \cite{Marklof-Welsh-I,Marklof-Welsh-II} considers boxes of any dimension however to the authors' best knowledge, no one has studied theta sums over sharp indicators of general polygons. 
\end{remark}

\bibliographystyle{plain}
\bibliography{bibliography}

\appendix 
\section{Semimartingale with Scaling and Stationarity implies Local Martingale or Bounded Variation}
We show the following lemma which might be of independent interest. 
\begin{lemma}\label{lemma:semimartingale-is-BM}
    Let $X(t)$ be a real valued continuous $L^4$ semimartingale so that $X(t+h)-X(t)=X(h)$ in distribution and so that $t\mapsto X(at)$ and $t\mapsto a^\gamma X(t)$ are equal in distribution for all $a>0$ and for some $\gamma>0$. Then either $X(t)$ is a scaled Brownian motion or it has no martingale part and is of bounded variation.
\end{lemma}

Let $X$ be as in Lemma \ref{lemma:semimartingale-is-BM} Without loss of generality, we assume $X(0)=0$ and $E[X^2(1)]=1$. We can scale and shift otherwise. From now on, we assume that $X$ has a nontrivial martingale part.

\begin{lemma}
   We have that $\gamma=1/2$.
\end{lemma}
\begin{proof}
Using stationarity of increments and self similarity gives
    \begin{align*}
        E[(X(t)-X(s))^2]&=E[X^2(t-s)]\\
        &=(t-s)^{2\gamma}E[X^2(1)]\\
        &=(t-s)^{2\gamma}.
    \end{align*}
    Take a sequence of partitions $\mathcal P_n$ of $[0,t]$ with $|\mathcal P_n|\to 0$. Then
    \begin{align*}
        E\left(\sum_{k=0}^n (X(t_{k+1})-X(t_k))^2\right)&=\sum_{k=0}^n (t_{k+1}-t_k)^{2\gamma},
    \end{align*}
    which blows up if $\gamma<1/2$ or is $0$ if $\gamma>1/2$. The quadratic variation of a semimartingale with nontrivial martingale part is positive and finite. Therefore $\gamma=1/2$.
\end{proof}
\begin{lemma}
   We have that $E(X(t))=0$.
\end{lemma}
\begin{proof}
    By stationarity, we have that $X(2)-X(1)=X(1)$ in distribution. Therefore $E(X(2))-E(X(1))=E(X(1))$, so by self similarity $2^{1/2}E(X(1))=2 E(X(1))$, therefore $E(X(1))=0$. Self similarity gives that $E(X(t))=0$ for all t. 
\end{proof}
\begin{lemma}
The covariance of $X$ is $E(X(t)X(s))=\min(s,t)$. 
\end{lemma}
\begin{proof}
   The polarization identity yields
   \begin{align*}
      E(X(t)X(s))&=\frac{1}{2}E\left(X^2(t)+X^2(s)-(X(t)-X(s))^2\right) \\
      &=\frac{1}{2}(t+s-|t-s|)\\
      &=\min(s,t).
   \end{align*}
\end{proof}

\begin{lemma}\label{lemma:uncorrelated}
    $X$ has uncorrelated increments.
\end{lemma}
\begin{proof}
    Let $0\leq s_1\leq t_1\leq s_2\leq t_2$. Then
    \begin{align*}
        E((X(t_2)-X(s_2))(X(t_1)-X(s_1)))
        &=E(X(t_2)X(t_1)-E(X(t_2)X(s_1))\\
        &-E(X(s_2)X(t_1))+E(X(s_2)X(s_1))\\
        &=t_1-s_1-t_1-s_1\\
        &=0.
    \end{align*}
\end{proof}
\begin{lemma}\label{lemma:nonpositively-correlated}
The squares of increments are nonpositively correlated.
\end{lemma}
\begin{proof}
Let $0\leq s_1\leq t_1\leq s_2\leq t_2$. Then by Cauchy-Schwarz we have that 
    \begin{align*}
        E((X(t_2)-X(s_2))^2&(X(t_1)-X(s_1))^2)\\
        &\leq \sqrt{E((X(t_2)-X(s_2))^4)} \sqrt{E((X(t_1)-X(s_1))^4}\\
        &=(t_2-s_2)(t_1-s_1)\\
        & =E((X(t_2)-X(s_2))^2)E((X(t_1)-X(s_1))^2).
    \end{align*}
\end{proof}
\begin{lemma}\label{lemma:4th-moment}
There exists some $C\geq 1$ so that
    \begin{equation}
        E((X(t)-X(s))^4)= C (t-s)^2.
    \end{equation}
\end{lemma}
\begin{proof}
    By stationarity and self similarity, we have that
    \begin{align*}
        E((X(t)-X(s))^4)&=E(X^4(t-s))\\
        &=(t-s)^2 E[X^4(1)]\\
        &=C(t-s)^2,
    \end{align*}
    where $\sqrt[4]{C}=\sqrt[4]{E[X^4(1)]}\geq \sqrt{E[X^2(1)]}=1$.
\end{proof}
\begin{proposition}\label{prop:covariation}
    The quadratic variation of $X$ on $[0,t]$ is 
    \begin{equation}
    [X](t)=t.
    \end{equation}
\end{proposition}
\begin{proof}
     Take a sequence of partitions $\mathcal P_n$ of $[0,t]$ with $|\mathcal P_n|\to 0$. Then
    \begin{align*}
        E\left(\sum_{k=0}^n E[(X(t_{k+1})-X(t_k))^2]\right)&=\sum_{k=0}^n (t_{k+1}-t_k)=t.
    \end{align*}
    Then we need to show that 
    \begin{equation*}
         E\left(\left(\sum_{k=0}^n E[(X(t_{k+1})-X(t_k))^2]-t\right)^2\right)\to 0.
    \end{equation*}
    We can recognize this as the variance of the variational sum. Therefore, 
    \begin{align*}
        E\left(\left(\sum_{k=0}^n E[(X(t_{k+1})-X(t_k))^2]-t\right)^2\right)&=\operatorname{var}\left(\sum_{k=0}^n E[(X(t_{k+1})-X(t_k))^2]\right)\\
        &\leq \sum_{k=0}^n \operatorname{var} \left(E[(X(t_{k+1})-X(t_k))^2]\right),
    \end{align*}
    where we used Lemma \ref{lemma:nonpositively-correlated}. We now just need to study $\operatorname{var} \left(E[(X(t_{k+1})-X(t_k))^2]\right)$. From Lemma \ref{lemma:4th-moment}, we have that 
    \begin{equation*}
        \operatorname{var} \left(E[(X(t_{k+1})-X(t_k))^2]\right)=C(t_{k+1}-t_k)^2-(t_{k+1}-t_k)^2=(C-1)(t_{k+1}-t_k)^2.
    \end{equation*}
    Therefore
    \begin{align*}
        E\left(\left(\sum_{k=0}^n E[(X(t_{k+1})-X(t_k))^2]-t\right)^2\right)&\leq (C-1)\sum_{k=0}^n (t_{k+1}-t_k)^2\\
        &\leq (C-1) |\mathcal P_n| t\\
        &\to 0.
    \end{align*}
\end{proof}
\begin{corollary}
    $X(t)=B(t)+V(t)$ for a standard Brownian motion $B$ and adapted bounded variation process $V$.
\end{corollary}
\begin{proof}
Because $X$ is a semimartingale that is continuous, it is a continuous semimartingale. That is, $X=M+V$ for continuous local martingale $M$ and continuous bounded variation function $V$. See \cite{Rogers-Williams}, page 358. Therefore we know that the quadratic variation of $X$ is the quadratic variation of $M$. 
By L\`evy characterization theorem we have that $M=B$.
\end{proof}
\begin{proposition}
There exists some constant $C'$ so that for all $t\geq s$ we have that
\begin{equation}
    E(V(t)V(s))=C'\min(s,t).
\end{equation}
In particular, we have that
    $$E((V(t)-V(s))^2)=C'(t-s).$$
\end{proposition}
\begin{proof}
 From Lemma \ref{prop:covariation}  we know for $t\geq s$ that
\begin{align*}
    s&=E(X(t)X(s))\\
    &=E\left(B(t)B(s)\right)+E\left(B(t)V(s)\right)+E\left(B(s)V(t)\right)+E(V(t)V(s))\\
    &=s+E\left(B(t)V(s)\right)+E\left(B(s)V(t)\right)+E(V(t)V(s))\\
    &=s+E\left(B(t)V(s)\right)+E(V(t)X(s))\\
    &=s+E\left(E\left(B(t)V(s)|\mathcal F_s\right)\right)+E(V(t)X(s))\\
    &=s+E\left( V(s)E\left(B(t)|\mathcal F_s\right)\right)+E(V(t)X(s))\\
    &=s+E\left(V(s)B(s)\right)+E(V(t)X(s)),
\end{align*}
using the tower property of conditional expectation. This implies that 
$$E(V(t)X(s))$$
is independent of $t$. 
 Therefore
$$E(V(t)X(s))=E(V(s)X(s))=sE(V(1)X(1))=:Cs=C\min(s,t)$$
by self similarity. Investigating the correlation of increments gives that
\begin{align*}
    E((X(t_1)-X(s_1))(V(t_2)-V(s_2))
    &=C(\min(t_1,t_2)-\min(t_1,s_2)\\
    &~-\min(s_1,t_2)+\min(s_1,s_2))\\
    &=0
\end{align*}
if either $0\leq s_1\leq t_1\leq s_2\leq t_2$ or $0\leq s_2\leq t_2\leq s_1\leq t_1$. As we know that increments of $X$ are uncorrelated by Lemma \ref{lemma:uncorrelated}, we have that
\begin{align*}
    0&=E((X(t_1)-X(s_1))(X(t_2)-X(s_2))\\
    &=E\left(\left(B(t_1)-B(s_1)+V(t_1)-V(s_1)\right)(X(t_2)-X(s_2))\right)\\
    &=E\left(\left(B(t_1)-B(s_1)\right)(X(t_2)-X(s_2))\right)\\
    &=E\left(\left(B(t_1)-B(s_1)\right)\left(B(t_2)-B(s_2)+V(t_2)-V(s_2)\right)\right)\\
    &=E\left(\left(B(t_1)-B(s_1)\right)\left(V(t_2)-V(s_2)\right)\right).
\end{align*}
Therefore, $V$ also has uncorrelated increments. That is, for $0\leq s_1\leq t_1\leq s_2\leq t_2$ we have
\begin{equation}
    E((V(t_2)-V(s_2))(V(t_1)-V(s_1)))=0.
\end{equation}
We thus get for $t\geq s$ that 
$$E((V(t)-V(s))V(s))=0,$$
so $$E(V(t)V(s))=E(V^2(s))=sE(V^2(1))=:C'\min(s,t)$$
by self similarity. For $t\geq s$ we have
\begin{align*}
    E((V(t)-V(s))^2)&=E(V^2(t))-2E(V(t)V(s))+E(V^2(s))\\
    &=C(t-2s+s)\\
    &=C'(t-s). 
\end{align*}
\end{proof}
\begin{proposition}
    The quadratic variation of $V$ on $[0,t]$ is $[X](t)=C't$.
\end{proposition}
\begin{proof}
    The argument is the same as Proposition \ref{prop:covariation}. Note that as $B$ has infinitely many moments, $X$ has $4$ moments, $V$ also has $4$ moments. 
\end{proof}
\begin{proof}[Proof of Lemma \ref{lemma:semimartingale-is-BM}]
As $V$ is of bounded variation and continuous, its quadratic variation is $0$. Therefore we have that $C'=0$ and $V(t)$ is identically $0$. 
\end{proof}
\end{document}